\documentclass[a4paper]{amsart}

\usepackage[dvipsnames]{xcolor}
\usepackage[utf8]{inputenc}
\usepackage[english]{babel}
\usepackage{tikz-cd}
\usetikzlibrary{positioning,decorations.text,quotes}
\usetikzlibrary{arrows,babel}
\usepackage{amsmath,amsfonts,amssymb}
\usepackage{mathtools}
\usepackage{float}
\usepackage{amsthm}
\usepackage{stackrel}
\usepackage{pgfplots}
\usepackage{graphicx}
\usepackage{bbm}
\usepackage{amssymb}
\usepackage{geometry,enumerate}
\usepackage{graphicx}
\usepackage{psfrag}
\usepackage{amscd}
\usepackage{comment}
\usepackage[all]{xy}
\usepackage{url}
\usepackage{leftindex}
\usepackage{mathrsfs}

\usepackage{fullpage}

\usepackage[backref=page]{hyperref}

\hypersetup{
 colorlinks,
 citecolor=Green,
 linkcolor=blue,
 urlcolor=Blue}

\usepackage{booktabs}
\usepackage{color}
\usepackage{todonotes}

\usepackage{braket}

\usepackage{stmaryrd} 
\usepackage{mathabx}
\usepackage{mathbbol}
\usepackage{xspace}
\usepackage{enumerate}
\usepackage{caption}
\usepackage{subcaption}

 \usepackage{cleveref}
\crefformat{section}{#2#1#3} 
\crefformat{subsection}{#2#1#3}

\newcommand{\ZZ}{\mathbb{Z}}

\newcommand{\NN}{\mathbb{N}}
\newcommand{\RR}{\mathbb{R}}

\renewcommand{\AA}{\mathbb{A}}

\newcommand{\A}{\mathcal{A}}

\renewcommand{\O}{\mathcal{O}}

\renewcommand{\L}{\mathcal{L}}

\newcommand{\I}{\mathcal{I}}
\newcommand{\K}{\mathcal{K}}
\newcommand{\F}{\mathcal{F}}
\newcommand{\D}{\mathcal{D}}
\newcommand{\E}{\mathcal{E}}
\newcommand{\X}{\mathcal{X}}

\newcommand{\J}{\mathcal{J}}
\newcommand{\U}{\mathcal{U}}
\newcommand{\V}{\mathcal{V}}

\newcommand{\un}{\underline}
\newcommand{\ov}{\overline}
\newcommand{\wh}{\widehat}
\newcommand{\wt}{\widetilde}

\newcommand{\s}{\mathfrak{s}}
\renewcommand{\t}{\mathfrak{t}}

\newcommand{\Mbargn}{\overline{\mathcal{M}}_{g,n}}

\newcommand{\tf}{\operatorname{TF}}

\newcommand{\fX}{\mathfrak{X}}

\newcommand{\val}{\operatorname{val}}

\newcommand{\supp}{\operatorname{supp}}

\newcommand{\Proj}{\operatorname{Proj}}

\newcommand{\Hilb}{\operatorname{Hilb}}

\newcommand{\coker}{\operatorname{coker}}

\newcommand{\VStab}{\operatorname{VStab}}
\newcommand{\Pol}{\operatorname{Pol}}
\newcommand{\Deg}{\operatorname{Deg}}
\newcommand{\rk}{\operatorname{rk}}
\newcommand{\BCon}{\operatorname{BCon}}
\newcommand{\Con}{\operatorname{Con}}
\newcommand{\Sub}{\operatorname{Sub}}

\newcommand{\Pic}{\operatorname{Pic}}
\renewcommand{\Im}{\operatorname{Im}}
\newcommand{\End}{\operatorname{End}}
\newcommand{\Aut}{\operatorname{Aut}}
\newcommand{\Spec}{\operatorname{Spec}}
\newcommand{\Gr}{\operatorname{Gr}}

\newcommand{\Ext}{\operatorname{Ext}}

\newcommand{\Xsing}{\operatorname{X_{\rm{sing}}}}
\newcommand{\NF}{\operatorname{NF}}
\newcommand{\TF}{\operatorname{TF}}
\newcommand{\Simp}{\operatorname{Simp}}
\newcommand{\PIC}{\operatorname{PIC}}
\newcommand{\Gm}{\operatorname{\mathbb{G}_m}}

\newcommand{\ST}{\operatorname{ST}}
\newcommand{\Z}{\mathbb{Z}}

\pgfplotsset{compat=1.17}

\definecolor{LivGreen}{cmyk}{68 0 100 0}

\newtheorem{theorem}{Theorem}[section]
\newtheorem{corollary}[theorem]{Corollary}
\newtheorem{lemma}[theorem]{Lemma}

\newtheorem{proposition}[theorem]{Proposition}
\newtheorem{proposition-definition}[theorem]{Proposition-Definition}
\newtheorem{lemma-definition}[theorem]{Lemma-Definition}
\newtheorem{question}[theorem]{Question}

\newtheorem*{theoremalpha*}{Main Theorem}
\newtheorem*{theorem*}{Theorem}

\newtheorem{theoremalpha}{Theorem}

\theoremstyle{definition}
\newtheorem{definition}[theorem]{Definition}
\newtheorem{example}[theorem]{Example}
\newtheorem{remark}[theorem]{Remark}

\numberwithin{equation}{section}

\newenvironment{sis}{\left\{\begin{aligned}}{\end{aligned}\right.}

\begin{document}

\title{A new class of compactified Jacobians for families of reduced curves}

%\author{Marco Fava, Nicola Pagani, Filippo Viviani}
%\\ Dipartimento di Matematica, Universit\`a di Roma ``Tor Vergata'', Via della Ricerca Scientifica 1, 00133 Roma, Italy\\ viviani@mat.uniroma2.it}

%\author{}
%\date{}
\author{Marco Fava}
\address{Marco Fava, Department of Mathematical Sciences, University of Liverpool, Liverpool, L69 7ZL, United Kingdom}
\email{marco.fava@liverpool.ac.uk}

\author{Nicola Pagani}
\address{Nicola Pagani, Department of Mathematical Sciences, University of Liverpool, Liverpool, L69 7ZL, United Kingdom; Dipartimento di Matematica,  Universit\`a di Bologna, Piazza di Porta S. Donato, 5, 40126.}
\email{pagani@liv.ac.uk}
\email{nicolatito.pagani@unibo.it}
\urladdr{http://pcwww.liv.ac.uk/~pagani/}\urladdr{http://pcwww.liv.ac.uk/~pagani/}

\author{Filippo Viviani}
\address{Filippo Viviani, Dipartimento di Matematica, Universit\`a di Roma ``Tor Vergata'', Via della Ricerca Scientifica 1, 00133 Roma, Italy}
\email{viviani@mat.uniroma2.it}

\keywords{Compactified Jacobians, nodal curves.}

\subjclass[msc2000]{{14H10}, {14H40}, {14D22}.}

		\begin{abstract}
This is the first paper of a series of three. Here we give an abstract definition of the relative compactified Jacobian of a family of reduced curves. 
We prove that, under some mild assumptions on the family of curves, the fibres of the relative Jacobian are schemes (and not just algebraic spaces).

We define V-stability conditions, and use them to construct  relative compactified Jacobians. This extends the classical methods to produce modular compactifications of the Jacobian. To conclude, we show that, in the case when the curves have at worst planar singularities, the compactified Jacobians constructed from V-stability conditions have the same good properties of the classical ones.
	\end{abstract} %%%%%%%%%

\maketitle
	\bigskip
	
	\tableofcontents

%\section{Relative V-compactified Jacobians}

\section{Introduction}

The aim of this paper (which continues the work started in \cite{pagani2023stability} and \cite{viviani2023new}) is to produce a combinatorial method to construct compactifications of the Jacobian (or Picard) variety parameterizing line bundles on a fixed singular curve, and to similarly produce a method to construct relative compactifications in the case of a family of curves.

\vspace{0.1cm}

When $X$ is a reduced curve, a natural starting point to construct a modular compactification of the Jacobian is the moduli stack ${\TF}_X$ of torsion free, rank-$1$ sheaves on $X$.  By results of Altman-Kleiman \cite{altmankleiman}, if $X$ is also irreducible, then the rigidification $\TF_X \fatslash \Gm$ is a scheme, which may be decomposed into connected components ${\TF}^\chi_X\fatslash \mathbb{G}_m$,  indexed by the Euler characteristic of its points, and each component is proper.

\vspace{0.1cm}

When $X$ is reduced, connected but not necessarily irreducible, by work of Esteves \cite{esteves}, each moduli stack ${\TF}^\chi_X$ still satisfies the existence part of the valuative criterion of properness, but its rigidification is not a scheme (or an algebraic space) and it fails to be separated and of finite type  (and hence it is not proper).  Esteves' work constructs then open substacks $\overline{\mathcal{J}}_X$ of $\TF^\chi_X$ whose $\mathbb{G}_m$-rigidification $\overline{{J}}_X$ is a proper scheme, i.e. they are  \emph{fine compactified Jacobians} (in more classical terminology, there exists a Poincar\'e or tautological sheaf  on the product $X \times \overline{{J}}_X$ that is unique up to a pullback from $ \overline{{J}}_X$).

However, in the literature there are several examples of 'compactified Jacobians' of $X$ that are not fine in the sense of the previous paragraph, but are instead constructed as an open substack of $\TF^\chi_X$ (often out of some stability condition), and which admit a proper moduli space (often parameterizing polystable sheaves). 

We thus propose the following general definition: a \textbf{compactified Jacobian stack} $\ov \J_X^\chi$ of characteristic $\chi \in \mathbb{Z}$ of  a reduced curve $X$ is an open  substack  of the stack $\TF_X^\chi$  that admits a proper good moduli space $\ov J_X^\chi$ (in the sense of Alper \cite{Alp}), which is called the associated \textbf{compactified Jacobian space}  (see Definition~\ref{D:cJ-fam}). 
%(For the purposes of this introduction, we just mention that Alper's definition of a good moduli space for a stack generalizes the GIT notion of a good quotient)
 %The compactified Jacobian contains the generalized Jacobian of $X$, but the latter containment is, in general, not \emph{dense}.
\vspace{0.1cm}

%The generality of reduced curves, and the ideas behind  some of our arguments, are inspired by the seminal work \cite{esteves} by Esteves. Our work generalizes loc.cit. in two main directions. Firstly, the compactified Jacobians here are not necessarily \emph{fine} (by the latter we mean that the (rigidified) compactified Jacobian stack is proper). Secondly, 

Our main result is the construction of several compactified Jacobians as moduli spaces of sheaves that are semistable with respect to  what we call a \emph{V-stability condition}. This terminology was introduced in \cite{viviani2023new}, where 'general' V-stability conditions were introduced to construct \emph{fine} compactified Jacobians of nodal curves only. The terminology is motivated by the fact that one such stability condition consists of the datum of an integer for each partial smoothing of the starting curve to a \emph{vine curve}, i.e. to a curve with two irreducible components. The paper \cite{viviani2023new} established the existence of fine compactified Jacobians of nodal curves that come from a V-stability condition, but that cannot be constructed from Esteves' method.

\vspace{0.1cm}

This paper is the first of a series of three. In the second paper \cite{FPV} we restrict to the case of nodal curves and prove a completeness result for our theory: every compactified Jacobian that is also \emph{smoothable}, i.e. that can be extended to a relative compactified Jacobian for any smoothing of the nodal curve over  a trait, comes from some V-stability condition. (Exotic examples of connected yet non-smoothable fine compactified Jacobians for nodal curves have been constructed in \cite{PTgenus1}). This generalizes the analogue statement obtained for \emph{fine} compactified Jacobians of nodal curves from a combination of \cite{pagani2023stability} and \cite{viviani2023new}.  In the third paper \cite{FPV3}, we classify and study compactified universal Jacobians over the moduli stack $\Mbargn$ of $n$-marked stable curves of genus $g$,  generalizing earlier results in the fine case (\cite{pagani2023stability} for $n=0$ and \cite{fava2024}).

%In the special case where a compactified Jacobian stack $\ov \J_X^d$ is proper (which is equivalent to the fact that it coincides with its associated compactified Jacobian space $\ov J_X^d$), we say that $\ov \J_X^d=\ov J_X^d$ is a \emph{fine compactified Jacobian}. For fine compactified Jacobians, the above Question was solved by Pagani-Tommasi in  \cite{pagani2023stability} and by Viviani in \cite{viviani2023new}. The aim of this paper is to answer the above Question in the general case.
%where new subtleties appear. 

\vspace{0.1cm}

Non-fine compactified Jacobians appear naturally. For example, if $X$ is a nodal stable curve, then one can consider slope-stability with respect to the ample canonical polarization $\omega_X$ (as defined by Simpson \cite{simpson}); this yields  `canonical'  compactified Jacobians of any given Euler characteristic $\chi$, which are fine if and only if $\gcd(\chi, 2g(X)-2)=1$ and which can acquire different  `degree of degeneracy' for different $\chi$'s, with the maximally degenerate case appearing when $\chi$ is divisible by $2g(X)-2$.
%Vecchia frase col grado 
%degree $d$, which are fine if and only if $\gcd(d+1-g(X), 2g(X)-2)=1$ and which can acquire a different `degree of degeneracy' for different $d$'s, with the maximally degenerate case appearing when $d$ equals $g(X)-1$ modulo $2g(X)-2$. 
These canonical compactified Jacobians are the fibers of Caporaso's compactification of the universal Jacobian over the moduli space $\ov{\mathcal{M}}_g$ of stable curves of genus $g$ (see \cite{caporaso}) and in \cite{FPV3} we prove that these are the only compactified universal Jacobians over $\ov{\mathcal{M}}_g$ up to translation (a fact that was proved in \cite{pagani2023stability} in the \emph{fine} case only).

\vspace{0.1cm}

%Let us give some historical prospective to the results of this paper. 

The problem of finding compactified Jacobians is very old and it goes back to the work of Igusa in \cite{igusa}  and of Mayer-Mumford in \cite{mayermumford} in the 50's--60's. Since then, several authors have worked on the problem using different techniques and constructing a class of compactified Jacobians that we refer to as \emph{classical compactified Jacobians} (see later in the introduction for a precise definition of this notion): Oda-Seshadri \cite{Oda1979CompactificationsOT} using graph theory and toric geometry; Simpson \cite{simpson} using slope stability;  Caporaso \cite{caporaso} using GIT of embedded curves; Esteves \cite{esteves} using vector bundle stability. We refer to \cite[Sec. 1,2]{alexeev}, \cite[Sec. 2]{meloviviani}, \cite[Sec. 2.2]{CMKVlocal}, \cite[Sec. 2.2]{MRV}  for an account on the way the different constructions relate to one another.  The theory of V-stability conditions provides a common framework to describe all classical compactified Jacobians, and at the same time it produces new compactified Jacobians. 

The theory of compactified Jacobians of singular curves has seen a resurgence in recent years, for example in the context of the double ramification cycle theory (see the introduction of \cite{panda} and the references therein), in the context of the Hitchin system (see \cite{MSY} and the references therein) and in the context of a mathematical definition of the Gopakumar-Vafa invariants (see \cite{mauliktoda}).

\subsection*{The results}

 Our results are valid more generally for the case of families of reduced curves. In order to discuss the results, we fix some notation. 
 
 We start by introducing the notion of a V-stability condition for (torsion-free, rank~$1$) sheaves on a fixed reduced curve $X$. A subcurve $Y\subseteq X$ is a union of some of the irreducible components of $X$, and we say $Y$  is \emph{biconnected} if both 
$Y$ and its complementary subcurve $Y^\mathsf{c}:=\ov{X\setminus Y}$ are connected. We denote by $\BCon(X)$ the collection of all biconnected subcurves of $X$. (Note that, if $Y$ is biconnected, smoothing out the singular locus of $Y\setminus Y^\mathsf{c}$ and of $Y^\mathsf{c} \setminus Y$ produces a curve with two irreducible components, i.e. a \emph{a vine curve}).

A \emph{stability condition of vine type} (or simply a \textbf{V-stability condition}, see Definition~\ref{D:VStabX}) of characteristic $\chi$ on  $X$ is a function  
 \begin{align*}
        \s:\BCon(X)&\to \ZZ\\
        Y&\mapsto \s_Y
    \end{align*}
    satisfying the following:
\begin{enumerate}
\item  for any $Y\in \BCon(X)$, we have 
\begin{equation*}
\mathfrak s_Y+\mathfrak s_{Y^\mathsf{c}}-\chi
\in \{0,1\}.
\end{equation*}
%\begin{equation}\label{E:sum-n}
%\mathfrak n_Y+\mathfrak n_{Y^\mathsf{c}}-g(Y)-g(Y^c)+g(X)-d+1
%\in \{0,1\}.
%\end{equation}
A subcurve $Y\in \BCon(X)$ is said to be \emph{$\s$-degenerate} if $\s_Y+\s_{Y^c}-\chi=0$,
%$\mathfrak n_Y+\mathfrak n_{Y^\mathsf{c}}-g(Y)-g(Y^c)+g(X)-d+1=0$, 
and \emph{$\s$-nondegenerate} otherwise.

\item   for any subcurves $Y_1,Y_2,Y_3\in \BCon(X)$ without pairwise common irreducible components, and such that $X=Y_1\cup Y_2\cup Y_3$, we have:
\begin{enumerate}
 \item if two of the subcurves $Y_1,Y_2,Y_3$ are $\s$-degenerate, then so is  the third. %is $\s$-degenerate;
            \item the following property holds
            \begin{equation*}
            \sum_{i=1}^{3}\s_{Y_i}-\chi
            %\sum_{i=1}^{3}\n_{Y_i}-\sum_{i=1}^3 g(Y_i)+g(X)-d+2
            \in \begin{cases}
                \{1,2\} \textup{ if $Y_i$ is $\s$-nondegenerate for all $i=1,2,3$};\\
                \{1\} \textup{ if there exists a unique   } i\in \{1,2,3\} \text{ such that $Y_i$ is $\s$-degenerate};\\
                \{0\} \textup{ if $Y_i$ is $\s$-degenerate for all $i=1,2,3$}.
            \end{cases}
        \end{equation*}
\end{enumerate}
\end{enumerate}
The \emph{degeneracy subset} of $\s$ is the collection
\begin{equation*}
\D(\s):=\{Y\in \BCon(X): Y \text{ is $\s$-degenerate}\}.
\end{equation*}
    A V-stability condition $\s$ is called \emph{general}  if $\D(\s)=\emptyset$.

%The terminology of vine type (or V-type for short) is due to the following interpretation: given a biconnected subcurve $Y\subset X$, by smoothing all the nodes of $X$ except the ones in $Y\cap Y^\mathsf{c}$, we obtain a partial smoothing of $X$ to a vine curve $\wt X_Y$ and all such partial smoothings are obtained in this way. Then a pair $\{\n_Y,\n_{Y^\mathsf{c}}\}$ satisfying Condition \eqref{E:condi1} is equivalent to the datum of a compactified Jacobian of degree $d$ on the vine curve $\wt X_Y$ by Proposition \ref{P:vine-cJ}, and Condition \eqref{E:condi2} can be seen as a compatibility condition among the chosen compactified Jacobians on the partial smoothings $\{\wt X_Y\}_{Y\in \BCon(X)}$.

%Note that a  V-stability conditions on a nodal curve is the same as a V-stability condition on its dual graph in the sense of Definition~\ref{D:Vstab}.  

\vspace{0.1cm} 

Then we observe that both the notion of a compactified Jacobian and that of a V-stability condition can be extended to the relative case.  Let $X\to S$ be a family of connected reduced curves. 

A \emph{compactified Jacobian stack} (Definition~\ref{D:cJ-fam}) of characteristic $\chi$ for $X/S$ is an open substack $\ov \J_{X/S}^{\chi}$ of the stack $\TF_{X/S}^{\chi}$ parametrizing relative rank-$1$ torsion-free sheaves on $X/S$ of relative Euler characteristic $\chi$, which admits a relative proper good moduli space $\ov \J_{X/S}^{\chi}\to  \ov J_{X/S}^{\chi}\to S$, called its associated \emph{compactified Jacobian space}.

The notion of a V-stability condition on $X/S$ of characteristic $\chi$ is given in Definition~\ref{D:VStabXS} as the datum $\s = \{ \s^s\}$ of a V-stability condition of characteristic $\chi$ (as defined earlier) for each geometric $s$ point of $S$, which is also compatible with respect to all \'etale specializations of geometric points of $S$.

Our main result is that to any V-stability condition we can associate a compactified Jacobian:

\begin{theoremalpha}\label{ThmA} (Theorem \ref{T:VcJ})
%\begin{theoremalpha*}\label{ThmA} (Theorem \ref{T:VcJ}) 
Let $\pi:X\to S$ be a family of connected reduced curves over a  quasi-separated and locally Noetherian algebraic space $S$.
For any V-stability condition $\s= \{\s^s\}$ on $X/S$ of characteristic $\chi$,  the substack $\ov \J_{X/S}(\s)$  of  $\TF^\chi_{X/S}$ defined  by
\begin{equation*}
\ov \J_{X/S}(\s):=\{I \in \TF_{X/S}^\chi\: : \:  \chi(I_{Y_s})\geq \mathfrak \s^s_{Y_s}  \text{ for any geometric point $s$ of $S$ and any $ Y_s\in \BCon(X_s)$}\}.
\end{equation*}
(see Lemma-Definition~\ref{LD:VStab}) is a compactified Jacobian stack of $X/S$ of characteristic $\chi$, called  \textbf{the V-compactified Jacobian (stack) associated to $\s$}. 
%\end{theoremalpha*}
\end{theoremalpha}
\noindent The sheaves belonging to $\ov \J_{X/S}(\s)$ are called  $\s$-semistable and they are studied in detail in Section~\ref{Sub:ss}. 
We prove in Proposition~\ref{P:pt-VcJspa} that the geometric points of the corresponding compactified Jacobian \emph{space} $\overline{J}_{X/S}(\s)$ are the $\s$-polystable sheaves.

The proof of our Main Theorem is in three steps (in increasing order of difficulty). To prove openness, we show that the moduli stack of semistable sheaves is constructible and that it is stable under generalization (see Lemma-Definition \ref{LD:VStab}). To prove that it satisfies the existence part of the valuative criterion of properness, we generalize an existing algorithm by Esteves \cite{esteves} for the case of classical compactified Jacobians (see Proposition \ref{P:univ-cl}). Finally, to show the existence of a separated good moduli (algebraic) space, we use the criteria established in \cite{AHLH}, i.e. we prove that the moduli space of $\s$-semistable sheaves is $\Theta$-complete (see Proposition \ref{P:Th-compl}) and $S$-complete (see Proposition \ref{P:S-compl}). 

%In fact, both the definition of a compactified Jacobian and the notion of a V-stability condition are also given for a \emph{family} of reduced curves, and Theorem~\ref{T:VcJ} is stated and proved in that generality.
\vspace{0.1cm}

The classical compactified Jacobians (constructed in \cite{Oda1979CompactificationsOT}, \cite{simpson},  \cite{caporaso} and  \cite{esteves}) are special cases of V-compactified Jacobians, as we now explain. 
If $\psi$ is a  \emph{numerical polarization} on $X/S$ of characteristic $\chi\in \Z$, i.e. the datum $\psi=\{ \psi_s\}$ of an additive function 
$$
\psi_s:\{\text{Subcurves of } X_s\} \longrightarrow \RR \: \text{ such that } \psi_{X_s}=\chi,
$$
for all geometric points $s$ of $S$ that is in addition compatible under \'etale specializations of geometric points of $S$ (see Lemma-Definition~\ref{LD:numpol-fam}), then the datum $\s(\psi)= \{ \s(\psi)^s\}$ defined by 
 \begin{equation*}
 \s(\psi)^s:=\left\{\s(\psi^s)_Y:=\lceil \psi^s_Y \rceil  \: \text{ for any  } Y_s\in \BCon(X_s)\right\}
 \end{equation*}
 is a  V-stability condition of characteristic $\chi$, called the V-stability condition associated to $\psi$. We say that a V-compactified Jacobian associated to a V-stability of the form  $\mathfrak \s(\psi)$ is a \textbf{classical compactified Jacobian}. %associated to $\psi$:  
%\begin{equation}\label{E:fcJphi}
%\ov \J_X(\psi):=\ov \J_X(\mathfrak \s(\psi))=\left\{I \in \TF_X^\chi : \:  \chi(I_Y)\geq \psi_Y \text{ for any } Y\in \BCon(X)\right\}.
%\end{equation}
A first example of a non-classical (fine) V-compactified Jacobian for a family (of nodal curves) was found in \cite[Sec. 6]{PTgenus1}, and of a non-classical (fine) V-compactified Jacobian for a (nodal) curve over an algebraically closed field was found in \cite[Thm. C]{viviani2023new}. 

 \vspace{0.1cm}

Note that not all compactified Jacobians are V-compactified Jacobians. If $X$ is a reducible curve on an algebraically closed field, it is easy to see that there are exist always two (far away) V-stability conditions that give rise to V-compactified Jacobians whose union is not a V-compactified Jacobian. Adding the requirement that compactified Jacobians should be \emph{connected} does not solve the issue.  On the one hand, there are exotic examples from \cite[Sec. 3]{PTgenus1} of connected fine compactified Jacobians for nodal curves, that are not V-compactified Jacobians (see \cite[Thm. C]{viviani2023new}). Moreover, we do not yet know whether all V-compactified Jacobians of \emph{reduced} curves are connected (even in the classical case, this is a well-known open problem, see Remark~\ref{R: connected}), though we show that this is the case for connected curves with planar singularities (in Theorem~\ref{T:vcJ-fam-pla}).

 \vspace{0.1cm}

Then we prove that the class of V-compactified Jacobians of a fixed curve $X$ has good combinatorial and geometric properties. Firstly, we show via Definition~\ref{D:VStab-pos} that the set of all V-compactified Jacobians of $X$ can naturally be organized in a poset in such a way that the maximal elements are exactly the general V-stability conditions (see Corollary \ref{C:max-pos}); then we show in Proposition~\ref{P:fin-tran} that, modulo a natural translation action, there are only finitely many V-stability conditions of $X$. In particular (see Remark~\ref{R:inc-VcJ}) this shows that there are only finitely many isomorphism classes of V-compactified Jacobians of $X$.

Secondly, we investigate when relative compactified Jacobian spaces are (relative) schemes or projective schemes.

\begin{theoremalpha}\label{ThmB}
(Corollary \ref{C:projective} and Theorem~\ref{T:scheme}). Let $X\to S$ be a family of connected reduced curves over a quasi-separated and locally Noetherian algebraic space  $S$. Assume that the following holds: 

\un{Condition (*)}: There exist smooth sections $\{\sigma_i:S\to X\}_{i=1}^n$ of $X/S$ such that, for any $s\in S$, any irreducible component of $X_s$ contains $\sigma_i(s)$ for some $i$.

Then we have that:
\begin{enumerate}
\item  Any classical compactified Jacobian space $\ov J_{X/S}(\s(\psi))$ of $X/S$ is locally projective over $S$.
    \item Any compactified Jacobian space $\ov J^\chi_{X/S}$ of $X/S$ is representable by schemes over $S$.
\end{enumerate}  
\end{theoremalpha}
As an intermediate step, we also show in Theorem~\ref{T:projective} that Esteves' classical compactified Jacobians are always (i.e. even without assuming Condition (*)) locally projective.

Observe that Condition (*) implies that, for any $s\in S$, any irreducible component of the fiber $X_s$ is geometrically integral by \cite[\href{https://stacks.math.columbia.edu/tag/0CDW}{Tag 0CDW}]{stacks-project}.  By an example due to Mumford (see \cite[Sec. 8.2]{Neron}),  without the latter condition even the Picard algebraic space is not representable by schemes over $S$.  Furthermore, observe that Condition (*) is always satisfied if $S=\Spec k$ with $k$ an algebraically closed field.

%As an intermediate step, we also show in Theorem~\ref{T:projective} that, under mild extra assumptions, \emph{classical} compactified Jacobians are projective. 

\vspace{0.1cm}

A lot of interest has recently been drawn to the case of compactified Jacobians of curves whose singularities are planar (see e.g. \cite{MRV} and \cite{Migliorini_2021}), and in the last section we restrict ourselves to this class of curves. We extend the results from the classical case (from \cite{MRV}) to the general case of all V-compactified Jacobians:

\begin{theoremalpha}\label{ThmC}  (Theorem~\ref{T:VcJ-univ}, Corollary~\ref{C:VcJ-pla}, Theorem~\ref{T:vcJ-fam-pla}, Theorem~\ref{T:char-vcJ-pla}, Corollary~\ref{C:unic-s}).
     Let $X\to S$ be a family of connected reduced curves with planar singularities  over a quasi-separated and locally Noetherian algebraic space  $S$.
    \begin{enumerate}
        \item For any relative V-stability condition $\s$ on $X/S$, we have that 
        \begin{enumerate}
            \item The V-compactified Jacobian stack $F_{X/S}(\s):\ov \J_{X/S}(\s)\to S$ is flat of relative dimension $g(X)-1$ with geometric fibers that are connected, reduced with locally complete intersection singularities and trivial dualizing sheaf. Moreover, the relative smooth locus is the open substack parametrizing line bundles on $X/S$ and the total space $\ov \J_{X/S}(\s)$ is regular along the fiber over a point  $t\in S$ if the family $X/S$ is versal at $t$. 
     \item The V-compactified Jacobian space  $f_{X/S}(\s):\ov J_{X/S}(\s)\to S$ is (proper and) flat of relative dimension $g(X)$ with geometrically connected and reduced fibers.
 \end{enumerate}
\item A relative compactified Jacobian stack $\ov \J^\chi_{X/S}$ is a V-compactified Jacobian stack, i.e. it is equal to $\J_{X/S}(\s)$ for some relative V-stability condition $\s$ on $X/S$, if and only if it is such on every geometric fiber. In this case, the V-stability condition $\s$ is unique. 
\end{enumerate}
\end{theoremalpha}

\subsection*{Open questions}

\vspace{0.1cm}

This paper leaves open some natural questions:
\begin{enumerate}
\item What is the structure of the poset $\VStab(X)$ of V-stability conditions on a reduced curve $X$? For example, is it a ranked poset?

See Question \ref{Q:Deg} for more specific open problems.

\item Are V-compactified Jacobian spaces (or even all compactified Jacobian spaces) on a reduced curve (or on a family satisfying Condition (*) of Theorem \ref{ThmB}) projective? 

Theorem \ref{ThmB} gives partial results on this Question. 
%As a partial result, we prove in Section~\ref{sub:scheme} that, under the same mild assumptions on the family, classical compactified Jacobian spaces are projective (see Theorem~\ref{T:projective} and Corollary \ref{C:projective}) and that all compactified Jacobian spaces are schemes (see Theorem~\ref{T:scheme}).
    \item Is there any sense in which V-compactified Jacobians form a complete theory of compactified Jacobians for reduced curves? 

    We  prove in \cite{FPV, FPV3} that, for nodal curves, V-compactified Jacobians are exactly those that are limits of Jacobians under any $1$-parameter smoothing of the curve or over the semiuniversal deformation space of the curve. We ask whether similar results hold for more general reduced curves, e.g. for reduced curves with planar singularities. 
    
 %   \item What happens for nonreduced curves?
\end{enumerate}

\subsection*{Outline of the paper} 

\vspace{0.1cm}

The paper is organized as follows. In Section~\cref{Sub:curves} we fix the notation for reduced curves and their families. In Section~\ref{Sub:sheaves} we discuss torsion-free rank~$1$ sheaves on families of reduced curves, and prove some first results. In particular, we describe how such sheaves degenerate under $1$-parameter specializations (see Subsection~\cref{S:spec}) and under isotrivial specializations (see Subsection~\cref{S:isospec}). In Section~\cref{Sub:Vstab} we introduce V-stability conditions and their associated degeneracy subsets. In Subsection~\cref{Sub:numpol} we study the V-stability conditions associated to  numerical polarizations, that we call "classical" V-stability conditions. In Subsection~\cref{Sub:posetV} we define a poset structure on the collection of all V-stability conditions, and show that this poset admits an order-preserving and upper lifting map to the poset of degeneracy subsets and that it naturally admits a translation action with finitely many orbits. Finally, in Subsection~\cref{Sub:Vfamily}, we define relative V-stability conditions (and relative numerical polarizations) for families of reduced curves. In Section~\ref{Sub:ss}, we introduce semistable (resp. stable, resp. polystable) sheaves with respect to a V-stability condition and we prove that the category of semistable sheaves is stable under restriction to suitable subcurves, it is closed   under taking kernels, images, cokernels
and suitable extensions. Furthermore, we prove that every semistable sheaf isotrivially degenerates to a unique polystable sheaf and that any polystable sheaf is the direct sum of stable sheaves.
%to a given $V$-stability condition, we associate a corresponding collection of sheaves that are semistable with respect to it. We prove that the category of semistable sheaves is closed under taking kernels, images and cokernels. 
In Section \cref{Sub:VcJ}, we introduce compactified Jacobians as stacks and as algebraic spaces, and prove our main result: moduli of semistable sheaves with respect to a V-stability condition are a compactified Jacobian (that we call V-compactified Jacobian). In Section~\ref{sub:scheme} we prove that, under mild assumptions on the family, classical compactified Jacobian spaces are projective and that all compactified Jacobian spaces are schemes. Finally, in Section~\ref{sec:VcJ-planar} we focus on families of curves with locally planar singularities, and  we prove that, for such families, V-compactified Jacobian stacks and spaces enjoy some nice geometric properties. 
%that describe the geometry of their compactified Jacobians can be extended from the classical case to the case of the V-compactified Jacobians.  

\vspace{0.5cm}

\paragraph { \bf Acknowledgements}
MF is supported by the DTP/EPSRC award 
EP/W524001/1.

NP is funded by the PRIN 2022 ``Geometry Of Algebraic Structures: Moduli, Invariants, Deformations'' funded by MUR, and he is a member of the GNSAGA section of INdAM. 

FV is funded by the MUR  ``Excellence Department Project'' MATH@TOV, awarded to the Department of Mathematics, University of Rome Tor Vergata, CUP E83C18000100006, by the  PRIN 2022 ``Moduli Spaces and Birational Geometry''  funded by MUR,  and he is a member of INdAM and of the Centre for Mathematics of the University of Coimbra (funded by the Portuguese Government through FCT/MCTES, DOI 10.54499/UIDB/00324/2020).

We thank Dario Weissmann for some inspiring discussions that led to the proof of Theorem~\ref{T:scheme}.

\section{Reduced curves} \label{Sub:curves}

Let $X$ be a reduced (projective) curve over $k=\ov k$, and let $\gamma(X)$ be the number of connected components of $X$. The arithmetic genus of $X$ is 
$$g(X):=1-\chi(\O_X)=h^1(\O_X)+1-\gamma(X).$$
 Since $X$ is reduced, then it is Cohen-Macaulay and its dualizing sheaf $\omega_X$ is a rank-$1$ torsion-free sheaf of degree 
 $$\deg(\omega_X):=\chi(\O_X)+g(X)-1=-\chi(\O_X)+g(X)-1=2g(X)-2.$$
We will denote by $I(X)$ the set of irreducible components of $X$ and by $X_v$ the irreducible component of $X$ corresponding to $v\in I(X)$. We will denote the singular locus of $X$ by $\Xsing$.

A \emph{subcurve} $Y\subseteq X$ is a closed subscheme of $X$ that is either a curve or the empty scheme. In other words, a subcurve $Y\subseteq X$ is the union of some irreducible components of $X$, so that there is a bijection 
\begin{equation}\label{E:sub-Vert}
 \begin{aligned}
   \left\{\text{Subsets of } I(X)\right\} & \leftrightarrow \left\{\text{Subcurves of } X \right\}:=\Sub(X)  \\
   W &\mapsto X[W]:=\bigcup_{v\in W} X_v  \\
   I(Y) & \mapsfrom Y.
 \end{aligned}   
\end{equation}
We say that a subcurve $Y$ is non-trivial if $Y\neq \emptyset, X$, which happens if and only if $I(Y)\neq \emptyset, I(X)$. The complementary subcurve of $Y$ is 
$$Y^\mathsf{c}:=\ov{X\setminus Y}=X[I(Y)^c].$$

If $X$ is connected, then a subcurve $Y$ of $X$ is called \emph{biconnected} if it is connected and if its complementary subcurve is also connected (in particular, $Y$ is non-trivial). The set of biconnected subcurves of $X$ is denoted by $\BCon(X)$. The set of connected subcurves of $X$ (which does not include the empty subcurve) is denoted by $\Con(X)$. 
Note that if $Y\in \Con(X)$ then $Y^c=\coprod_i W_i$ with each $W_i\in \BCon(X)$.

We define the \emph{join} and the \emph{meet} of two subcurves by 
$$
\begin{sis}
& Y_1\cup Y_2:=X[I(Y_1)\cup I(Y_2)], \\
& Y_1\wedge Y_2:=X[I(Y_1)\cap I(Y_2)].
\end{sis}
$$
In other words, the join of two subcurves is simply their union, while the meet of two subcurves is the union of their common irreducible components. 

% INTERSEZIONE di sottocurve
%Given two subcurves $Y_1,Y_2$ of $X$ without common irreducible components (i.e. such that $Y_1\wedge Y_2=\emptyset$), we denote by $Y_1\cap Y_2$ the (zero-dimensional) scheme-theoretic intersection of $Y_1$ and $Y_2$ and by $|Y_1\cap Y_2|$ its length. From the exact sequence
%$$
%0\to \O_{Y_1\cup Y_2}\to \O_{Y_1}\oplus \O_{Y_2}\to \O_{Y_1\cap Y_2}\to 0,
%$$
%we deduce 
%\begin{equation}\label{E:inter}
%    |Y_1\cap Y_2|=\chi(\O_{Y_1})+\chi(\O_{Y_2})-\chi(\O_{Y_1\cup Y_2})=g(Y_1\cup Y_2)-g(Y_1)-g(Y_2)+1.
%\end{equation}

A \emph{family of connected reduced curves} $\pi:X\to S$  over an algebraic stack $S$ is a projective and flat morphism $\pi$ whose geometric fibers are connected reduced curves.

\section{Torsion-free, rank-1 sheaves on reduced curves}\label{Sub:sheaves}

Let $X$ be a connected reduced curve over an algebraically closed field $k$. We say that  a coherent sheaf $I$ on $X$ is:
\begin{itemize}
\item  \emph{torsion-free on $X$}  if its associated points are the generic points of $X$, or equivalently if $I$ is pure of dimension $1$ (i.e. $I$ does not have torsion subsheaves) with support $\supp(I)$ equal to $X$.
%The support of $I$, denoted by $\supp(I)$, is a subcurve of $X$. We say that $I$ is a torsion-free sheaf \emph{on $X$} if $I$ is torsion-free with $\supp(I)=X$. 
\item \emph{rank-$1$} if $I$ has rank one on each generic point of $\supp(I)$. 
\item  \emph{simple} if $\End(I) = k $, or equivalently if $\Aut(I)=\Gm$.
\end{itemize}
Note that a torsion-free sheaf $I$ on $X$ is locally free away from the singular locus of $X$ and that each line bundle on $X$ is a simple, rank-$1$, torsion-free sheaf on $X$. 

The \emph{degree} of a torsion-free rank-$1$ sheaf $I$ on $X$ is defined to be
\begin{equation}\label{E:degI}
\deg(I)=\chi(I)-\chi(\O_X). 
\end{equation}

There are two natural ways of "restricting" a  torsion-free sheaf on $X$ to a torsion-free sheaf on a subcurve $Y\subseteq X$, as explained in the following.

\begin{definition}\label{D:IY}
    Let $I$ be a torsion-free sheaf on $X$ and let $Y$ be a subcurve of $X$. We denote by:
    \begin{itemize}
        \item $I_Y$ the quotient of the restriction $I_{|Y}$ modulo the torsion subsheaf; 
        \item $\leftindex_{Y^c}{I}$ the kernel of the surjection $I\twoheadrightarrow I_Y$.
    \end{itemize}
    Hence,  we have an exact sequence 
\begin{equation}\label{E:resY}
0\to \leftindex_{Y^c}{I}\to I \to I_Y\to 0,
\end{equation}
\end{definition}
from which we deduce the equality 
\begin{equation}\label{E:add-chi}
    \chi(I)=\chi(I_Y)+\chi(\leftindex_{Y^c}{I}).
\end{equation}
%We let $\deg_Y (I)$ denote the degree of $I_Y$, that is, $\deg_Y(I) := \chi(I_Y )-\chi(\O_Y)$.

\begin{remark}\label{R:IY}
Note that: 
    \begin{enumerate}[(i)]
\item $I_Y$ is torsion-free  on $Y$ and it is the smallest quotient of $I$ with support equal to $Y$.
\item $\leftindex_{Y^c}{I}$ is torsion-free  on $Y^c$ and it is the largest subsheaf of $I$ that is supported on $Y^c$.
\end{enumerate}
\end{remark}

The relation between the two constructions of Definition \ref{D:IY} is clarified by the following 

\begin{lemma}\label{L:IY}
    Let $I$ be a torsion-free sheaf on $X$ and let $Y$ be a subcurve of $X$. 
    \begin{enumerate}[(i)]
    \item \label{L:IY1} The following maps are injective: 
    \begin{enumerate}[(a)]
    \item the compositions $\leftindex_{Y}{I}\hookrightarrow I \twoheadrightarrow I_Y$ and  $\leftindex_{Y^c}{I}\hookrightarrow I \twoheadrightarrow I_{Y^c}$;
    \item the simultaneous restriction map $I\to I_Y\oplus I_{Y^c}$;
    \item the simultaneous inclusion map $\leftindex_{Y}{I}\oplus \leftindex_{Y^c}{I}\to I$.
   \end{enumerate}
    \item \label{L:IY2} We have the following canonical isomorphisms 
    $$\frac{I_Y}{\leftindex_{Y}{I}}\xleftarrow{\cong}\frac{I_Y\oplus I_{Y^c}}{I} \xrightarrow{\cong}  \frac{I_{Y^c}}{\leftindex_{Y^c}{I}} \xleftarrow{\cong}\frac{I}{\leftindex_{Y}{I}\oplus \leftindex_{Y^c}{I}} \xrightarrow{\cong}  \frac{I_{Y}}{\leftindex_{Y}{I}}$$
    \end{enumerate}
    of sheaves supported on the $0$-dimensional subscheme $Y\cap Y^c$. 
\end{lemma}
\begin{proof}
From the commutative diagram
 \[
\begin{tikzcd}
  & \leftindex_{Y}{I} \arrow[r, equal] \arrow[d, hookrightarrow] & \leftindex_{Y}{I} \arrow[d]   \\
     \leftindex_{Y^c}{I} \arrow[r, hookrightarrow] \arrow[d, equal] & I \arrow[r, twoheadrightarrow] \arrow[d, twoheadrightarrow] & I_Y  \\
       \leftindex_{Y^c}{I} \arrow[r]  & I_{Y^c}  &    
\end{tikzcd}
\]
in which the central row and column are exact, we deduce that:

$\bullet$ $\ker(\leftindex_{Y^c}{I}\to I_{Y^c})\cong \ker(\leftindex_{Y}{I}\to I_{Y})=0$ because it is a torsion-free sheaf that is contained in $Y\cap Y^c$.

$\bullet$ $\coker(\leftindex_{Y^c}{I}\to I_{Y^c})\cong \coker(\leftindex_{Y}{I}\to I_{Y})$, which is supported on $Y\cap Y^c$.

From the commutative diagram
 \[
\begin{tikzcd}
 \leftindex_{Y}{I} \arrow[r,hookrightarrow] \arrow[d]& I  \arrow[r, twoheadrightarrow]\arrow[d] & I_{Y^c}\arrow[d,equal]     \\
    I_Y \arrow[r, hookrightarrow]  & I_Y\oplus I_{Y^c}  \arrow[r, twoheadrightarrow]  & I_{Y^c}        
\end{tikzcd}
\]
in which the two rows are exact and the left vertical arrow is injective as seen above, we deduce that:

$\bullet$ the map $I\to I_Y\oplus I_{Y^c}$ is injective.

$\bullet$ $\coker(I\to I_Y\oplus I_{Y^c})\cong \coker(\leftindex_{Y}{I}\to I_{Y})$. 

Finally, from the commutative diagram
 \[
\begin{tikzcd}
\leftindex_{Y}{I} \arrow[d, equal]\arrow[r, hookrightarrow] & \leftindex_{Y}{I}\oplus \leftindex_{Y^c}{I} \arrow[r,twoheadrightarrow] \arrow[d]& \leftindex_{Y^c}{I}\arrow[d] \\
 \leftindex_{Y}{I} \arrow[r,hookrightarrow]& I  \arrow[r, twoheadrightarrow] & I_{Y^c}     
\end{tikzcd}
\]
in which the two rows are exact and the right vertical arrow is injective by what proved above, we deduce that:

$\bullet$ the map $\leftindex_{Y}{I}\oplus \leftindex_{Y^c}{I}\to I$ is injective.

$\bullet$ $\coker(\leftindex_{Y}{I}\oplus \leftindex_{Y^c}{I}\to I)\cong \coker(\leftindex_{Y^c}{I}\to I_{Y^c})$. 
 
 We obtain the proof of the Lemma by combining all of the above.
\end{proof}

\begin{example}\label{Ex:IY-nodal}
    Let $X$ be a nodal curve and let $I$ be a rank-$1$ torsion-free sheaf on $X$. Then for any subcurve $Y\subseteq X$, we have that 
    $$
    \leftindex_Y{I}=I_Y(-(Y\cap Y^c\cap \NF(I)^c)),
    $$
    where $\NF(I)$ is the set of nodes of $X$ at which $I$ is not free. 
    Therefore, $\displaystyle \frac{I_Y}{\leftindex_Y{I}}$ is the structure sheaf of the $0$-dimensional reduced scheme $Y\cap Y^c\cap \NF(I)^c$. In particular, $I$ splits at $Y$ (i.e. $I=I_Y\oplus I_{Y^c}$) if and only if $I$ is not free at all nodes of $Y\cap Y^c$.
\end{example}

We now investigate the behaviour of the constructions in Definition \ref{D:IY} with respect to two subcurves $Y\subset Z$. 

\begin{lemma}\label{L:IYZ}
   For any  torsion-free sheaf $I$ on $X$ and any two subcurves $Y\subset Z\subseteq X$, we have: 
\begin{enumerate}[(i)]
\item \label{L:IYZ1}  $(I_Z)_{Y}=I_{Y} \quad \text{ and }  \quad \quad \leftindex_{Y}{I}=\leftindex_{Y}{(\leftindex_{Z}I)}$.
%\item \label{L:IYZ2}  $\leftindex_{Y}{(I_Z)}\cong\leftindex_{Y}{I}$.
%\textcolor{green}{Wrong!}
%\item \label{L:IYZ3} $(\leftindex_{Z}{I})_Y\cong I_Y$.
%\textcolor{green}{Wrong!}
\item \label{L:IYZ2} $\leftindex_{Z-Y}(I_Z)=(\leftindex_{Y^c}{I})_{Y^c- Z^c}$.
\end{enumerate}
\end{lemma}
\begin{proof}
 Part \eqref{L:IYZ1}: $(I_Z)_Y$ is a quotient of $I$ that is supported on $Y$ and is torsion-free, hence it must equal $I_Y$.
On the other hand, the sheaf $\leftindex_{Y}{(\leftindex_{Z}I)}$ is a subsheaf of $I$ supported on $Y$ and it is saturated, i.e. the quotient of $I$ by $\leftindex_{Y}{(\leftindex_{Z}I)}$ is torsion-free, and hence it must be equal to $\leftindex_{Y}{I}$.

 In order to prove Part \eqref{L:IYZ2}, consider the following commutative diagram 
 \[
\begin{tikzcd}
    0\arrow[r] & \leftindex_{Y^c}{I} \arrow[r] \arrow[d, "\alpha"] & I \arrow[r] \arrow[d, "\beta"] & I_{Y} \arrow[r]\arrow[d, equal] & 0  \\
     0\arrow[r] & \leftindex_{Z-Y}{(I_Z)} \arrow[r]  & I_Z \arrow[r] & (I_Z)_Y\arrow[r] & 0,
\end{tikzcd}
\]
with exact rows, and where we have used that $(I_Z)_Y=I_Y$ by part \eqref{L:IYZ1}.
By applying the snake lemma, we get that 
$$
\begin{sis}
& \coker \alpha\cong \coker \beta=0,\\   
& \ker \alpha\cong \ker \beta=\leftindex_{Z^c}{I}.
\end{sis}
$$
Therefore, using that $\leftindex_{Z^c}{I}=\leftindex_{Z^c}{(\leftindex_{Y^c}{I})}$ by Part \eqref{L:IYZ1}, we get that $\alpha$ induces the required  isomorphism 
$$
(\leftindex_{Y^c}{I})_{Y^c-Z^c}=\frac{\leftindex_{Y^c}{I}}{\leftindex_{Z^c}{(\leftindex_{Y^c}{I})}}\xrightarrow{\cong} \leftindex_{Z-Y}{(I_Z)}.
$$ 
\end{proof}

Any rank-$1$ torsion-free sheaf $I$ on $X$ admits a canonical decomposition
\begin{equation}\label{E:decI}
I=I_{Y_1}\oplus \cdots \oplus I_{Y_{c(I)}},
\end{equation}
such that each $I_{Y_i}$ is indecomposable on $Y_i$, i.e. it cannot be written in a non-trivial way as a sum of sheaves of the form $I_{Z}$ for $Z\subset Y_i$.  The automorphism group of $I$ is equal to 
\begin{equation}\label{E:AutI}
\Aut(I)=\Gm^{c(I)},
\end{equation}
where each copy of the multiplicative group acts on the corresponding summand in \eqref{E:decI} by scalar multiplication. In particular, a rank-$1$ torsion-free sheaf $I$ is simple (i.e. $\Aut(I)=\Gm$) if and only if $I$ is indecomposable.

For any family $\pi:X\to S$ of connected reduced curves over an algebraic stack $S$, we denote by $\TF_{X/S}$ the  \emph{stack of relative torsion-free rank-$1$ sheaves} on $X/S$.  More precisely, $\TF_{X/S}$ is the stack  over $S$-schemes such that the fiber over a $S$-scheme $T$ is the groupoid of $T$-flat coherent sheaves $\I$ on $X\times_S T$ such that $\I|_{X_t}$ is a torsion-free, rank-$1$ sheaf on $X_t$ for each geometric point $t$ of $T$. The stack $\TF_{X/S}$ is an algebraic stack that is quasi-separated, locally of finite type over $S$ and its diagonal is affine and of finite presentation over $S$ by \cite[\href{https://stacks.math.columbia.edu/tag/0DLY}{Lemma 0DLY}]{stacks-project} and \cite[\href{https://stacks.math.columbia.edu/tag/0DLZ}{Lemma 0DLZ}]{stacks-project}. We denote by  $\I$ the universal sheaf on $X\times_S \TF_{X/S}$. Since $\Gm$ sits in the automorphism group of each sheaf $I\in \TF_{X/S}$ as the scalar automorphism, we can form the rigidification $\TF_{X/S}\to \TF_{X/S}\fatslash \Gm$ which is a $\Gm$-gerbe. 

The stack $\TF_{X/S}$ contains two open substacks
$$\PIC_{X/S}\subseteq \Simp_{X/S}\subseteq \TF_{X/S},$$ 
where $\Simp_{X/S}$ is the open substack parameterizing relative simple  sheaves, and $\PIC_{X/S}$ is the Picard stack of $X$, which parameterizes relative line bundles. Note that $\Simp_{X/S}\fatslash \Gm$ is the largest open substack of $\TF_{X/S}\fatslash \Gm$ such that  the morphism $\Simp_{X/S}\fatslash \Gm\to S$ is representable.  

We have a decomposition into open and closed substacks  
\begin{equation}\label{E:Torsd}
\TF_{X/S}=\coprod_{\chi\in \ZZ} \TF^\chi_{X/S}
\end{equation}
where $\TF^\chi_{X/S}$ parameterizes sheaves of Euler characteristic  $\chi$. The decomposition of \eqref{E:Torsd} induces the following decompositions 
$$
\Simp_{X/S}=\coprod_{\chi\in \ZZ} \Simp^\chi_{X/S} \: \text{ and }\: \PIC_{X/S}=\coprod_{\chi\in \ZZ} \PIC^\chi_{X/S}.
$$

%\begin{fact}\label{F:diagTF}
%\noindent 
%    \begin{enumerate}
%        \item \label{F:diagTF1} The stack $\TF_X$ is quasi-separated and locally of finite type over $k$.
%        \item \label{F:diagTF2} The diagonal of $\TF_X$ is affine and of finite presentation.
%    \end{enumerate}
%\end{fact}

%The Picard variety $\PIC_X$ of $X$ acts on $\TF_X$ via tensor product. We can restrict this action to the subgroups 
%$$
%\PIC_X^{\un 0}\subseteq \PIC_X^0\subseteq \PIC_X
%$$
%where $\PIC^{\un 0}_X$, called the \emph{generalized Jacobian} of $X$, is the smooth commutative algebraic group parametrizing  line bundles of multidegre zero.

\subsection{Specializations}\label{S:spec}

In this subsection, we study the $1$-parameter specializations of torsion-free rank-$1$ sheaves on connected reduced curves.

Let $R$ be  a discrete valuation ring with residue field $k=\ov k$ and quotient field $K$, and denote by $\val:K\to \ZZ\cup\{\infty\}$ the associated valuation. Set $\Delta:=\Spec R$ with generic point $\eta:=\Spec K$ and special point $o:=\Spec k$. 

Let $\pi:X\to \Delta$ be a family of connected reduced curves, and denote by $X_o$ the special fiber of $\pi$ and by $X_{\eta}$ the generic fiber of $\pi$.   
Given a relative torsion-free rank-$1$ sheaf $\I\in \TF_{X/\Delta}(\Delta)$, we denote by $\I(\eta)\in \TF_{X_{\eta}}(\eta)$ its generic fiber and by $\I(o)\in \TF_{X_o}(k)$ its special fiber.  

Our goal here is to describe the fibers of the  map 
$$
\begin{aligned}
(-)_{\eta}: |\TF_{X/\Delta}(\Delta)|& \to |\TF_{X_{\eta}}(\eta)|\\
\I & \mapsto \I(\eta),
\end{aligned}
$$
where $|\TF_{X/\Delta}(\Delta)|$ denotes the set of isomorphism classes of the groupoid $\TF_{X/\Delta}(\Delta)$ and similarly for $|\TF_{X_{\eta}}(\eta)|$.

With this aim, we introduce the twist of a sheaf $\I\in \TF_{X/\Delta}(\Delta)$ by a subcurve of the central fiber.

\begin{definition}\label{D:twist}
  Let $\I$ be a relative rank-$1$ torsion-free sheaf on $X/\Delta$.
  Then  the \emph{twist} of $\I$ by a subcurve $Y \subseteq X_o$ is defined by
  $$
  \I^Y:=\ker(\I\twoheadrightarrow \I(o)\twoheadrightarrow \I(o)_Y).
  $$
\end{definition}

We summarize in the following lemma the properties of the twisting operation.

\begin{lemma}\label{L:twist}
    Let $\I$ be a relative rank-$1$ torsion-free sheaf on $X/\Delta$ of relative degree $d$ and let $Y\subseteq X_o$ be a subcurve. 
\begin{enumerate}[(i)]
    \item \label{L:twist1}  $\I^Y$ is a relative rank-$1$ torsion-free sheaf on $X/\Delta$ of relative degree $d$ such that $(\I^Y)_{\eta}=\I_\eta$.
    \item \label{L:twist2}  We have that 
    $$
    \begin{sis}
     & \Im(\I^Y(o)\to \I(o))=\I^Y(o)_{Y^c}=\leftindex_{Y^c}{\I(o)}, \\
     & \leftindex_{Y}{\I^Y(o)}=\ker(\I^Y(o)\to \I(o))\cong \coker(\I^Y(o)\to \I(o))=\I(o)_Y.
    \end{sis}$$
    %\begin{equation}\label{E:IIY}
   %\begin{tikzcd}
    %0 & \I^Y(o)_{Y^c}\arrow[l] \arrow[d, equal] & \I^Y(o) \arrow[l] %\arrow[d] & \I(o)_Y \arrow[l] & 0\arrow[l]  \\
    % 0\arrow[r] & \leftindex_{Y^c}{\I(o)} \arrow[r]  & \I(o) \arrow[r] %&\I(o)_Y \arrow[r] & 0, 
  %\end{tikzcd}
  %\end{equation}

    \item \label{L:twist3} If $Z\subseteq X_o$ is another subcurve then 
    $$
    \begin{sis}
    & \I^{Y\cup Z}=\I^Y\cap \I^Z\subseteq \I,\\
     & (\I^Y)^Z=(\I^Z)^Y=\I^Y\cap \I^Z\subseteq \I \text{ if } Y\wedge Z=\emptyset.
    \end{sis}$$
    \item \label{L:twist4} We have that $\I^{X_o}=t \I\cong \I$.
    \item \label{L:twist5} If $Y,Z$ are two subcurves of $X_o$ such that $Y\wedge Z=\emptyset$, then we have that 
    $
    \leftindex_{Z}\I^Y(o)\cong \leftindex_{Z}\I(o).
    %\quad \text{ and } \quad \I^Y(o)_Z\cong \I(o)_Z. 
    $
\end{enumerate}
\end{lemma}
\begin{proof}
Part \eqref{L:twist1}: the fact that $\I^Y$ is a relative rank-$1$ torsion-free sheaf on $X/\Delta$ can be shown as in \cite[Prop. 6]{langton}. The equality $(\I^Y)_\eta=\I_\eta$ follows from the fact that the inclusion $\I^Y\hookrightarrow \I$ is an equality on $X\setminus Y$ by definition.

Part \eqref{L:twist2} is proved in \cite[Lemma 23(1)]{esteves}.

Part \eqref{L:twist3}: first of all, observe that 
$$
\I^Y\cap \I^Z=\ker(\I\rightarrow \I(o)_Y\oplus \I(o)_Z).
$$
Since we have a factorization 
$$
\I\twoheadrightarrow \I(o)_{Y\cup Z}\hookrightarrow \I(o)_Y\oplus \I(o)_Z,
$$
where the second map is injective, we get that $\I^{Y\cup Z}=\I^Y\cap \I^Z$. 

Assume now that $Y\wedge Z=\emptyset$. Consider the commutative diagram  
 \[
\begin{tikzcd}
    0\arrow[r] & (\I^Y)^Z \arrow[r] \arrow[d, hookrightarrow] & \I^Y \arrow[r] \arrow[d, hookrightarrow] & \I^Y(o)_Z \arrow[r]\arrow[d] & 0  \\
     0\arrow[r] & \I^Z \arrow[r]  & \I \arrow[r] & \I(o)_Z \arrow[r] & 0,
\end{tikzcd}
\]
which has exact rows and injective central and left vertical arrows by construction. Since $\I/\I^Y$ is supported on $Y$ and $Y\wedge Z=\emptyset$, then the right vertical arrow $\I^Y(o)_Z\to \I(o)_Z$ is also injective. Using this and the above commutative diagram, we get that $(\I^Y)^Z=\I^Y\cap \I^Z$. By exchanging the roles of $Y$ and $Z$, we also get the other equality $(\I^Z)^Y=\I^Y\cap \I^Z$.

Part \eqref{L:twist4}: the equality $\I^{X_o}=t \I$ follows from the fact that $\I^{X_o}=\ker(\I\twoheadrightarrow \I(o))$. The isomorphism $\I\xrightarrow{\cong} t\I$ is given by multiplication by $t$.

Part \eqref{L:twist5}: denote by $J$ the image of $\leftindex_Z{\I^Y(o)}$ via the morphism $\psi: \I^Y(o)\to \I(o)$. Part \eqref{L:twist2} implies that $J\subseteq \leftindex_{Y^c}{\I(o)}=\Im(\psi)$ and that the restriction of $\psi$ to $\leftindex_Z{\I^Y(o)}$ is injective since $\leftindex_Z{\I^Y(o)}\cap \ker(\psi)=\leftindex_Z{\I^Y(o)}\cap \leftindex_{Y}{\I^Y(o)}=0$ because of our assumption that $Y\wedge Z=\emptyset$. Consider now the following commutative diagram
$$
\begin{tikzcd}
 0 \arrow[r] \arrow[d] &\leftindex_{Y}{\I^Y(o)}\arrow[r,equal] \arrow[d,hookrightarrow] & \leftindex_Y{\I^Y(o)} \arrow[d, hookrightarrow] \\
\leftindex_Z{\I^Y(o)} \arrow[r, hookrightarrow] \arrow[d, "\psi"', "\cong"] & \I^Y(o) \arrow[r,twoheadrightarrow] \arrow[d,twoheadrightarrow] & \I^Y(o)_{Z^c} \\ 
J \arrow[r,hookrightarrow] & \leftindex_{Y^c}{\I(o)} & 
\end{tikzcd}
$$
in which the two upper rows and two left columns are exact, and the morphisms in the last row and and in the  right column are injective. 
Therefore, we get that 
$$\coker(J\hookrightarrow \leftindex_{Y^c}{\I(o)})\cong \coker(\leftindex_Y{\I^Y(o)}\hookrightarrow \I^Y(o)_{Z^c})=\I^Y(o)_{Y^c\cap Z^c}. $$
Hence, $J$ is a subsheaf of $\leftindex_{Y^c}{\I^(o)}$ supported on $Z\subseteq Y^c$ whose quotient is torsion-free, which implies (see Remark \ref{R:IY}) that $J=\leftindex_{Z}(\leftindex_{Y^c}{\I(o)})=\leftindex_{Z}{\I(o)}$. Therefore, we have proved that $\psi$ induces the required isomorphism $\leftindex_Z{\I^Y(o)}\xrightarrow{\cong} \leftindex_{Z}{\I(o)}$.
\end{proof}

Recall that $I(X)$ denotes the set of irreducible components of $X$. Using Lemma \ref{L:twist}\eqref{L:twist3}, we can  extend  the above twisting construction linearly, by defining  for any $g:I(X)\to \NN$ and any $\I\in \TF_{X/\Delta}(\Delta)$:
$$
g(\I):=\I^{\sum g(v) (X_o)_v},
$$
where $\I^{\sum g(v) (X_o)_v}$ denotes the iteration of the twisting construction $g(v)$ times for each irreducible component $(X_o)_v\subseteq X_o$. 
Moreover, since $\I^{nX_o}=t^n\I\cong \I$ for any $n\in \NN$ by Lemma \ref{L:twist}\eqref{L:twist4}, we can define an action  \begin{equation}\label{E:action}
\begin{aligned}
\frac{\ZZ^{I(X)}}{\ZZ}\times |\TF_{X/\Delta}(\Delta)|& \longrightarrow |\TF_{X/\Delta}(\Delta)|\\
([g],\I) & \mapsto g(\I).
\end{aligned}
\end{equation}
where $\ZZ$ is embedded in $\ZZ^{I(X)}$  diagonally, and the representative $g$ of $[g]$ is chosen to be with all values in $\NN$.

\begin{theorem}\label{T:limits}
The action of \eqref{E:action} induces a bijection
$$
\begin{aligned}
\ov{(-)_{\eta}}:\frac{|\TF_{X/\Delta}(\Delta)|}{\ZZ^{I(X)}/\ZZ}& \xrightarrow{\cong} |\TF_{X_{\eta}}(\eta)|\\
[\I] &\mapsto \I_{\eta}.
\end{aligned}
$$
\end{theorem}
\begin{proof}
The fact that the generic fiber map $(-)_{\eta}$ descends to the quotient by the action of \eqref{E:action} follows from the fact that the twisting operation does not change the generic fiber by Lemma \ref{L:twist}\eqref{L:twist1}.

The fact that the map $()_{\eta}$ (and hence also $\ov{(-)_{\eta}}$) is surjective follows from \cite[Lemma 7.8(i)]{altmankleiman}. 

In order to prove the injectivity of $\ov{(-)_{\eta}}$, we argue similarly to the proof of \cite[Prop. 26]{esteves}. Suppose that $\I,\J\in \TF_{X/\Delta}(\Delta)$ are such that $\I(\eta)\cong \J(\eta)$. Then there exists $\lambda:\I\to \J$ such that $\lambda(\eta)$ is an isomorphism and $\lambda(o)\neq 0$. The image of $\lambda(o)$ is (non-zero) torsion-free rank-$1$ on a subcurve of $X_o$ since it is contained in $\J(o)$, and hence there exists a  subcurve $Z_1\subsetneq X_o$ such that $\Im(\lambda(o))=\I(o)_{Z_1^c}$. If $Z_1=\emptyset$ then $\lambda(o)$ is an embedding and, since $\lambda(\eta)$ is an isomorphism, it follows that $\lambda(o)$ is an isomorphism. Otherwise, $\lambda$ factors as $\lambda:\I\xrightarrow{\lambda_1} \J_1:=\J^{Z_1}\subsetneq \J$. Arguing as before, we find a  subcurve $Z_2\subsetneq X_o$ such that $\Im(\lambda_1(o))=\I(o)_{Z_2^c}$. 
Note that $Z_1\supseteq Z_2$ since the image of $\lambda_1(o)$ surjects onto the image of $\lambda(o)$. If $Z_2=\emptyset$ then $\lambda_1$ is an isomorphism, otherwise we iterate the argument. After finitely many steps, the process  terminates and we end up with a chain of subcurves $$X_o\supsetneq Z_1\supseteq Z_2\supseteq \ldots \supseteq Z_r\supsetneq \emptyset$$ 
such that  $\lambda$ gives rise to an isomorphism $\lambda:\I \xrightarrow{\cong} \J^{\sum_i Z_i}\subseteq \J$. Therefore, $\I\cong g(\J)$ for a certain $g\in \ZZ^{I(X)}$ and we are done. 
\end{proof}

\begin{remark}\label{R:twist}
\noindent
\begin{enumerate}
    \item If $\pi:X=X(o)\times \Delta\to \Delta$ is the constant family of a nodal curve $X(o)$, then the action of \eqref{E:action} coincides with the explicit action defined in \cite[Eq. (2.17)]{viviani2023new} and Theorem \ref{T:limits} coincides with \cite[Prop. 2.7(1)]{viviani2023new}.
    %was examined in detail in \cite[Sec. 2.2.1]{viviani2023new} and \cite[Sec. 4.2.1]{FPV}.
    \item If the total space $X$ is regular and the generic fiber $X_\eta$ is smooth, then we have that $\TF_{X/\Delta}(\Delta)=\PIC_{X/\Delta}(\Delta)$. In this case,  the action of \eqref{E:action} is the standard tensorization of a line bundle on $X$ by twistors, i.e. line bundles associated to Cartier divisors supported on $X_o$, and Theorem \ref{T:limits} says that two line bundles on $X$ that agree on the generic fiber differ by a twistor (see \cite{raynaud70}).
\end{enumerate}
\end{remark}

\subsection{Isotrivial specializations}\label{S:isospec}

In this subsection we describe a special class of specializations in the stack $\TF_X$, where $X$ is a connected reduced curve over $k=\ov k$.

Consider the quotient stack $\Theta_k:=[\AA_k^1/\Gm]$. This stack has two $k$-points: the open point $1:=[\AA^1\setminus\{0\}/\Gm]$ with trivial stabiizer and the closed point $0:=[0/\Gm]=B\Gm$ with stabilizer equal to $\Gm$. 
Given two sheaves $I,J\in \TF_X(k)$, we say that $J$ is an \emph{isotrivial (or very close) specialization} of $\I$ if there exists a morphism $f:\Theta_k\to \TF_X$ such that $f(1)=I$ and $f(0)=J$. The morphism $f:\Theta_k\to \TF_X$ is called an isotrivial (or very close) specialization from $f(1)$ to $f(0)$. 

In the next proposition, we describe isotrivial specializations in $\TF_X$. With this aim, consider the following construction. 

An ordered partition of $X$ (by subcurves)
    $$
    Y_{\bullet}=(Y_0, \ldots, Y_q),
    $$
    is an ordered collection of subcurves covering $X$ and without common pairwise irreducible components.
 The ordered partition $Y_\bullet$ determines a decreasing filtration of $X$ (by subcurves)
 \begin{equation}\label{E:WY}
 W_\supseteq(Y_\bullet): \  X=W_0\supsetneq W_1\supsetneq \ldots \supsetneq W_q\supsetneq W_{q+1}=\emptyset, \text{ where } W_k:=\bigcup_{h\geq k}Y_h.
 \end{equation}
 And, conversely, any decreasing filtration of $X$
 $$
 W_\supseteq: X=W_0\supsetneq W_1\supsetneq W_k \supsetneq W_q\supsetneq W_{q+1}=\emptyset,
 $$
determines an ordered partition of $X$ by 
\begin{equation}\label{E:YW}
Y_\bullet(W_\supseteq)=(Y_0,\ldots, Y_q) \text{ where } Y_k:=W_k-W_{k+1}. 
\end{equation}

 Consider now a sheaf $I\in \TF_X(k)$. An ordered partition $Y_\bullet$ with associated decreasing filtration $W_\supseteq=W_\supseteq(I_\bullet)$ determines a decreasing  filtration of  $I$ by setting
\begin{equation}\label{E:Ileq}
I_\geq(Y_\bullet): I=I_0\supsetneq I_1\supsetneq \ldots \supsetneq I_{q}\supsetneq I_{q+1}=0, \text{ where } I_i:=\leftindex_{W_i}{I}=\leftindex_{W_i}{(\leftindex_{W_{i-1}}I)}=\ker(I_{i-1}\twoheadrightarrow (I_{i-1})_{Y_{i-1}}).
\end{equation}
We now set 
\begin{equation}\label{E:GrY}
\Gr_{Y_\bullet}(I):=\bigoplus_{i=0}^q I_{i}/I_{i+1}= \bigoplus_{i=0}^q (I_{i})_{Y_i}=\bigoplus_{i=0}^q (\leftindex_{W_i}{I})_{Y_i}.
\end{equation}
The sheaf $\Gr_{Y_\bullet}(I)$ belongs to $\TF_X(k)$ and Formula~\eqref{E:add-chi} implies that $\chi(\Gr_{Y_\bullet}(I))=\chi(I)$.

\begin{proposition}\label{P:iso-spec}
    A sheaf $I\in \TF_X(k)$ isotrivially specializes to $J\in \TF_X(k)$ if and only if $J=\Gr_{Y_\bullet}(I)$ for some ordered partition $Y_\bullet$ of $X$.     
\end{proposition}
\begin{proof}
By definition, a sheaf $I\in \TF_X(k)$ isotrivially specializes to $J\in \TF_X(k)$ if and only if there exists a morphism  $f:\Theta_k\to \TF_X$ such that $f(1)=I$ and $f(0)=J$. Arguing as in \cite[Lemma 1.10]{heinloth-HM},  such a morphism is equivalent to a decreasing filtration of $f(1)=I$
$$
I_\geq: I=I_{0}\supsetneq I_1\supsetneq \ldots \supsetneq I_q\supsetneq I_{q+1}=0
$$
such that 
$$f(0)=J=\Gr(I_\geq):=\bigoplus_{i=0}^q I_i/I_{i+1}.$$

We now claim that  a decreasing filtration $I_\geq$ of $I$ such that $\Gr(I_\geq)\in \TF_X(k)$ is equivalent to  an ordered partition $Y_{\bullet}=(Y_0, \ldots, Y_q)$ of $X$ and, under this correspondence, we have that $\Gr(I_\geq)=\Gr_{Y_\bullet}(I)$.

Indeed, we have seen that an ordered partition $Y_{\bullet}:=(Y_0, \ldots, Y_q)$ of $X$ defines a decreasing filtration of $I$:
$$
I_\geq(Y_\bullet):\  I=I_0\supsetneq I_1\supsetneq \ldots \supsetneq I_{q}\supsetneq I_{q+1}=0, \text{ where } I_i:=\ker(I_{i-1}\twoheadrightarrow (I_{i-1})_{Y_{i-1}}),
$$
such that 
$$
\Gr(I_\geq)=\Gr_{Y_\bullet}(I)\in \TF_X(k).
$$

Conversely, a decreasing  filtration $I_\geq$ of $I$ such that $\Gr(I_\geq)\in \TF_X(k)$ determines an ordered partition of $X$:
$$Y_{\bullet}(I_\geq):=(Y_0, \ldots, Y_q) \text{ where } Y_i:=\supp(I_i/I_{i+1}).$$
Note that this is an ordered partition of $X$ since $\Gr(I_\leq)$ is rank~$1$ on $X$.

We conclude by observing that these two associations are one the inverse of the other:

\begin{itemize}
    \item $Y_\bullet(I_\geq(Y_\bullet))=Y_\bullet$, since $I_{i}/I_{i+1}=(I_i)_{Y_i}$ is supported in $Y_i$ for any successive quotient $I_{i}/I_{i+1}$ of $I_\geq(Y_\bullet)$;
    \item $I_\geq(Y_\bullet(I_\geq))=I_\geq$ using that $I_i/I_{i+1}=(I_i)_{Y_i}$ since the former sheaf is supported on $Y_i$ (by definition of $Y_\bullet(I_\leq)$) and it is torsion-free by the assumption that $\Gr(I_\bullet)$ is torsion-free.  
\end{itemize}
%It remains to observe that if $I_\leq$ corresponds to $Y_\bullet$ under the above bijection, we have that $\Gr(I_\leq)=\Gr_{Y_\bullet}(I)$, which follows from the fact that  $I_i/I_{i+1}=(I_i)_{Y_i}$. 
\end{proof}

\section{V-stability conditions}\label{Sub:Vstab}

In this section, we define V-stability conditions for  connected reduced curves (and for their families).  

\begin{definition}\label{D:VStabX}
Let $X$ be a connected reduced curve. 
   A \emph{stability condition of vine type} (or simply a \textbf{V-stability condition)}    of characteristic $\chi\in \ZZ$ on  $X$ is a function
    \begin{align*}
        \s:\BCon(X)&\to \ZZ\\
        Y&\mapsto \s_Y
    \end{align*}
    satisfying the following properties:
\begin{enumerate}
\item \label{E:condi1} for any $Y\in \BCon(X)$, we have 
\begin{equation}\label{E:sum-n}
\mathfrak s_Y+\mathfrak s_{Y^\mathsf{c}}-\chi
\in \{0,1\}.
\end{equation}
%\begin{equation}\label{E:sum-n}
%\mathfrak n_Y+\mathfrak n_{Y^\mathsf{c}}-g(Y)-g(Y^c)+g(X)-d+1
%\in \{0,1\}.
%\end{equation}
A subcurve $Y\in \BCon(X)$ is said to be \emph{$\s$-degenerate} if $\s_Y+\s_{Y^c}-\chi=0$,
%$\mathfrak n_Y+\mathfrak n_{Y^\mathsf{c}}-g(Y)-g(Y^c)+g(X)-d+1=0$, 
and \emph{$\s$-nondegenerate} otherwise.

\item  \label{E:condi2} given subcurves $Y_1,Y_2,Y_3\in \BCon(X)$ without pairwise common irreducible components such that $X=Y_1\cup Y_2\cup Y_3$, we have that:
\begin{enumerate}
 \item if two among the subcurves $\{Y_1,Y_2,Y_3\}$ are $\s$-degenerate, then so is  the third. %is $\s$-degenerate;
            \item the following condition holds
            \begin{equation} \label{E:tria-n}
            \sum_{i=1}^{3}\s_{Y_i}-\chi
            %\sum_{i=1}^{3}\n_{Y_i}-\sum_{i=1}^3 g(Y_i)+g(X)-d+2
            \in \begin{cases}
                \{1,2\} \textup{ if $Y_i$ is $\s$-nondegenerate for all $i=1,2,3$};\\
                \{1\} \textup{ if there exists a unique   } i\in \{1,2,3\} \text{ such that $Y_i$ is $\s$-degenerate};\\
                \{0\} \textup{ if $Y_i$ is $\s$-degenerate for all $i=1,2,3$}.
            \end{cases}
        \end{equation}
\end{enumerate}
\end{enumerate}

The characteristic $\chi$ of $\s$ will also be denoted by $|\s|$. The \emph{degeneracy set} of $\s$ is the collection
\begin{equation*}
\D(\s):=\{Y\in \BCon(X): Y \text{ is $\s$-degenerate}\}.
\end{equation*}
    A V-stability condition $\s$ is called \emph{general}  if every $Y\in \BCon(X)$ is $\s$-nondegenerate, i.e. if $\D(\s)=\emptyset$.

The collection of all V-stability conditions of characteristic $\chi$ on $X$ is denoted by $\VStab^\chi(X)$ and the collection of all V-stability conditions on $X$ is denoted by 
$$\VStab(X)=\coprod_{\chi \in \ZZ}\VStab^\chi(X).$$
\end{definition}

Note that, if $X$ is a nodal curve over an algebraically closed field, then the notion of a general V-stability condition introduced here agrees with the notion of a general V-stability condition introduced in \cite[Definition~1.4]{viviani2023new}.

\begin{remark}\label{R:n-deg}
Let $\s$ be a V-stability condition on $X$.
\begin{enumerate}
%\item \label{R:n-deg0}
%Using \eqref{E:inter}, Conditions \eqref{E:condi1} and \eqref{E:condi2} can be restated using that 
%$$
%\begin{sis}
%&\mathfrak n_Y+\mathfrak n_{Y^\mathsf{c}}-g(Y)-g(Y^c)+g(X)-d+1= \mathfrak n_Y+\mathfrak n_{Y^\mathsf{c}}+|Y\cap Y^\mathsf{c}|-d,\\
%&\sum_{i=1}^{3}\n_{Y_i}-\sum_{i=1}^3 g(Y_i)+g(X)-d+2=\sum_{i=1}^{3}\n_{Y_i}-d+|Y_i\cap Y_j|+|(Y_i\cup Y_j)\cap Y_k|,
%\end{sis}
%$$
%for any $\{i,j,k\}=\{1,2,3\}$. If, moreover, the curve $X$ has locally planar singularities, then $|(Y_i\cup Y_j)\cap Y_k|=|Y_i\cap Y_k|+|Y_j\cap Y_k|$ and therefore we also have that
%\footnote{F: Questa formula e' falsa per curve ridotte qualsiasi}
%$$
%\sum_{i=1}^{3}\n_{Y_i}-\sum_{i=1}^3 g(Y_i)+g(X)-d+2=\sum_{i=1}^{3}\n_{Y_i}+\sum_{1\leq i<j\leq 3} |Y_i\cap Y_j|-d.
%$$
%\item \label{R:vgeneral} If $X$ is a nodal curve over an algebraically closed field, then the notion of a general V-stability condition introduced here agrees with the notion of a general V-stability condition introduced in \cite[Definition~1.4]{viviani2023new}. 

    \item \label{R:n-deg1}
   The degeneracy subset $\D(\s)$  satisfies the following two properties (which follow easily from Definition~\ref{D:VStabX}):
    \begin{itemize}
        \item for any $Y\in \BCon(X)$, we have that $Y\in \D(\s)$ if and only if $Y^\mathsf{c}\in \D(\s)$.
        \item for any $Y_1,Y_2\in \BCon(X)$ such that $Y_1\wedge Y_2=\emptyset$ and  $Y_1\cup Y_2\in \BCon(X)$, we have that if two  elements of $\{Y_1,Y_2,Y_1\cup Y_2\}$ belong to $\D(\s)$ then so does the third. 
    \end{itemize} 
   \item \label{R:n-deg2} If $Y_1,Y_2\in \BCon(X)$ are such that $Y_1\wedge Y_2=\emptyset$ and  $Y_1 \cup Y_2\in \BCon(X)$, then we have that 
   \begin{equation}\label{E:n-union}
    \s_{Y_1\cup Y_2}-\s_{Y_1}-\s_{Y_2}
   \in 
   \begin{cases}
        \{0\} & \text{ if $Y_1$  or  $Y_2$ is $\s$-degenerate};\\
        \{-1\} & \text{ if $Y_1$  and  $Y_2$ are $\s$-nondegenerate} \\
        & \text{ and $Y_1\cup Y_2$ is $\s$-degenerate};\\
      \{0,-1\} & \text{ if $Y_1, Y_2, Y_1\cup Y_2$ are $\s$-nondegenerate}.\\
   \end{cases}
   \end{equation}
   This follows by applying \eqref{E:sum-n} to $Y_1\cup Y_2$ and \eqref{E:tria-n} to $(Y_1,Y_2,(Y_1\cup Y_2)^\mathsf{c})$. 
   
   Conversely, Properties \eqref{E:sum-n} and \eqref{E:n-union} imply Property \eqref{E:tria-n}, as it follows by applying \eqref{E:n-union} to $(Y_1,Y_2)$ and by using \eqref{E:sum-n} for $(Y_1\cup Y_2,Y_3)$. 
   \end{enumerate}
\end{remark}

We now introduce the extended degeneracy set and the extended V-function of a V-stability condition.

\begin{definition}\label{D:ext-n}
    Let $\s$ be a V-stability condition on $X$.
    The \emph{extended degeneracy set} of $\s$ is 
       $$
        \wh \D(\s):=\left\{
    \begin{aligned}
        & W\in \Con(X)\: :  W^\mathsf{c}=Z_1\coprod\ldots \coprod Z_k \\
        &    \text{ with } Z_i \in \D(\s) \text{ for all } i=1, \ldots, k.
     \end{aligned}   
        \right\} \bigcup\, \{X\}. $$
        The elements $W$ of $\D(\s)$ are called $\s$-degenerate subcurves.
\end{definition}
Note that the extended degeneracy set is an extension of the degeneracy set since 
$$\D(\s)=\wh \D(\s)\cap \BCon(X).$$

\begin{definition}\label{D:ext-s}
Let $\s$ be a V-stability condition on $X$. We extend $\s$ to a function denoted with the same symbol $ \s: \Sub(X) \to \ZZ$, and called the \emph{extended V-function} (of the original $\s$), as follows.

If $Y$ is connected and $Y^\mathsf{c}=\bigsqcup_{i\in I}Z_i$ is the decomposition into connected components of its complement (so $Z_i \in \BCon(X)$ for all $i \in I$), we  set
\begin{equation} \label{eq: ext-conn}
 \s_Y:= |\s|-\sum_{i\in I}\s_{Z_i}+|\{Z_i : Z_i\not \in \D(\s)\}|.
\end{equation}
If $Y$ is not connected and $Y= \bigsqcup_{j \in J} W_j$ is its decomposition into connected components, we set
$$
s_Y:= \sum_{j \in J} \s_{W_j},
$$
where the summands on the right hand side have been defined in \eqref{eq: ext-conn}.
\end{definition}
Note that the above definition is indeed an extension of Definition~\ref{D:VStabX}: if $Y \in \BCon(X)$, then the right hand side of \eqref{eq: ext-conn} agrees with the original value of $\s_Y$  thanks to Equation~\eqref{E:sum-n}. 

\begin{remark}\label{R:ext-s}
Let $\s$ be a V-stability condition of characteristic $\chi$ on $X$. It follows from Definition \ref{D:ext-s}, that if  $Y\in \Con(X)$ is such that $Y^\mathsf{c}=\bigsqcup_{i\in I}Z_i$ with $Z_i\in \Con(X)$, then we have
        $$
        {\s}_Y+{\s}_{Y^\mathsf{c}}=\chi+|\{Z_i\notin\D(\s)\}|.
        $$
In particular, for any $W\in \wh \D(\s)$ we have that 
$$ \s_W=
    \begin{cases}
     |\s| & \text{ if } W=X,\\
     |\s|-\sum_{i=1}^k \s_{Z_i}  
     &\text{ if }  
        W^\mathsf{c}=Z_1\coprod\ldots \coprod Z_k 
   \text{ with } Z_i\in  \D(\s). \\
    \end{cases}
    $$
\end{remark}

\begin{lemma}\label{L:ExtV-funP}
        Let  $Y,Z\in\Con(X)$ be such that $Y\wedge Z=\emptyset$ and $Y\cup Z\in\Con(X)$. Write the  decomposition of $(Y\cup Z)^\mathsf{c}$ into connected components as
    $$
    (Y\cup Z)^\mathsf{c}=\bigsqcup_{i\in I}A_i\sqcup\bigsqcup_{j\in J}B_j\sqcup\bigsqcup_{k\in K}C_k,
    $$
    where $\{A_i\}$ are the connected components that intersect $Y$ but not $Z$, $\{B_j\}$ are the connected components that intersect both $Y$ and $Z$, and $\{C_k\}$ are the connected components that intersect $Z$ but not $Y$.
Set $Y':=Y \cup \bigcup_{i\in I} A_i\in \BCon(X)$ and $Z':=Z \cup \bigcup_{k\in K} C_k\in \BCon(X)$. For any  V-stability condition $\s$ on $X$, we have two cases:
\begin{enumerate}[(a)]
    \item \label{L:ExtV-funPA} If $Y', Z', B_j\in \D(\s)$ then 
    $$\s_{Y\cup Z}-\s_Y-\s_Z=0.$$
%    In particular, this holds if $Y,Z,Y\cup Z\in \wh \D(\s)$. 
    \item \label{L:ExtV-funPB} Otherwise, we have that 
     $$1-|\{Y', Z'\notin \D(\s)\}|\leq \s_{Y\cup Z}-\s_Y-\s_Z\leq |\{B_j\notin \D(\s)\}|-1$$
\end{enumerate}
 %   \end{enumerate}
\end{lemma}
The content of the above Lemma, if we further assume $Y,Z,Y\cup Z\in \BCon(X)$, follows immediately from Remark \ref{R:n-deg}\eqref{R:n-deg2}. 

\begin{proof}
For any subcurve $W\in\BCon(X)$, we define the integer $$
    \delta_W:=\begin{cases}
        1 \ \ \textup{ if } W\notin\D(\s),\\
        0 \ \ \textup{ otherwise}.
    \end{cases}
    $$
    Remark \ref{R:n-deg}\eqref{R:n-deg1} implies that $\delta_W=\delta_{W^c}$.
    
Observe that the decompositions into connected components of $Y^\mathsf{c}$ and $Z^\mathsf{c}$ are respectively
    $$
    Y^\mathsf{c}=\bigsqcup_{i\in I}A_i\sqcup\biggl(Z'\cup\bigcup_{j\in J}B_j\biggr)=\bigsqcup_{i\in I}A_i\sqcup (Y')^c \quad\textup{and}\quad Z^\mathsf{c}=\bigsqcup_{k\in K}C_k\sqcup\biggl(\ov Y\cup\bigcup_{j\in J}B_j\biggr)=\bigsqcup_{k\in K}C_k\sqcup (Z')^c.
    $$
    %For any subcurve $W\in\BCon(X)$, we define the integer $$
    %\delta_W:=\begin{cases}
    %    1,\textup{ if }W\notin\D(\s),\\0,\textup{ otherwise}.
    %\end{cases}
    %$$
    Hence, by Definition~\ref{D:ext-s} and Equation \ref{E:sum-n}, we have
    \begin{equation}\label{E:diff-s}
    \begin{aligned}
        & \s_{Y\cup Z}-\s_Y-\s_Z=\chi-\sum_{j\in J}\s_{B_j}+\sum_{j\in J}\delta_{B_j}-\chi+\s_{(Y')^c}-\delta_{Y'} -\chi+\s_{(Z')^c} -\delta_{Z'}=\\
        & =\chi-\sum_{j\in J}\s_{B_j}+\sum_{j\in J}\delta_{B_j}-\s_{Y'}-\s_{Z'}. 
        \end{aligned}
    \end{equation}
Observe that if $|J|=0$ (i.e. $X=Y'\cup Z'$), then the statement follows from \eqref{E:sum-n} and \eqref{E:diff-s}. Hence we can assume that $|J|\geq 1$.
    
We now proceed by induction on $n:=|J|\geq 1$. Write $J=\{1,\ldots, n\}$.

If $|J|=n=1$ (i.e. $X=Y'\cup Z'\cup B_1$) then the statement follows from 
   \eqref{E:tria-n} and \eqref{E:diff-s}.

Assume now $n\geq 2$ and we will use the induction hypothesis, so  the statement is valid for $Y'':=Y'\cup B_1\in \BCon(X)$ and $Z'\in \BCon(X)$ since 
$$
(Y''\cup Z')^c=\coprod_{2\leq j \leq n} B_j.$$
Equation \eqref{E:diff-s} implies that 
\begin{equation}\label{E:diff-s2}
     \s_{Y\cup Z}-\s_Y-\s_Z=\s_{Y'\cup Z'}-\s_{Y'}-\s_{Z'}=[\s_{Y''\cup Z'}-\s_{Y''}-\s_{Z'}]+[\s_{Y''}-\s_{Y'}-\s_{B_1}+\delta_{B_1}].
    \end{equation}
We now distinguish several cases:
\begin{enumerate}
    \item Assume that $\{Y'',Z',B_j: j\geq 2\}\subseteq \D(\s)$. Then, by induction we get 
\begin{equation}\label{E:eqa1}
    \s_{Y''\cup Z'}-\s_{Y''}-\s_{Z'}=0.
\end{equation}
Using Remark \ref{R:n-deg}\eqref{R:n-deg1}, we can distinguish two subcases:
\begin{enumerate}[(i)]
    \item If $Y',B_1\in \D(\s)$ (i.e. we are in Case \eqref{L:ExtV-funPA}) then Remark \ref{R:n-deg}\eqref{R:n-deg2} implies that 
\begin{equation}\label{E:eqa2}
\s_{Y''}-\s_{Y'}-\s_{B_1}+\delta_{B_1}=0+0=0.    
\end{equation}
We conclude by putting together \eqref{E:diff-s2}, \eqref{E:eqa1} and \eqref{E:eqa2}. 
\item If $Y',B_1\not \in \D(\s)$  then Remark \ref{R:n-deg}\eqref{R:n-deg2} implies that 
\begin{equation}\label{E:eqa3}
\s_{Y''}-\s_{Y'}-\s_{B_1}+\delta_{B_1}=-1+1=0.    
\end{equation}
By putting together \eqref{E:diff-s2}, \eqref{E:eqa1} and \eqref{E:eqa3} we deduce that $ \s_{Y\cup Z}-\s_Y-\s_Z=0$ which shows that \eqref{L:ExtV-funPB} holds in this case since the left and right hand of Case~\eqref{L:ExtV-funPB} are $0$. 
\end{enumerate} 
\item Assume that $\{Y'',Z',B_j: j\geq 2\}\not\subseteq \D(\s)$. Then, by induction we get 
\begin{equation}\label{E:eqa4}
    1-\delta_{Y''}-\delta_{Z'}\leq \s_{Y''\cup Z'}-\s_{Y''}-\s_{Z'}\leq \sum_{j\geq 2} \delta_{B_j}-1.
\end{equation}
We now distinguish two subcases:
\begin{enumerate}[(i)]
    \item If either $B_1\in \D(\s)$ or $Y'\in \D(\s)$ then Remark \ref{R:n-deg}\eqref{R:n-deg2} implies that 
\begin{equation}\label{E:eqa5}
\s_{Y''}-\s_{Y'}-\s_{B_1}+\delta_{B_1}=\delta_{B_1}.
\end{equation}
By putting together \eqref{E:diff-s2}, \eqref{E:eqa4} and \eqref{E:eqa5} we get the desired inequalities
\begin{equation*}\label{E:eqa6}
   1-\delta_{Y'}-\delta_{Z'}\leq 1-\delta_{Y''}-\delta_{Z'}+\delta_{B_1} \leq \s_{Y\cup Z}-\s_Y-\s_Z\leq \sum_j \delta_{B_j}-1,
\end{equation*}
where in the first inequality we have used that $\delta_{Y''}\leq \delta_{Y'}+\delta_{B_1}$ which follows from the fact that if $Y', B_1\in \D(\s)$ then also $Y''\in \D(\s)$ by Remark \ref{R:n-deg}\eqref{R:n-deg1}. 
 \item If $B_1\not\in \D(\s)$ and $Y'\not\in \D(\s)$ then Remark \ref{R:n-deg}\eqref{R:n-deg2} implies that 
\begin{equation}\label{E:eqa7}
-1+\delta_{B_1}\leq \s_{Y''}-\s_{Y'}-\s_{B_1}+\delta_{B_1}\leq \delta_{B_1}.
\end{equation}
By putting together \eqref{E:diff-s2}, \eqref{E:eqa4} and \eqref{E:eqa7} we get the desired inequalities
\begin{equation*}\label{E:eqa8}
   1-\delta_{Y'}-\delta_{Z'}\leq 1-\delta_{Y''}-\delta_{Z'}-1+\delta_{B_1} \leq \s_{Y\cup Z}-\s_Y-\s_Z\leq \sum_j \delta_{B_j}-1,
\end{equation*}
where in the first inequality we have used that $1+\delta_{Y''}\leq 2=\delta_{Y'}+\delta_{B_1}$.
\end{enumerate}
\end{enumerate}

\end{proof}

\begin{corollary}\label{C:add-whn} 
Let $W_1,W_2\in \wh \D(\s)$ such that $W_1\wedge W_2=\emptyset$ and $W_1\cup W_2\in \wh \D(\s)$.
 Then 
  $$
 \s_{W_1\cup W_2}= \s_{W_1}+\s_{W_2}.$$
\end{corollary}
Observe that if $W_1,W_2\in \D(\s)$ with $W_1\cup W_2=X$ then the above Corollary follows from \eqref{E:sum-n}, while if $W_1,W_2,W_1\cup W_2\in \D(\s)$ then the above Corollary follows from \eqref{E:tria-n} (see Remark \ref{R:n-deg}\eqref{R:n-deg2}). 
\begin{proof}
If $W_1\cup W_2=X$ then the statement follows from Remark \ref{R:ext-s}. Instead, if  $W_1\cup W_2\neq X$ then the statement follows from Lemma  \ref{L:ExtV-funP}\eqref{L:ExtV-funPA} using that if $Y\in \wh \D(\s)$ (resp. $Z\in \wh \D(\s)$) then $Y'\in \D(\s)$ (resp. $Z'\in \D(\s)$) and if $Y\cup Z\in \wh \D(\s)$ then $B_j\in \D(\s)$ for any $j\in J$.
\end{proof}

Next, we show that a V-stability $\s$ on $X$  induces a natural V-stability on each connected $\s$-degenerate subcurve of $X$ (as in Definition \ref{D:ext-n}), using the extended V-function of Definition \ref{D:ext-s}.

\begin{lemma-definition}\label{LD:V-subgr}
Let $\s$ be a V-stability condition on  $X$ and let  $Y\in \wh \D(\s)$. 
Then the \emph{restriction of $\s$ to $Y$} is the V-stability condition $\s(Y)$ on $Y$  defined by 
$$
\s(Y)_W:=\s_W \text{ for any } W\in \BCon(Y).
$$
%VECCHIA DEFINIZIONE
%as follows: consider the decomposition into connected components $Y^\mathsf{c}=\coprod_{i=1}^k Z_i$ and for any $W\in \BCon(Y)$ pick a subset $I\subseteq \{1,\ldots, k\}$ such that $Z_I\bigcup W:=\coprod_{i\in I} Z_i\bigcup W\in \BCon(X)$ and define
%\begin{equation}\label{E:nG}
%\s(Y)_W:=\s_{Z_I\bigcup W}-\sum_{i\in I} \s_{Z_i}.
%\end{equation}
We have that 
\begin{itemize}
    \item $|\s(Y)|=\s_{Y}$;
   % \item $\D(\s(Y))=\{ W\in \BCon(Y)\: : \: Z_I\bigcup W \text{ is $\s$-degenerate}\}$, where the above condition is independent of the chosen $I$ such that $Z_I\bigcup W\in \BCon(X)$. 
    \item $\wh \D(\s(Y))=\wh \D(\s)\cap \Con(Y)$.
\end{itemize} 
\end{lemma-definition}
\begin{proof}
We have to show that $\s(Y)$ is a V-stability condition of $Y$ of characteristic $\s_{Y}$, i.e. that it satisfies Conditions~\eqref{E:sum-n} and \eqref{E:tria-n}, and that for any $W\in \Con(Y)\subseteq \Con(X)$
\begin{equation}\label{E:degnG}
W \text{ is $\s(Y)$-degenerate} \Leftrightarrow  W \text{ is $\s$-degenerate}.
\end{equation}
Clearly, we can assume that $Y\neq X$, otherwise the result is trivial.

Let $W\in \BCon(Y)$ and set $W^\mathsf{c}:=\ov{Y\setminus W}\in \BCon(Y)$.
Since the complement of $Y=W\cup W^c$ is a disjoint union of connected subcurves that are $\s$-degenerate (because $Y\in \wh \D(\s)$), Lemma \ref{L:ExtV-funP} implies that 
\begin{equation}\label{E:cond1-sY}
\s_Y-\s(Y)_W-\s(Y)_{W^c}=\s_Y-\s_W-\s_{W^c}=
\begin{cases}
0 & \text{ if } W,W^c\in \wh \D(\s),\\
-1 & \text{ if } W,W^c\not \in \wh \D(\s),
\end{cases}
\end{equation}
where we have used that 
$$W\in \wh \D(\s)\Leftrightarrow W'\in \D(\s)\Leftrightarrow (W^c)'\in \D(\s)\Leftrightarrow W^c\in \wh \D(\s). $$
Equation \eqref{E:cond1-sY} shows that Conditions~\eqref{E:sum-n} is satisfied by $\s(Y)$ and that \eqref{E:degnG} holds true.  

In order to show that Condition \eqref{E:tria-n} holds true for $\s(Y)$, it is enough to check that Condition \eqref{E:n-union} holds true for $\s(Y)$ (see Remark \ref{R:n-deg}\eqref{R:n-deg2}). Fix $W_1,W_2\in \BCon(Y)$ such that $W_1\wedge W_2=\emptyset$ and $W_1\cup W_2\in \BCon(Y)$, and set $W_3:=\ov{Y\setminus (W_1\cup W_2)}\in \BCon(Y)$.
We are going to apply Lemma \ref{L:ExtV-funP} to $W_1,W_2\in \Con(X)$, of which we will also adopt the notation. Since $Y=W_1\cup W_2\cup W_3\in \wh \D(\s)$, exactly one of the subcurves $B_j\in \BCon(X)$, say $B_1$, will contain $W_3$ while all the other subcurves $\{A_i: i \in I\}$, $\{B_j: j\geq 2\}$ and $\{C_k: k\in K\}$ are $\s$-degenerate. This also implies that 
$$
\begin{sis}
& W_1\in \D(\s(Y))\Leftrightarrow W_1'\in \D(\s), \\   
& W_2\in \D(\s(Y))\Leftrightarrow W_2'\in \D(\s), \\
& W_1\cup W_2\in \D(\s(Y))\Leftrightarrow W_3\in \D(\s(Y)) \Leftrightarrow B_1\in \D(\s).
\end{sis}
$$
Using the above equivalences, Lemma \ref{L:ExtV-funP} implies that we have two cases:
\begin{enumerate}[(a)]
    \item If $W_1,W_2,W_1\cup W_2\in \D(\s(Y))$ then we have that 
    $$\s(Y)_{W_1\cup W_2}-\s(Y)_{W_1}-\s(Y)_{W_2}=0.$$
     \item Otherwise, we have that 
    $$1-|\{W_1,W_2\in \D(\s(Y)) \}|\leq \s(Y)_{W_1\cup W_2}-\s(Y)_{W_1}-\s(Y)_{W_2}\leq |\{W_1\cup W_2\in \D(\s)\}|-1.$$
\end{enumerate}
This shows that $\s(Y)$ satisfies Condition \eqref{E:n-union} with respect to $W_1$ and $W_2$, and we are done.
\end{proof}

\subsection{Numerical polarizations}\label{Sub:numpol}

The easiest way of producing V-stability conditions is via numerical polarizations, as we now explain.

\begin{lemma-definition}\label{LD:numpol}
Let $X$ be a connected reduced curve. A \emph{numerical polarization} on $X$ of characteristic $\chi\in \ZZ$ is a function  
$$
\begin{aligned}
\psi:\Sub(X) & \longrightarrow \RR\\
Y & \mapsto \psi_Y
\end{aligned}
$$
that is \emph{additive}, i.e. if $Y_1,Y_2$ are subcurves of $X$ such that $Y_1\wedge Y_2=\emptyset$ then $\psi_{Y_1\cup Y_2}=\psi_{Y_1}+\psi_{Y_2}$, and such  that $|\psi|:=\psi_X=\chi$. The function 
\begin{align*}
        \left\lceil - \right\rceil :\BCon(X)&\to \ZZ\\
        Y&\mapsto \s(\psi)_Y:=\left\lceil\psi_Y\right\rceil.
    \end{align*}
is a V-stability condition on $X$ of characteristic $\chi$ (called the V-stability condition associated to $\psi$), such that 
$$
\D(\s(\psi))=\{Y\in \BCon(X): \: \psi_Y\in \ZZ\}.
$$
\end{lemma-definition}
The V-stabilities of the form $\s(\psi)$ are called \emph{classical}.
\begin{proof}
 This follows  by taking the upper integral parts in the following two equalities
$$
\begin{sis}
& \psi_Y+\psi_{Y^c}-\chi=0 \text{ for any } Y\in \BCon(X),\\    
& \sum_{i=1}^3 \psi_{Y_i}-\chi=0 \text{ for any } \{Y_1,Y_2,Y_3\} \text{ as in } \eqref{E:condi2}. 
\end{sis}
$$
\end{proof}

\begin{remark}
Let $\s=\s(\psi)$ for some numerical polarization $\psi$ on $X$.
\begin{enumerate}
    \item From Definition \ref{D:ext-s}, it follows that the extended V-function of $\s(\psi)$ satisfies the inequality
 $$
 \s(\psi)_Y\geq \lceil \psi_Y\rceil \text{ for any } Y\in \Sub(X).
 $$
 However, the inequality may be strict for non-biconnected subcurves.
  \item  A subcurve $Y\in \Con(X)$ belongs to $\wh \D(\s(\psi))$ if and only if for the decomposition into connected components $Y^c=Z_1\coprod \ldots \coprod Z_k$ we have that $\psi_{Z_i}\in \ZZ$ (which then implies that $\psi_Y\in \ZZ$). Furthermore, we have that $\s(\psi)_Y=\psi_Y$ for any $Y\in \wh \D(\s(\psi))$.
 \item  \label{R:restr-clas}
 If  $Y\in \wh \D(\s(\psi))$ then 
 $$\s(\psi)(Y)_W=\psi_W \text{ for any } W\in \BCon(Y).$$
 Note that $|\s(\psi)(Y)|=\psi_Y$.
\end{enumerate}

\end{remark}

Because of the additivity property, a numerical polarization is completely determined by its values  on the irreducible components of $X$. Hence, the space of numerical polarizations on $X$ of characteristic $\chi\in \ZZ$, denoted by $\Pol^{\chi}(X)$, is a real affine subspace of $\RR^{I(X)}$ (with $I(X)$ the set of irreducible components of $X$) whose underlying real vector space is $\Pol^0(X)$. The space of all numerical polarizations on $X$ is
$$
\Pol(X)=\coprod_{\chi \in \ZZ}\Pol^\chi(X)\subset \RR^{I(X)}.
$$
Consider the arrangement of hyperplanes in the affine space $\RR^{I(X)}$ given by  
%(see \cite[Sec. 7]{Oda1979CompactificationsOT} and \cite[Sec. 3]{MRV}):
\begin{equation}\label{E:arr-hyper}
\A_{X}:=\left\{\psi_Y=n\right\}_{Y\in \BCon(X), n\in \ZZ}.
\end{equation}
This arrangement yields an induced wall and chamber decomposition on $\Pol(X)$ such that two numerical polarizations $\psi, \psi'$ belong to the same region, i.e. they have the same relative positions with respect to all the hyperplanes, if and only if $\s(\psi)=\s(\psi')$. In other words, the map
 \begin{equation}\label{E:map-s}
 \begin{aligned}
     \left\lceil - \right\rceil: \Pol(X) & \longrightarrow \VStab(X)\\
    \psi & \mapsto \s(\psi)
 \end{aligned}
 \end{equation}
 induces a bijection between the set of regions induced by  $\A_{X}$ on $\Pol(X)$ and the set of classical V-stability conditions on $X$. Note also that $\s(\psi)$ is a general V-stability condition if and only if $\psi$ belongs to a chamber (i.e. a maximal dimensional region), or equivalently if it does not lie on any of the hyperplanes of $\A_X$, in which case we say that $\psi$ is a \emph{general} polarization.

\vspace{0.1cm}

\subsection{The poset of V-stability conditions}\label{Sub:posetV}

We now focus on the combinatorial properties of the set of V-stability conditions. We first define a poset structure on the set of V-stability conditions on $X$. 

\begin{definition}\label{D:VStab-pos}
  Let $X$ be connected reduced curve. The set  $\VStab(X)$ of V-stability conditions on $X$ is endowed with the following order relation 
  $$
  \s\geq \t \Longleftrightarrow 
  \begin{sis}
  &|\s|=|\t|,\\
  & \s_Y\geq \t_Y \text{ for all } Y\in \BCon(X).\\
  \end{sis}
  $$
\end{definition}
Note that each $\VStab^{\chi}(X)$ is a union of connected components of the poset $\VStab(X)$.

\begin{remark}\label{R:VStab-pos}
    Let $\s,\t\in \VStab(X)$. If $\s\geq \t$ then from \eqref{E:sum-n} we have that
    \begin{itemize}
    \item $\D(\s)\subseteq \D(\t)$ and for any $Y\in \BCon(X)$ we have:
$$
(\s_Y,\s_{Y^c})=
   \begin{cases}
   (\t_Y,\t_{Y^c}) & \text{ if either } Y\in \D(\s) \text{ or } Y\not \in \D(\t), \\
   (\t_Y+1,\t_{Y^c}) \text{ or } (\t_Y,\t_{Y^c}+1) & \text{ if } Y\in \D(\t)- \D(\s).
   \end{cases}
$$
    \item $\wh \D(\s)\subseteq \wh \D(\t)$ and $\wh \t_{|\wh \D(\s)}=\wh s$. 
   \end{itemize}
%   In particular, if $\s\in \VStab^{\chi}(X)$ is general, then it is maximal in $\VStab^{\chi}(X)$.
\end{remark}

The poset $\VStab(X)$ comes with a map to the set of degeneracy subsets, as we now formalize.

\begin{definition}\label{D:deg-set}
    Let $X$ be connected reduced curve. A \emph{degeneracy subset} of $X$ is a subset $\D\subseteq \BCon(X)$ such that:
     \begin{enumerate}[(i)]
         \item \label{D:deg-set1} If $Y\in \D$ then $Y^c\in \D$.
         \item \label{D:deg-set2} If $Y_1,Y_2\in \D$ with $Y_1\wedge Y_2=\emptyset$ and $Y_1\cup Y_2\in \BCon(X)$, then $Y_1\cup Y_2\in \D$.
     \end{enumerate}
The collection of all possible degeneracy subsets of $X$ is denoted by $\Deg(X)$ and we have a map (by Remark \ref{R:n-deg}\eqref{R:n-deg1}), called the \emph{degeneracy map} 
\begin{equation}\label{E:map-D}
   \begin{aligned}
    \D:\VStab(X)& \longrightarrow \Deg(X) \\
    \s & \mapsto \D(\s).
    \end{aligned}
\end{equation}
\end{definition}

\begin{remark}\label{R:min-D}
Let $\D\in \Deg(X)$. Properties \eqref{D:deg-set1} and \eqref{D:deg-set2} imply that 
\begin{enumerate}
    \item[(iii)]  If $W_1,W_2\in \D$ with $W_1\subsetneq W_2$ and $W_2-W_1\in \BCon(X)$, then $W_2-W_1\in \D$.
\end{enumerate}
Indeed, $W_2-W_1=(W_1\cup W_2^c)^c$ which belongs to $\D$ by applying \eqref{D:deg-set1} and \eqref{D:deg-set2}.

The above property implies that, if we define the set of minimal elements of $\D$ by
$$\D_{\min}:=\{Y\in \D: \text{there exists no $W\in \D$ such that $W\subset Y$ and $Y-W\in \BCon(X)$}\},$$
then any element of $\D$ can be written uniquely as a union of minimal elements of $\D$ with pairwise empty meet.  
\end{remark}

 We now define a translation action on the set of V-stability conditions. 

\begin{definition}
\label{D:tranVStab}
\noindent 
\begin{enumerate}
    \item The \emph{translation} of $\s\in \VStab(X)$ by a function $\tau:I(X)\to \ZZ$ is the V-stability condition $\s+\tau$ on $X$ defined by 
 $$
    \begin{aligned}
     \s+\tau:\BCon(X)& \longrightarrow  \ZZ, \\
     Y & \mapsto  \s_Y+\tau_Y:=\s_Y+\sum_{X_v\subseteq Y} \tau_{v},
    \end{aligned}
    $$
This defines an action of $\ZZ^{I(X)}$ on $\VStab(X)$ with the property that  $|\s+\tau|=|\s|+|\tau|:=|\s|+\tau_X$.

    \item  Two V-stability conditions $\s$ and $\s'$ on $X$ are \emph{equivalent by translation} if there exists a function $\tau:I(X)\to \ZZ$ such that  $\s'=\s+\tau$. 
\end{enumerate}
\end{definition}

\begin{proposition}\label{P:fin-tran}
The action of $\ZZ^{I(X)}$ on $\VStab(X)$ is such that 
\begin{enumerate}[(i)]
    \item \label{P:fin-tran1} the degeneracy map $\D$ of \eqref{E:map-D} is invariant under the action of $\ZZ^{I(X)}$ on the domain;
    \item \label{P:fin-tran2} the action preserves the poset structure, i.e. if $\s\geq t$ then $\s+\tau\geq \s+\tau$ for any $\tau \in \ZZ^{I(X)}$;
    \item \label{P:fin-tran3} there are finitely many orbits.
\end{enumerate}
\end{proposition}
    \begin{proof}
    Parts \eqref{P:fin-tran1} and \eqref{P:fin-tran2} are obvious.
    
    In order to prove Part \eqref{P:fin-tran3}, consider the incidence graph $\Gamma$ of $X$, i.e. the simple graph whose vertex set is the set $I(X)$ of irreducible components of $X$ and such that there is an edge between $v$ and $w$ if and only if the corresponding irreducible components $X_v$ and $X_w$ have a nonempty intersection. (Note that $\Gamma$ is connected since so is $X$). We fix a spanning tree $T$ of $\Gamma$. 
    
    Each subcurve $Y\in \Sub(X)$ determines, respectively, an induced subgraph $\Gamma[Y]$ of $\Gamma$ and an induced subgraph $T[Y]$ of $T$, whose vertices correspond to the irreducible components of $Y$. For any $Y\in \Sub(X)$, denote by $\val_T(Y)$ the number of edges of $T$ that join a vertex of $T[Y]$ with a vertex of $T[Y^c]$. 

    Note that a subcurve $Y$ of $X$ is biconnected if and only if $\Gamma[Y]$ and $\Gamma[Y^c]$ are connected, which is implied by the stronger condition that $T[Y]$ and $T[Y^c]$ are connected, which is in turn equivalent to the condition that  $\val_T(Y)=1$. 
    
   \un{Claim 1:} Every V-stability condition on $X$  is equivalent by translation to a V-stability condition $\s\in \VStab(X)$ such that 
   \begin{equation}\label{E:equitrasl}
   |\s|=0 \quad \text{ and } \quad  0\leq \s_Y \leq 1 \quad \text{ for any } Y\in \BCon(X) \text{ such that } \val_T(Y)=1.
   \end{equation}

   Indeed, let $\t$ be any V-stability condition on $X$. Fix an orientation of $T$. For any $e\in E(T)$, we get a subcurve $Y_e\in \Sub(X)$ such that $T\setminus e=T[Y_e]\coprod T[Y_e^c]$ and that the edge $e$, with the chosen orientation, goes from $T[Y_e]$ to $T[Y_e^c]$. In particular, $\val_T(Y_e)=\val_T(Y_e^c)=1$ (which implies that $Y_e, Y_e^c\in \BCon(X)$) and all the subcurves of $X$ with this property are of the form $Y_e$ or $Y_e^c$ for some $e\in E(T)$. Consider now the following function
   $$
   \begin{aligned}
      f: \left\{Y\in \BCon(X): \: \val_T(Y)=1\right\} & \longrightarrow \NN\\
      Y_e & \mapsto -\t_{Y_e}\\
      Y_e^c & \mapsto 
      \begin{cases}
         -\t_{Y_e^c} & \text{ if } Y_e\in \D(\t),\\
         -\t_{Y_e^c}+1 & \text{ if } Y_e\not \in \D(\t).\\
      \end{cases}
   \end{aligned}
   $$
   By \eqref{E:sum-n}, we have that $f(Y_e)+f(Y_e^c)=-|\t|$ for any $e\in E(T)$. Hence, since $T$ is a tree, there exists a unique function 
   $\tau:I(X)\to \NN$ such that 
   $$\tau_{Y_e}=f(Y_e) \text{ and } \tau_{Y_e^c}=f(Y_e^c).$$
   If we set $\s:=\t+\tau$, then by construction we have that, for any $e\in E(T)$,
   $$\s_{Y_e}=0 \text{ and } \s_{Y_e^c}=
      \begin{cases}
         0 & \text{ if } Y_e\in \D(\t)=\D(\s),\\
         1 & \text{ if } Y_e\not \in \D(\t)=\D(\s).\\
      \end{cases}
   $$
   In particular, $\s$ satisfies \eqref{E:equitrasl} and we are done. 

   \un{Claim 2:} If $\s$ is a V-stability condition on $X$ that satisfies \eqref{E:equitrasl}, then we have that  
   \begin{equation}\label{E:norm-s}
-\val_T(Y)+1\leq  \s_Y\leq \val_T(Y) \text{ for any } Y\in \BCon(X).
\end{equation}

Indeed, we argue by induction on $\val_T(Y)$. If $\val_T(Y)=1$ then \eqref{E:norm-s} reduces to \eqref{E:equitrasl}. Assume that $\val_T(Y)>1$, or equivalently that either $T[Y]$ or $T[Y^c]$ is not connected. Upon possibly swapping $Y$ and $Y^\mathsf{c}$,  we can assume that $T[Y]$ is not connected. We write $T[Y]=T[W_1]\coprod \ldots \coprod T[W_k]$ for its decomposition into connected components (with $k\geq 2$).  Then we have that 
$$\val_T(Y)=\sum_{i=1}^k \val_T(W_i)\quad  \text{ and } \quad \val_T(W_i)\geq 1 \: \text{ for each } 1\leq i \leq k.$$
Since $Y$ is connected, there exists a non-trivial decomposition $Y=Z_1\bigcup Z_2$, where $Z_1$ and $Z_1$ are connected subcurves of $X$ that are unions of some of the subcurves $W_i$ and such that $Z_1\wedge Z_2=\emptyset$. We have that 
\begin{equation}\label{E:sum-val}
\val_T(Y)=\val_T(Z_1)+\val_T(Z_2) \quad  \text{ and } \quad \val_T(Z_i)\geq 1 \text{ for each } i=1,2.
\end{equation}
Moreover, $Z_1^\mathsf{c}=Z_2\cup Y^\mathsf{c}$ is a connected subcurve of $X$ since $Y^\mathsf{c}$ is connected, $Z_2$ is connected and there exists at least one edge of $T$ joining $T[Z_2]$ and $T[Y^\mathsf{c}]$. Similarly, also $Z_2^\mathsf{c}$ is a connected subcurve of $X$. By putting everything together, we get that $Z_1$ and $Z_2$ are biconnected subcurves of  $X$ such that $\val_T(Z_1),\val_T(Z_2)<\val_T(Y)$. Hence, we can apply our inductive hypothesis and deduce that 
$$
-\val_T(Z_i)+1\leq  \s_{Z_i}\leq \val_T(Z_i) \quad \text{ for } i=1,2.
$$
By applying \eqref{E:n-union} to $Z_1$, $Z_2$ and $Y=Z_1\cup Z_2$, we get
\begin{equation}\label{E:n-e-sum}
\s_{Y}-\s_{Z_1}-\s_{Z_2}\in \{-1,0\}
\end{equation}
Therefore, the inequality \eqref{E:norm-s} for $Y$ follows from the analogous inequalities for $Z_1$ and $Z_2$ (which hold true by induction) together with \eqref{E:n-e-sum} and \eqref{E:sum-val}.

\vspace{0.1cm}

    The proof of Part \eqref{P:fin-tran3} follows now by putting together Claim 1 and Claim 2, and using that there are a finite number of V-stability conditions that satisfy \eqref{E:norm-s}.
    \end{proof}

\begin{remark}\label{R:tran-pol}
We can also define the translation of a numerical polarization (Lemma-Definition~\ref{LD:numpol}) $\psi$ on $X$ by a function $\tau\in \ZZ^{I(X)}$ as the numerical polarization $\psi+\tau$ defined by 
 $$
    \begin{aligned}
     \psi+\tau:\left\{\text{Subcurves of } X \right\}& \longrightarrow  \RR, \\
     Y & \mapsto  \psi_Y+\tau_Y.
    \end{aligned}
    $$
    This defines an action of $\ZZ^{I(X)}$ on $\Pol(X)$ such that the map $ \left\lceil - \right\rceil $ of \eqref{E:map-s} is equivariant. We immediately deduce that the property of being classical for a V-stability is invariant under translation and that there are a finite number of classical V-stabilities up to translation. 
\end{remark}

In order to study the poset structure on $\VStab(X)$, we now define a poset structure on $\Deg(X)$ and study the compatibility of the degeneracy map $\D$ with the poset structures on the domain and the codomain. 

\begin{definition}\label{D:deg-pos}
Let $\D^1,\D^2\in \Deg(X)$. We say that  $\D^1\geq \D^2$ if $\D^1\subseteq \D^2$ and there exists a subset $\E\subset \D^2-\D^1$ such that the following conditions hold:
\begin{enumerate}[(i)]
\item \label{D:deg-pos1} $\D^2-\D^1=\E\coprod \E^c$, where $\E^c:=\{Y^c: \: Y\in \E\}$.
\item \label{D:deg-pos2} for any $Z_1,Z_2\in \D^2-\D^1$ such that $Z_1\wedge Z_2=\emptyset$ and $Z_1\cup Z_2\in \D^1$, we have that 
$$
|\{Z_i\in \E\}|=1.
$$
\item \label{D:deg-pos3} for any $Z_1,Z_2,Z_3\in \D^2-\D^1$ with pairwise empty meet and with  $X=Z_1\cup Z_2\cup Z_3$, we have that
$$
|\{Z_i\in \E\}|=1 \text{ or } 2.
$$
\end{enumerate}

\end{definition}
Note that the dominance relation $\D^1\geq \D^2$ is strictly stronger than the inclusion $\D^1\subseteq \D^2$, as we will show in Example \ref{Es:4comp}.

\begin{proposition}\label{P:D-pos}
    The degeneracy map $\D$ of \eqref{E:map-D} satisfies the following properties
    \begin{enumerate}[(i)]
        \item \label{P:D-pos1} $\D$ is order preserving, i.e. $\s^1\geq \s^2\Rightarrow \D(\s^1)\geq \D(\s^2)$.
        \item \label{P:D-pos2} $\D$ is upper lifting, i.e. for any $\D^1\geq \D^2$ in $\Deg(X)$ and any $\s^2\in \VStab(X)$ such that $\D(\s^2)=\D^2$, there exists $\s^1\in \VStab(X)$ such that $\s^1\geq \s^2$ and $\D(\s^1)=\D^1$.
    \end{enumerate}
\end{proposition}
\begin{proof}
Part \eqref{P:D-pos1}: consider two V-stability conditions $\s_1\,s_2$ on $X$ such that $\s_1\geq \s_2$, and set $\chi:=|\s_1|=|\s_2|$. We have to show that $\D(\s^1)\geq \D(\s^2)$. First of all, observe that $\D(\s^1)\subseteq \D(\s^2)$ by Remark \ref{R:VStab-pos}. We now show that the subset 
$$
\E:=\{Z\in \D(\s^2)-\D(\s^1): \: \s^1_Z=\s_Z^2+1\}\subset \D(\s^2)-\D(\s^1),
$$
satisfies the conditions of Definition \ref{D:deg-pos}:

$\bullet$ Condition \eqref{D:deg-pos1} follows from Remark \ref{R:VStab-pos}.

$\bullet$ Take $Z_1,Z_2\in \D(\s^2)-\D(\s^1)$ such that $Z_1\wedge Z_2=\emptyset$ and $Z_1\cup Z_2\in \D(\s^1)$. Applying condition \eqref{E:sum-n} to $\s^1$ and $\s^2$ with respect to the biconnected subcurves $\{Z_1,Z_2,(Z_1\cup Z_2)^c\}$, we get
\begin{equation}\label{E:triango1}
\begin{sis}
&\s^2_{Z_1}+\s^2_{Z_2}+\s^2_{(Z_1\cup Z_2)^c}=0 \text{ because } Z_1,Z_2,(Z_1\cup Z_2)^c\in \D(\s^2),  \\
& \s^1_{Z_1}+\s^1_{Z_2}+\s^1_{(Z_1\cup Z_2)^c}=1 \text{ because } Z_1,Z_2\not \in \D(\s^1) \text{ and } (Z_1\cup Z_2)^c\in \D(\s^1).
\end{sis}
\end{equation}
Since $\s^2_{(Z_1\cup Z_2)^c}=\s^1_{(Z_1\cup Z_2)^c}$ (because $Z_1\cup Z_2\in \D^1\subseteq \D^2$), Condition \eqref{E:triango1} implies that 
$$
|\{Z_i\in \E\}|=1,
$$
and hence Condition \eqref{D:deg-pos1} holds. 

$\bullet$ Take now  $Z_1,Z_2,Z_3\in \D(\s^2)-\D(\s^1)$ with pairwise empty meet and with  $X=Z_1\cup Z_2\cup Z_3$. Applying condition \eqref{E:sum-n} to $\s^1$ and $\s^2$ with respect to the biconnected subcurves $\{Z_1,Z_2,Z_3\}$, we get
\begin{equation}\label{E:triango2}
\begin{sis}
& \s^2_{Z_1}+\s^2_{Z_2}+\s^2_{Z_3}=0 & \text{ because } Z_1,Z_2,Z_3\in \D(\s^2),\\
& \s^1_{Z_1}+\s^1_{Z_2}+\s^1_{Z_3}=1\text{ or } 2 & \text{ because } Z_1,Z_2,Z_3\not \in \D(\s^1).
\end{sis}
\end{equation}
This implies that
$$
|\{Z_i\in \E\}|=1 \text{ or } 2.
$$

%OLD PROOF: wrong
%consider two V-stability conditions $\s_1\,s_2$ on $X$ such that $\s_1\geq \s_2$, and set $\chi:=|\s_1|=|\s_2|$. We have to show that $\D(\s^1)\geq \D(\s^2)$. 

%First of all, observe that $\D(\s^1)\subseteq \D(\s^2)$ by Remark \ref{R:VStab-pos}. Consider the following equivalence relation $\sim$ on $\D(\s^2)_{\min}$:
%$$
%Y_1\sim Y_2\Longleftrightarrow 
%\begin{sis} 
%& \text{for any $Z\in \D(\s^1)$ we have that } \\
%& Y_1\subseteq Z \Longleftrightarrow Y_2\subseteq Z.\\
%\end{sis}
%$$

%\un{Claim:} For any equivalence class $[Y]$ of $\sim$ on $\D(\s^2)_{\min}$, we have that 
%$$
%\ov{[Y]}:=\bigcup_{W\in [Y]} W \text{ is either } X \text{ or it belongs to } \D(\s^1).
%$$

%\textcolor{green}{Da finire}

%Now the Claim, together with the definition of $\sim$ and the fact that any subcurve of $\D(\s^1)\subseteq \D(\s^2)$ is a union of subcurves of $\D(\s^2)_{\min}$, implies that the minimal subcurves of $\D(\s^1)$ are equal to 
%$$
%\D(\s^1)_{\min}=\{\ov{[Y]}: \text{ for every equivalence class } [Y]  \text{ such that } \ov{[Y]}\neq X\}.
%$$
%This shows that $\D(\s^1)\geq \D(\s^2)$ by Definition \ref{D:deg-pos}.

Part \eqref{P:D-pos2}: let $\D^1\geq \D^2$ in $\Deg(X)$ and $\s^2\in \VStab(X)$ such that $\D(\s^2)=\D^2$. By the Definition \ref{D:deg-pos} of the order relation in $\Deg(X)$, there exists $\E\subset \D^2-\D^1$ satisfying the three conditions of loc. cit.
We set 
$$
\s^1_Y=
\begin{cases}
\s_Y^2 & \text{ if } Y \in \D^1 \text{ or } Y\not\in \D^2,\\    
\s^2_Y+1 & \text{ if } Y\in \E,\\
\s^2_Y & \text{ if } Y\in \E^c.\\
\end{cases}
$$
We will now check that $\s^1$ is a V-stability on $X$ of characteristic $|\s^2|$ with $\D(\s^1)=\D^1$, which concludes the proof because then, by construction, we will also have that $\s^1\geq \s^2$.
In order to show this, we will check the conditions of Definition \ref{D:VStabX}. 

Equations \eqref{E:sum-n} holds with $\chi=|\s^2|$ by property \eqref{D:deg-pos1} of $\E$ and we also get that the $\s^1$-degenerate subcurves are exactly the ones belonging to $\D^1$. Hence, the first condition of Definition \eqref{E:cond2} holds since $\D^1\in \Deg(X)$. 

In order to check Equation \eqref{E:tria-n}, take $X=Z_1\cup Z_1\cup Z_3$ with $Z_1,Z_2,Z_3$ biconnected subcurves with pairwise empty meet. If each subcurve belongs to either $\D^1$ or $(\D^2)^c$, then we have that $\s^1_{Z_i}=\s^2_{Z_i}$ for any $1\leq i\leq 3$, and hence \eqref{E:tria-n} for $\s^2$ implies the same condition for $\s^1$. Otherwise, using that $\D^1,\D^2\in \Deg(X)$, we can have three cases (up to permuting $Z_1,Z_2,Z_3$): 

$\star$ \un{Case I}: $Z_1\in \D^2-\D^1$ and $Z_2,Z_3\not\in \D^2$. 

We have that 
$$
\sum_{i=1}^3 \s^1_{Z_i}-\chi=
\begin{cases}
    \sum_{i=1}^3 \s^2_{Z_i}-\chi=1 & \text{ if } Z_1\in \E^c,\\
     \sum_{i=1}^3 \s^2_{Z_i}+1-\chi=2 & \text{ if } Z_1\in \E,\\
\end{cases}
$$
where we have used that $\sum_{i=1}^3 \s^2_{Z_i}-\chi=1$ because of \eqref{E:tria-n} for $\s^2$. 

$\star$ \un{Case II}:  $Z_1\in \D^1$ and $Z_2,Z_3\in \D^2-\D^1$. 

We compute  
$$
\sum_{i=1}^3 \s^1_{Z_i}-\chi=\sum_{i=1}^3 \s^2_{Z_i}+|\{Z_2,Z_3\}\cap \E|-\chi= 1,
$$
where we have used that $\sum_{i=1}^3 \s^2_{Z_i}-\chi=1$ because of \eqref{E:tria-n} for $\s^2$ and the fact that exactly one among $Z_2$ and $Z_3$ belongs to $\E$ by Definition \ref{D:deg-pos}\eqref{D:deg-pos2} for $\E$.

$\star$ \un{Case III}: $Z_1,Z_2,Z_3\in \D^2-\D^1$.

We get that
$$
\sum_{i=1}^3 \s^1_{Z_i}-\chi=\sum_{i=1}^3 \s^2_{Z_i}+|\{Z_i\in \E\}|-\chi=1 \text{ or } 2,
$$
where we have used that $\sum_{i=1}^3 \s^2_{Z_i}-\chi=0$ because of \eqref{E:tria-n} for $\s^2$ and that $|\{Z_i\in \E\}|=1$ or $2$ by Definition \ref{D:deg-pos}\eqref{D:deg-pos3} for $\E$.

In each of the above three cases, Equation \eqref{E:tria-n} is satisfied for $\s^1$ wit respect to the three subcurves $\{Z_1,Z_2,Z_3\}$ and we are done. 
\end{proof}

\begin{example}\label{Ex:elem-moves}
There are two easy moves that imply $\D^1\geq \D^2$:
\begin{itemize}
    \item (Move I) $\D^1_{\min}=\D^2_{\min}- \{Y,Y^c\}$, for some subcurve $Y$ such that $Y,Y^c\in \D^2_{\min}$.
    \item (Move II) $\D^1_{\min}=\D^2_{\min}\cup\{Y_1\cup Y_2\}- \{Y_1,Y_2\}$, for some subcurves $Y_1,Y_2\in \D^2_{\min}$ such that $Y_1\wedge Y_2=\emptyset$ and $Y_1\cup Y_2\in  \BCon(X)$.
\end{itemize}
Indeed, in the first move we clearly have the inclusion $\D^1\subseteq \D^2$ and the subset 
$$\E=\{Z\in \D^2-\D^1: \: Y\subseteq Z\}$$ 
satisfies the conditions of Definition \ref{D:deg-pos}.

In the second move, again the inclusion $\D^1\subseteq \D^2$ is clear and the subset 
$$\E=\{Z\in \D^2-\D^1: \: Y_1\subseteq Z\}$$ 
satisfies the conditions of Definition \ref{D:deg-pos}.
\end{example}

\begin{corollary}\label{C:max-pos}
The poset $\VStab(X)$ satisfies the following two properties:
    \begin{enumerate}[(i)]
        \item a V-stability $\s\in \VStab(X)$ is maximal in $\VStab(X)$ if and only if $\D(\s)=\emptyset$  (i.e. $\s$ is general).
         \item a V-stability $\s\in \VStab(X)$ is submaximal in $\VStab(X)$, i.e. it is not maximal and it is dominated only by maximal elements, if and only if $\D(\s)=\{Y,Y^c\}$ for some $Y\in \BCon(X)$. 
    \end{enumerate}
\end{corollary}
\begin{proof}
  This follows from Proposition \ref{P:D-pos} and the fact that $\emptyset$ is the maximal element of $\Deg(X)$, while its submaximal elements are of the form $\{Y,Y^c\}$, for some $Y\in \BCon(X)$, as it follows easily from Example \ref{Ex:elem-moves}.   
\end{proof}

\begin{example}\label{Es:4comp}
    Let $X$ be a connected reduced curve with $4$ irreducible components and assume that all its subcurves are connected (which implies that all its non-trivial subcurves are biconnected). We fix an identification $I(X)=\{1,2,3,4\}$. Then the Hasse diagram of the poset $\Deg(X)$, up to the permutation action of the symmetric group $S_4$, looks as follows (where $\{i,j,k,l\}=\{1,2,3,4\}$)
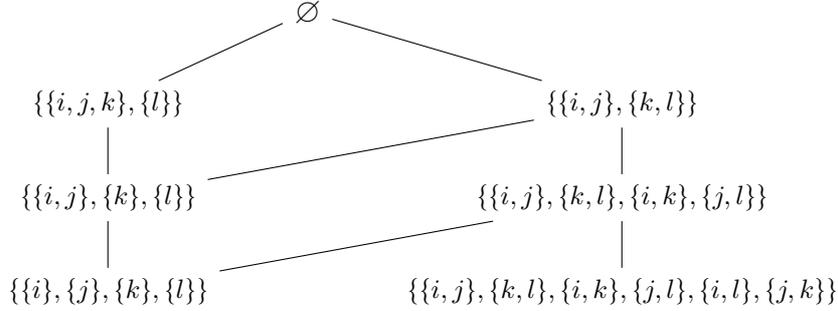
    $$
    \begin{tikzcd}
      & \emptyset \arrow[dash]{dr} \arrow[dash]{dl}& \\
      \{\{i,j,k\},\{l\}\} \arrow[dash]{d} && \{\{i,j\},\{k,l\}\}\arrow[dash]{dll} \arrow[dash]{d} \\
      \{\{i,j\},\{k\},\{l\}\} \arrow[dash]{d} & &  \{\{i,j\},\{k,l\}, \{i,k\},\{j,l\}\}\arrow[dash]{d}\arrow[dash]{dll} \\
      \{\{i\},\{j\},\{k\},\{l\}\}  & &  \{\{i,j\},\{k,l\}, \{i,k\},\{j,l\},\{i,l\},\{j,k\}\} 
    \end{tikzcd}
    $$
    where we have represented a degeneracy subset by the set of its minimal elements.

This example exhibits two interesting features:
\begin{itemize}
    \item The minimal element on the right is contained in the minimal element on the left but it is not greater than it.
    \item All the covering relations are given by the two moves in Example \ref{Ex:elem-moves} except for the covering relation 
    $$\D^1:=\{\{i,j\},\{k,l\}, \{i,k\},\{j,l\}\}\geq \D^2:=\{\{i\},\{j\},\{k\},\{l\}\}, $$
    which is due to the following subset (or its complementary subset):
    $$
    \E:=\{S\subset \D^2-\D^1:\: S\subseteq \{i,l\} \text{ or } \{i,l\}\subseteq S\}.$$
\end{itemize}
\end{example}

We end up this discussion on the poset of V-stability conditions with the following two natural questions:

\begin{question}\label{Q:Deg}
\noindent
\begin{enumerate}
    \item Is the poset $\Deg(X)$ a ranked poset of rank equal to $|I(X)|-1$?
    \item Is the degeneracy map $\D:\VStab(X)\to \Deg(X)$ surjective?
\end{enumerate}
\end{question}

A positive answer to the Question~\ref{Q:Deg} would imply, using Proposition \ref{P:D-pos}, that also $\VStab(X)$ is a ranked poset of rank equal to $|I(X)|-1$ and that the map $\D$ is rank preserving. 

\begin{comment}
\textcolor{green}{F: Le cose qua sotto vanno rivalutate: servono o le possiamo togliere?}

\begin{corollary}\label{C:genV-funBound}
    Let $Y,Z\subseteq X$ be subcurves, such that $Y=\bigsqcup_{i\in I}Y_i$, $Z=\bigsqcup_{j\in J}Z_j$ and $Y\cup Z\in\Con(X)$. Then, there exists an integer $N$ such that 
    $$
     \max\{-1,-|Y_i,Z_j\textup{ is $\s$-nondegenerate}+1\}\leq \tilde{\s}_{Y\cup Z}-\tilde{\s}_Y-\tilde{\s}_Z\leq N.
    $$
\end{corollary}
\begin{proof}
    It follows by inductively applying \ref{P:genV-funP}(\ref{P:genV-funTria}) to $Y\cup Z=\bigsqcup_{i\in I}Y_i\cup\bigsqcup_{j\in J}Z_j$.
\end{proof}
\end{comment}

\subsection{V-stability conditions for families}\label{Sub:Vfamily}

The aim of this subsection is to define V-stability conditions for families of reduced curves. 

We start with some preliminary remarks. Given a family  $\pi:X\to S$ of connected reduced curves and an 
\'etale specialization $\xi: s\rightsquigarrow t$ of geometric points of $S$, there is an induced surjective map on subcurves 
\begin{equation}\label{E:xi*}
\begin{aligned}
\xi_*:\{\text{Subcurves of } X_s\}& \twoheadrightarrow \{\text{Subcurves of } X_t\} \\
Y & \mapsto \ov{Y}\cap X_t
\end{aligned}
\end{equation}
(here $\overline{Y}$ is the closure of $Y$ in $X$), which clearly preserves joins and meets, and the number of connected components (see \cite[Lemma 5.2]{MRV}). 
%Since the morphism $\ov Y\to \ov{\{s\}}$ is flat, the map $\xi_*$ preserves the arithmetic genus and the number of connected components. 
In particular,  $\xi_*$ sends connected (resp. biconnected) subcurves into connected (resp. biconnected) subcurves. 
Therefore, there is an induced pull-back map on V-stability conditions (of a given characteristic $\chi$):
\begin{equation}\label{E:VStab-fun}
   \begin{aligned}
      \xi^*:\VStab^\chi(X_t) & \to \VStab^\chi(X_s) \\
      \s & \mapsto \xi^*(\s)=\{\xi^*(\s)_Y:=\s_{\xi_*(Y)}: \: Y\in \BCon(X_s)\}.
   \end{aligned} 
\end{equation}
Observe that the degeneracy subset (resp. the extended degeneracy subset) and the V-function (resp. the extended V-function) of the pull-back are given by the following formulas:
$$
\begin{sis}
& \D(\xi^*(\s))=(\xi_*)^{-1}(\D(\s)),\\
& \wh \D(\xi^*(\s))=(\xi_*)^{-1}(\wh \D(\s)),\\ 
& \wh{\xi^*(\s)}=\xi^*(\wh{\s}).
\end{sis}
$$

\begin{definition}\label{D:VStabXS}
Let $\pi:X\to S$ be a family of connected reduced curves. 

 A \emph{(relative) stability condition of vine type} (or simply a \textbf{V-stability condition)} of characteristic $\chi$ on $X/S$ is a collection of V-stabilities 
 $$\s=\{\s^s\in \VStab^\chi(X_s)\: : s \text{ is a geometric point of } S\},$$
such that, for  any \'etale specialization $\xi: s\rightsquigarrow t$ of geometric points of $S$, we have that $\xi^*(\s^t)=\s^s$. 

We say that $\s=\{\s^s\}$ is \emph{general} if $\s^s$ is general for every geometric point $s$ of $S$.

We denote by $\VStab^\chi(X/S)$ the set of all V-stability conditions on $X/S$ of characteristic $\chi$.
\end{definition}
In other words, the set of V-stability conditions on $X/S$ of characteristic $\chi$ is equal to the inverse limit of sets
$$
\VStab^\chi(X/S)=\varprojlim \VStab^\chi(X_s),
$$
where the inverse limit is taken with respect to the pull-back maps induced by \'etale specializations of geometric points on $S$.

    \begin{example} \label{Ex; PTfine}  (Fine PT stabilities for families of nodal curves). Let $X/S$ be a family of connected nodal curves with $S$ connected. Then  each  general V-stability condition $\s$ of characteristic $\chi$ on $X/S$ 
    %with the property that $\s^s$ is general for every geometric point $s$ of $S$, 
    corresponds to a family of degree $d=\chi-1+g(X_s)$ stability assignments according to \cite[Definition~4.4, Definition~4.11]{pagani2023stability}. 

This follows by applying \cite[Theorem~1.20]{viviani2023new} to $\s^s$ for all geometric points $s$ of $S$, and by observing that the  condition prescribed in Definition~\ref{D:VStabXS} coincides with the compatibility condition of \cite[Definition~4.11]{pagani2023stability}.
    %with $S$ irreducible and with generic element $X_\theta/\theta$ smooth. 
    \end{example}

Again, the easiest way of producing relative V-stability conditions is via relative numerical polarizations, as we now explain. Observe that if $\xi: s\rightsquigarrow t$ is an \'etale specialization of geometric points of $S$, then, using the map of \eqref{E:xi*}, we can define a pull-back 
\begin{equation}\label{E:Pol-fun}
   \begin{aligned}
      \xi^*:\Pol^\chi(X_t) & \to \Pol^\chi(X_s) \\
      \psi & \mapsto \xi^*(\psi)=\{\xi^*(\psi)_Y:=\psi_{\xi_*(Y)}: \: Y\subseteq X_s\}.
   \end{aligned} 
\end{equation}

\begin{lemma-definition}\label{LD:numpol-fam}
Let $\pi:X\to S$ be a family of connected reduced curves. 

A \emph{(relative) numerical polarization} on $X/S$ of characteristic $\chi\in \ZZ$ is  a collection of numerical polarizations
 $$\psi=\{\psi^s\in \Pol^\chi(X_s)\: : s \text{ is a geometric point of } S\},$$
such that, for  any \'etale specialization $\xi: s\rightsquigarrow t$ of geometric points of $S$, we have that $\xi^*(\psi^t)=\psi^s$.

The function 
\begin{align*}
    \s(\psi)=\{\s(\psi^s)\in \VStab^\chi(X_s)\: : s \text{ is a geometric point of } S\}
    \end{align*}
is a V-stability condition on $X/S$ of characteristic $\chi$ (called the V-stability condition associated to $\psi$).
\end{lemma-definition}
The relative V-stability conditions of the form $\s(\psi)$ are called \emph{classical}.
\begin{proof}
 This follows  from Lemma-Definition \ref{LD:numpol} and from the fact that the compatibility $\psi^t\circ \xi_*=\psi^s$ implies $\xi^*(\s(\psi^t))=\s(\psi^s)$. 
\end{proof}

The set of numerical polarizations on $X/S$ of characteristic $\chi$ is denoted by $\Pol^\chi(X/S)$ and it is equal to the inverse limit 
$$
\Pol^\chi(X/S)=\varprojlim \Pol^\chi(X_s),
$$
with respect to the pull-back maps induced by \'etale specializations of geometric points on $S$. Therefore, $\Pol^\chi(X/S)$ is a real affine space  with underlying real vector space $\Pol^0(X/S)$. 

%We also set 
%$$
%\Pol(X/S)=\coprod_{\chi \in \ZZ}\Pol^\chi(X/S).
%$$

%Consider the arrangement of hyperplanes in the affine space $\RR^{I(X)}$ given by  
%\begin{equation}\label{E:arr-hyper}
%\A_{X}:=\left\{\psi_Y=n\right\}_{Y\in \BCon(X), n\in \ZZ}.
%\end{equation}
% We get an induced wall and chamber decomposition on $\Pol(X)$ such that two numerical polarizations $\psi, \psi'$ belong to the same region, i.e. they have the same relative positions with respect to all the hyperplanes, if and only if $\s(\psi)=\s(\psi')$. In other words, 

The map
 \begin{equation}\label{E:map-s-fam}
 \begin{aligned}
     \left\lceil - \right\rceil: \Pol(X/S):= \coprod_{\chi \in \ZZ}\Pol^\chi(X/S)& \longrightarrow \VStab(X/S):=\coprod_{\chi \in \ZZ}\VStab^\chi(X/S)\\
    \psi & \mapsto \s(\psi)
 \end{aligned}
 \end{equation}
 is the inverse limit of the maps of \eqref{E:map-s} on  geometric fibers. 
 
 %induces a bijection between the set of regions induced by  $\A_{X}$ on $\Pol(X)$ and the set of classical V-stability conditions on $X$. Note also that $\s(\psi)$ is a general V-stability condition if and only if $\psi$ belongs to a chamber (i.e. a maximal dimensional region), or equivalently if it does not lie on any of the hyperplanes of $\A_X$. 

\begin{example}\label{Ex:VBstab}
There are two standard ways of producing relative numerical polarizations on a family $X/S$ of connected reduced curves:
%(see \cite[Rmk. 2.16, 2.17]{MRV})

\begin{enumerate}
    \item \label{Ex:VBstab1}  Let $L$ be a relatively ample line bundle on $X$ of  constant relative degree $\deg(L):=\deg(L_{|X_s})>0$ (the condition that $\deg(L_{|X_s})$ is constant is automatic if $S$ is connected). Then, for any $\chi \in \Z$,  
    $$
    \psi(L,\chi):=\left\{\psi(L,\chi)^s_Y:=\frac{\deg(L_{|Y})\cdot \chi}{\deg(L)} \: \text{ for any } Y\subseteq X_s\right\}_{s}
    $$
    is a numerical polarization on $X/S$ of characteristic $\chi$.
    \item \label{Ex:VBstab2} Let $E$ be a vector bundle on $X/S$ with constant  integral relative slope $\mu_{X/S}(E):=\mu(E_{|X_s})\in \ZZ$ 
%and assume that the geometric fibers of $X/S$ have constant genus equal to $g(X/S)$ 
(the condition that $\mu(E_{|X_s})$ is constant is automatic if $S$ is connected). 
Then
$$
\psi(E):=\left\{\psi(E)^s_Y:=-\mu(E_{|Y})=-\frac{\deg(E_{|Y})}{\rk(E)}\: \text{ for any } Y\subseteq X_s \right\}_s
$$
is a numerical polarization on $X/S$ of characteristic $\chi=-\mu_{X/S}(E)$.    
\end{enumerate}
Notice that \eqref{Ex:VBstab1} is a special case of \eqref{Ex:VBstab2} since 
$$\psi(L,\chi)=\psi(\O_X^{\oplus (\deg(L)-1)}\oplus L^{-\chi}).$$

The fact that $\psi(E)$ is a well-defined numerical polarization on $X/S$ follows from the fact that, for any \'etale specialization $\xi:s\rightsquigarrow t$ of geometric points of $S$ and any subcurve $Y\subset X_s$, we have that $\mu(E_{|Y})=\mu(E_{|\xi_*(Y)})$ and $g(Y)=g(\xi_*(Y))$ because of the flatness of $\ov Y\to \{s\}$.

The geometric meaning of the above two relative numerical polarizations will be clarified in Example~\ref{Ex:classcJ}.
\end{example}

\section{The category of semistable sheaves}\label{Sub:ss}

Throughout this section, we fix a connected reduced curve $X$ and a V-stability condition $\s$ on $X$ of characteristic $\chi$.

The aim of this section is to study the category of rank-$1$ torsion-free sheaves on subcurves of $X$ that are semistable (resp. polystable, resp. stable) with respect to $\s$, as in the following:

\begin{definition}\label{D:ss}
   Let $\s$ be a V-stability of characteristic $\chi$ on a connected reduced curve $X$. Let $I$ be a coherent sheaf on $X$ that is rank~$1$ and torsion-free on some subcurve $Y=\supp(I)\subseteq X$, and let $Y=\coprod_{i=1}^h Y_i$ be its decomposition  into connected components. 

We say that $I$ is \emph{$\s$-semistable} if for any $1\leq i\leq h$:
\begin{enumerate}[(i)]
    \item $Y_i$ is $\s$-degenerate, i.e. $Y_i\in \wh \D(\s)$.
    \item $I_{Y_i}$ is $\s(Y_i)$-semistable, i.e. 
    \begin{itemize}
    \item $\chi(I_{Y_i})=|\s(Y_i)|=\s_{Y_i}$;
    \item $\chi(I_Z)=\chi((I_{Y_i})_Z)\geq \s(Y_i)_Z$, for any $Z\in \BCon(Y_i)$. 
    \end{itemize}
\end{enumerate}
If $I$ is $\s$-semistable, we also say that $I$ is:
\begin{itemize}
    \item\label{D:ss-ps} \emph{$\s$-polystable} if, for each $i=1,\ldots,h$, and for each $Z\in\D(\s(Y_i))$, whenever $\chi(I_Z)=\chi((I_{Y_i})_Z)= \s(Y_i)_Z$, we have $I_{Y_i}=(I_{Y_i})_Z\oplus (I_{Y_i})_{\ov{Y_i-Z}}$;
    \item\label{D:ss-s} \emph{$\s$-stable} if, for each $i=1,\ldots,h$, and for each $Z\in\D(\s(Y_i))$, we have $\chi(I_Z)=\chi((I_{Y_i})_Z)> \s(Y_i)_Z$.
\end{itemize}
\end{definition}
Note that, in particular, a rank-$1$ torsion-free sheaf on $X$ is $\s$-semistable if $\chi(I)=|\s|$ and $\chi(I_Z)\geq \s_Z$ for any $Z\in \BCon(X)$.

\begin{remark}\label{R:gen-s}
    For any V-stability condition $\s$ on $X$ and any rank-$1$ torsion-free sheaf on $X$, we have that 
    $$
    I \text{ is $\s$-stable} \Rightarrow I \text{ is $\s$-polystable} \Rightarrow I \text{ is $\s$-semistable}
    $$
    and the three notions coincide if $\s$ is general.
\end{remark}

From the definition of $\s$-semistability, we deduce several other inequalities which we collect in the following Lemma (in which we discuss, for simplicity, only sheaves supported on $X$).

\begin{lemma}\label{L:ss-deg}
 Let $I$ be a $\s$-semistable sheaf  on $\supp(I)=X$. Then we have:
 \begin{enumerate}[(i)]
     \item \label{L:ss-deg1} If $W\in \BCon(X)$, then 
     $$
     \chi(\leftindex_{W}{I})\leq 
     \begin{cases}
       \s_W-1 & \text{ if $W$ is $\s$-nondegenerate,}\\
        \s_W & \text{ if $W$ is $\s$-degenerate.}\\
     \end{cases}
     $$
     \item \label{L:ss-deg2}  For any  $W\in \Con(X)$, we have that    
     $$\chi(I_W)\geq \s_W\geq |\s|-\s_{W^c}\geq  \chi(\leftindex_W{I}).$$
     %VECCHI ENUNCIATI
     %\item \label{L:ss-deg2} If $W\in \Con(X)$ and we denote by $W^c=\coprod_i Z_i$ the decomposition into connected components (so that each $Z_i$ is biconnected), then 
  %$$\chi(I_W)\geq \s_W=|\s|-\sum_i \s_{Z_i}+|\{Z_i \text{ is $\s$-nondegenerate}\}|\geq |\s|-\sum_i \s_{Z_i}\geq \chi(\leftindex_W{I}).$$
   %  \item \label{L:ss-deg3} If $W\in \wh \D(\s)$, then     $$\chi(I_W)\geq \s_W\geq  \chi(\leftindex_W{I}).$$
 \end{enumerate}
\end{lemma}
\begin{proof}
Part \eqref{L:ss-deg1} follows from 
$$
\chi(\leftindex_{W}{I})=\chi(I)-\chi(I_{W^c})\leq |\s|-\s_{W^c}=
\begin{cases}
       \s_W-1 & \text{ if $W$ is $\s$-nondegenerate,}\\
        \s_W & \text{ if $W$ is $\s$-degenerate,}\\
     \end{cases}
$$
where we have used \eqref{E:add-chi} in the first equality, the $\s$-semistability of $I$ in the inequality and \eqref{E:sum-n} in the second equality. 

Part \eqref{L:ss-deg2}: let $W^c=\coprod_i Z_i$ be the decomposition into connected components (so that each $Z_i$ is biconnected). From the $\s$-semistability of $I$ and from Part~\eqref{L:ss-deg1}, we get (for any $Z_i$):
\begin{equation*}\label{E:ineqZi}
\chi(I_{Z_i})\geq \s_{Z_i} \text{ and } \chi(\leftindex_{Z_i}{I})\leq 
\begin{cases}
    \s_{Z_i}-1 & \text{ if $Z_i$ is $\s$-nondegenerate,}\\
        \s_{Z_i} & \text{ if $Z_i$ is $\s$-degenerate.}\\
\end{cases}
\end{equation*}
Summing the above inequality over all the $Z_i$ and using that 
$I_{W^c}=\oplus_i I_{Z_i}$ and $\leftindex_{W^c}{I}=\oplus_i \leftindex_{Z_i}{I}$, we get 
\begin{equation}\label{E:ineqWc}
\chi(I_{W^c})\geq \sum_i\s_{Z_i} \text{ and } \chi(\leftindex_{W^c}{I})\leq \sum_i \s_{Z_i} -|\{Z_i: Z_i\not\in \D(\s)\}|. 
\end{equation}
Part \eqref{L:ss-deg2} follows now from the above inequalities \eqref{E:ineqWc} together with the fact that 
$$
\chi(I_W)=|\s|-\chi(\leftindex_{W^c}{I})\geq |\s|-\sum_i \s_{Z_i}+|\{Z_i: Z_i\not\in \D(\s)\}|=\s_W\geq |\s|- \s_{w^c}\geq |\s|-\chi(I_{W^c})=\chi(\leftindex_{W}{I}),
$$
where the first and last equality follows \eqref{E:add-chi} and $\chi(I)=|\s|$, and the middle equality follows from the Definition \ref{D:ext-s} of the extended V-function. 
%Part \eqref{L:ss-deg3} follows from \eqref{L:ss-deg2} using that $W\in \wh \D(\s)$ if and only if each $Z_i$ is $\s$-degenerate and that, in this case, Lemma \ref{L:add-whn} implies that 
%$$
%\wh \s_{W}=|\s|-\sum_i\s_{Z_i}.
%$$
\end{proof}

We now prove three properties of the category of $\s$-semistable sheaves. 

\begin{proposition}\label{P:res-ss}
  Let $I$ be $\s$-semistable and let $Y$ be a subcurve of $\supp(I)$. Then the following conditions are equivalent:
  \begin{enumerate}[(i)]
      \item \label{P:res-ss1} $I_Y$ is $\s$-semistable;
      \item \label{P:res-ss2} the connected components $\{Y_i\}_{i\in I}$ of $Y$ are $\s$-degenerate and we have $\chi(I_{Y_i})=\s_{Y_i}$ for any $i\in I$;
      \item \label{P:res-ss3} $\leftindex_{Y^c}{I}$ is $\s$-semistable;
      \item \label{P:res-ss4} the  connected components of $\{W_j\}_{j\in J}$ of $Y^c$ are $\s$-degenerate and we have $\chi(\leftindex_{W_j}{I})=\s_{W_j}$ for any $j\in J$.
  \end{enumerate}
\end{proposition}
\begin{proof}
It is enough to argue on each connected component of $\supp(I)$; hence, we can assume that $\supp(I)=X$. We prove the following implications:

$\bullet$ $\eqref{P:res-ss1}\Longrightarrow \eqref{P:res-ss2}$ is obvious from Definition \ref{D:ss}.

$\bullet$ $\eqref{P:res-ss2}\Longrightarrow \eqref{P:res-ss1}$: we have to show that $(I_Y)_{Y_i}=I_{Y_i}$ (see Lemma \ref{L:IYZ}\eqref{L:IYZ1}) is $\s(Y_i)$-semistable for any $i\in I$. Fix one such index $i\in I$. By assumption,  we have that $Y_i$ is $\s$-degenerate and that $\chi(I_{Y_i})=\s_{Y_i}=|\s(Y_i)|$. It remains to observe that for any $Z\in \BCon(Y_i)$ we have that 
$$
\chi((I_{Y_i})_Z)=\chi(I_Z)\geq \s_Z=\s(Y_i)_Z,
$$
where the first equality follows from Lemma \ref{L:IYZ}\eqref{L:IYZ1}, the inequality follows from Lemma \ref{L:ss-deg}\eqref{L:ss-deg2} using that $I$ is $\s$-semistable on $X=\supp(I)$ and the last equality follows from the definition of $\s(Y_i)$ (see Lemma-Definition \ref{LD:V-subgr}).

$\bullet$ $\eqref{P:res-ss3}\Longrightarrow \eqref{P:res-ss4}$ follows from  Definition \ref{D:ss} together with the fact that $\leftindex_{Y^c}{I}=\oplus_{j} \leftindex_{W_j}{I}$ because $Y^c=\coprod_{j} W_j$, which then implies, using Lemma \ref{L:IYZ}\eqref{L:IYZ1} and Lemma \ref{L:IY}\eqref{L:IY2}, that 
\begin{equation}\label{E:resI-Wi} 
\leftindex_{W_j}{I}=\leftindex_{W_j}{(\leftindex_{Y^c}{I})}=(\leftindex_{Y^c}{I})_{W_j}.
\end{equation}

$\bullet$ $\eqref{P:res-ss4}\Longrightarrow \eqref{P:res-ss3}$:  we have  to show that $(\leftindex_{Y^c}{I})_{W_j}=\leftindex_{W_j}{I}$ (see \eqref{E:resI-Wi}) is $\s(W_j)$-semistable for any $j\in J$. Fix one such index $j\in J$. By assumption, we have that $W_j$ is $\s$-degenerate and that $\chi(\leftindex_{W_j}{I})=\s_{W_j}=|\s(W_j)|$.  It remains to observe that for any $Z\in \BCon(W_j)$ we have that 
$$
\chi((\leftindex_{W_j}{I})_Z)=\chi(I_Z)\geq \s_Z=\s(W_j)_Z,
$$
where the inequality follows from Lemma \ref{L:ss-deg}\eqref{L:ss-deg2} using that $I$ is $\s$-semistable on $X=\supp(I)$, the last equality follows from the definition of $\s(W_j)$ (see Lemma-Definition \ref{LD:V-subgr}) and the first equality follows by applying Lemma \ref{L:IYZ}\eqref{L:IYZ1}
$$
(\leftindex_{W_j}{I})_Z=\leftindex_{Z}{(I_{\ov{X-(W_j-Z)}})}=\leftindex_{Z}{(I_{Z\coprod_j Y_j \coprod_{k\neq j }W_k})}=I_Z.
$$

$\bullet$ $\eqref{P:res-ss2}\Longleftrightarrow \eqref{P:res-ss4}$: Lemma \ref{L:deg-compl} below implies that each $Y_i$ is $\s$-degenerate if and only if each $W_j$ is $\s$-degenerate. 
Hence, we can assume that each $Y_i$ and each $W_j$ is $\s$-degenerate. 
Therefore we have that 
$$
\sum_i \s_{Y_i}+\sum_j \s_{W_j}=|\s|=\chi(I)=\sum_{i} \chi(I_{Y_i})+\sum_j\chi(\leftindex_{W_j}{I}),
$$
where the first equality follows from Corollary \ref{C:add-whn}, the second one from the $\s$-semistability of $I$ and the third one by \eqref{E:add-chi} together with the fact that $I_Y=\oplus_i I_{Y_i}$ and $\leftindex_{Y^c}{I}=\oplus_j \leftindex_{W_j}{I}$.
We conclude by applying Lemma \ref{L:ss-deg}\eqref{L:ss-deg2} which gives 
$$
\chi(I_{Y_i})\geq \s_{Y_i} \text{ and } \chi(\leftindex_{W_j}{I})\leq \s_{W_j}. 
$$

\end{proof}

\begin{lemma}\label{L:deg-compl}
   Let $Y$ be a subcurve of $X$. Then every connected component of $Y$ is $\s$-degenerate if and only if every connected component of $Y^c$ is $\s$-degenerate.  
\end{lemma}
\begin{proof}
Denote by $Y=\coprod_{i=1}^n Y_i$ and $Y^c=\coprod_{j=1}^m W_j$ the decompositions into connected components. 
Since the statement is symmetric in $Y$ and $Y^c$, it is enough to prove that if each $Y_i$ is $\s$-degenerate then each $W_j$ is $\s$-degenerate. We will argue by induction on the pair $(n.m)$. If either $n=1$ or $m=1$ then the statement follows from the definition \ref{D:ext-n} of $\wh D(\s)$. Assume that $n,m\geq 2$. Consider the decomposition $Y_1^c=\coprod_{k=1}^l Z_k$ into connected components. 
Since $Y_1$ is $\s$-degenerate, it follows that each $Z_k$ is $\s$-degenerate (and biconnected). There  exist two functions 
$$
\alpha:\{1,\ldots,n\}\setminus\{1\}\to \{1,\ldots, k\} \text{ and } \beta:\{1,\ldots, m\}\to \{1,\ldots,k\} 
$$
such that 
$$
Z_k=\coprod_{i\in \alpha^{-1}(k)}Y_i\coprod_{j\in \beta^{-1}(k)}W_j.
$$
Consider now the V-stability condition $\s(Z_k)$. Lemma-Definition \ref{LD:V-subgr} and the assumption that $Y_i\in \wh \D(\s)$ for any $i$, imply that $Y_i\in \wh \D(\s(Z_k))$ for any $i\in \alpha^{-1}(k)$. We now apply our induction hypothesis to conclude that 
$W_j\in \wh \D(\s(Z_k))$ for any $j\in \beta^{-1}(k)$. This implies, again by Lemma-Definition \ref{LD:V-subgr}, that $W_j\in \wh \D(\s)$ for any $j\in \beta^{-1}(k)$. Since this is true for any $1\leq k\leq l$, the proof is completed. 
\end{proof}

\begin{proposition}\label{P:morf-ss}
Let $f:I\to J$ be a morphism between $\s$-semistable sheaves. Then $\ker(f)$, $\Im(f)$ and $\coker(f)$ are $\s$-semistable. 
\end{proposition}

\begin{proof}
We first show the following 

 \un{Claim}: $\Im(f)$ is $\s$-semistable and we have that 
 $$f:I\twoheadrightarrow I_{\supp(\Im(f))}=\Im(f)=\leftindex_{\supp(\Im(f))}{J}\subseteq J. $$

  Indeed, observe that $\Im(f)$ is torsion-free since it is a subsheaf of the torsion-free sheaf $J$. Therefore, we have that 
\begin{equation}\label{E:Imf}
\begin{sis}
    & \Im(f)=I_{\supp(\Im(f))},\\
    &\Im(f)\subseteq \leftindex_{\supp(\Im(f))}{J}  \text{ with torsion cokernel sheaf.}
\end{sis}   
\end{equation}

Pick a connected component of $\supp(\Im(f))$ and call it $W$. Let $Y$ and $Z$ be the connected components of $\supp(I)$ and $\supp(J)$, respectively, such that $W\subseteq Y\cap Z$.  Consider the decomposition of $W^c$ into connected components $W^c=\coprod_i V_i$. We now compute 

\begin{equation}\label{E:ImfIneq}
\begin{aligned}
& \chi(\Im(f)_W)=\chi(I_W)=\chi((I_Y)_W) & \text{by \eqref{E:Imf}}\\
&\geq\s(Y)_W & \text{by Lemma~\ref{L:ss-deg}\eqref{L:ss-deg2} using that $I$ is $\s$-semistable}\\
& \geq\s_W & \text{by Lemma \ref{L:ExtVfunCompar}\eqref{L:ExtVfunCompar1}} \\
& =|\s|-\sum_i \s_{V_i}+|\{V_i: V_i\not\in \D(\s)\}| \geq |\s|- \s_{W^c} & \text{by Definition \ref{D:ext-s}}\\
& \geq|\s(Z)|-\s(Z)_{\ov{Z- W}} & \text{by Lemma \ref{L:ExtVfunCompar}\eqref{L:ExtVfunCompar2}} \\
& \geq \chi(\leftindex_W{(J_Z)})=\chi(\leftindex_W{J})& \text{by Lemma~\ref{L:ss-deg}\eqref{L:ss-deg2} using that $J$ is $\s$-semistable}\\ 
&\geq \chi(\leftindex_W{\Im(f)})=\chi(\Im(f)_W) & \text{by \eqref{E:Imf}}.
\end{aligned}
 \end{equation}

%where the third to last inequality holds, since every connected component $V_i$ of $W^\mathsf{c}$ is the union of connected components of $Z\setminus W$ and $Z^\mathsf{c}$, with the latter being all degenerate. As $V_i\in\BCon(X)$ and is filtered by biconnected subcurves, by Lemma~\ref{L:ExtV-funP}, we obtain
%$$
%|\s(Z)|-\s(Z)_{Z\setminus W}\leq\s_Z-\s_{Z\setminus W}=|\s|-%\s_{Z^\mathsf{c}}-\s_{Z\setminus W}\leq|\s|-\s_{W^\mathsf{c}},
%$$
%where the first inequality is given by Lemma~\ref{L:ExtVfunCompar}.
 
Therefore, all the inequalities in \eqref{E:ImfIneq} must be equalities, from which we deduce that 
\begin{enumerate}[(i)]
 \item \label{E:cond1} $\Im(f)_W=\leftindex_W{J}$.
    \item \label{E:cond2} Each $V_i$ is $\s$-degenerate, which implies that $W\in \wh \D(s)$.
    \item \label{E:cond3} $\chi(\Im(f)_W)=|\s|-\sum_i \s_{V_i}=\s_{W}$ by \eqref{E:cond2} and Corollary \ref{C:add-whn}. 
\end{enumerate}
Since \eqref{E:cond1} is true for any connected component $W$ of $\supp(\Im(f))$, we deduce that 
\begin{equation}\label{E:Imf2}
\Im(f)=\leftindex_{\supp(\Im(f))}{J}
\end{equation}
Moreover, using \eqref{E:cond2} and \eqref{E:cond3} for any connected component $W$ of $\supp(\Im(f))$, Proposition \ref{P:res-ss} implies that 
\begin{equation}\label{E:Imf-ss}
    \Im(f) \text{ is $\s$-semistable.}
\end{equation}
The Claim follows by putting together \eqref{E:Imf}, \eqref{E:Imf2} and \eqref{E:Imf-ss}.

We now show:

\un{End of the proof}: The sheaf $\Im(f)$ is $\s$-semistable by the Claim. The sheaves $\ker(f)$ and $\coker(f)$ are $\s$-semistable by Proposition \ref{P:res-ss} using that 
$$
\begin{sis}
   &\ker(f)=\ker(I\twoheadrightarrow I_{\supp(\Im(f))}=\Im(f))=
   \leftindex_{\supp(\Im(f))^c}{I}, \\
   &\coker(f)=\coker(\Im(f)=\leftindex_{\supp(\Im(f))}{J}\hookrightarrow J)=J_{\supp(\Im(f))^c}.
\end{sis}
$$

\end{proof}

\begin{lemma}\label{L:ExtVfunCompar}
    Let $\s$ be a $V$-stability condition on $X$, and let $Y\in\wh\D(\s)$. Then, for any connected subcurve $W\subseteq Y$, we have that:
    \begin{enumerate}[(i)]
        \item \label{L:ExtVfunCompar1} $\s(Y)_W\geq\s_W$.
        \item \label{L:ExtVfunCompar2} $|\s(Y)|-\s(Y)_{\ov{Y-W}}\leq |\s|-\s_{\ov{X-W}}$.
    \end{enumerate}
\end{lemma}
\begin{proof}
    Let $\ov{Y- W}=\bigsqcup_{i\in I}V_i$ with $V_i\in \BCon(Y)$ and $Y^\mathsf{c}=\bigsqcup_{j\in J}Z_j$ with $Z_j\in \D(\s)\subseteq \BCon(X)$ be the respective decompositions into connected components. Then, there exist partitions $\{I(k)\}_{k\in K}$ and $\{J(k)\}_{k\in K}$ of $I$ and $J$ respectively, such that $W^\mathsf{c}=\ov{X-W}$ decomposes into connected components as follows 
    $$
    W^\mathsf{c}=\bigsqcup_{k\in K}B_k,
    $$
    where $B_k=\bigcup_{i\in I(k)}V_i\cup\bigcup_{j\in J(k)}Z_j\in \BCon(X)$.

   % We observe that, for every $k\in K$, the subcurve $B_k$ belongs in $\BCon(X)$ and that it can be filtered by biconnected subcurves. Hence, by inductively applying Property~\ref{D:VStabX}\eqref{E:cond2}, we obtain
   By applying iteratively Lemma \ref{L:ExtV-funP} to $B_k=\bigcup_{i\in I(k)}V_i\cup\bigcup_{j\in J(k)}Z_j$ and using that the subcurves $Z_j$ are all $\s$-degenerate, we get that 
    \begin{equation}\label{E:ExtVfunCompar1}
    0\geq  \s_{B_k}-\sum_{i\in I(k)}\s_{V_i}-\sum_{j\in J(k)}\s_{Z_j}\geq
     |\{B_k\not\in \D(\s)\}|-|\{V_i: i\in I(k) \text{ and } V_i\not\in \wh \D(\s)\}|. 
       % \s_{B_k}-\sum_{i\in I(k)}\s_{V_i}-\sum_{j\in J(k)}\s_{Z_j}\geq\min\{-|\{V_i\textup{ $\s$-nondegenerate}\}|+1,0\},
    \end{equation} 
   Part \eqref{L:ExtVfunCompar1} follows from 
    \begin{equation*}
    \begin{aligned}
        \s_W &=|\s|-\sum_{k\in K}\s_{B_k}+|\{B_k: B_k \not \in \D(\s) \}|&\textup{by Definition~\ref{D:ext-s}}\\
        &\leq |\s|-\sum_{i\in I}\s_{V_i}-\sum_{j\in J}\s_{Z_j}+|\{V_i: V_i\not\in\wh \D(\s)\}|&\textup{by \eqref{E:ExtVfunCompar1}}\\
        &= \s_Y-\sum_{i\in I}\s_{V_i}+|\{V_i: V_i\not\in\wh \D(\s)\}|&\textup{by Corollary \ref{C:add-whn} since $Z_j\in\D(\s)$}\\
        &=|\s(Y)|-\sum_{i\in I}\s(Y)_{V_i}+|\{V_i: V_i\not\in \D(\s(Y))\}|&\textup{by  Lemma-Definition \ref{LD:V-subgr}}\\
        &=\s(Y)_W&\textup{by Definition~\ref{D:ext-s}}.
    \end{aligned}
    \end{equation*}
    Part \eqref{L:ExtVfunCompar2} follows from 
    \begin{equation*}
    \begin{aligned}
     & |\s|-\s_{\ov{X-W}}=\s_Y+\sum_{j\in J}\s_{Z_j}-\sum_{k\in K}\s_{B_k}
     & \textup{by Corollary \ref{C:add-whn} since $Z_j\in\D(\s)$}\\
     & \geq \s_Y-\sum_{i\in I}\s_{V_i}=|\s(Y)|-\s(Y)_{\ov{Y-W}} &\textup{by \eqref{E:ExtVfunCompar1}.}\\
    \end{aligned}
    \end{equation*}
\end{proof}

\begin{proposition}\label{P:ext-ss}
    Let $I$ and $J$ two $\s$-semistable sheaves such that $\supp(I)\wedge \supp(J)=\emptyset$ and $\supp(I)\cup \supp(J)$ is a disjoint union of subcurves in $\wh \D(\s)$. Then any extension 
    \begin{equation}\label{E:ext-K}
    0\to I\to K \to J \to 0
    \end{equation}
    is such that $K$ is $\s$-semistable.
\end{proposition}

\begin{proof}
 Since $I$ and $J$ are torsion-free of rank-$1$ with $\supp(I)\wedge \supp(J)=\emptyset$, then $K$ is torsion-free of rank-$1$ on $\supp(K)=\supp(I)\cup \supp(J)$.
 Moreover, since $\supp(K)$ is a disjoint union of subcurves of $\wh \D(\s)$, we can assume, up to resticting to each of these subcurves, that $\supp(K)=X$. Equivalently, if we set $Y:=\supp(J)$ then $Z:=Y^c=\supp(I)$. Note that \eqref{E:ext-K} implies that $J=K_Y$ and $I=\leftindex_{Z}{K}$. Let $Y=\bigsqcup_iY_i$ and $Z=\bigsqcup_jZ_j$ be the respective decompositions into connected components. 
 
 Fix $W\in \BCon(X)$ and consider the following commutative diagram
\begin{equation}\label{E:diagIJK}
    \begin{tikzcd}
&\leftindex_{W^c}K\arrow[r,"\delta"]\arrow[d,hookrightarrow]&  \leftindex_{Y\cap W^c}{(K_Y)}   \arrow[d,hookrightarrow] \\
 \leftindex_{Z}K=I\arrow[r,hookrightarrow] \arrow[d,"\gamma"]& K \arrow[r,twoheadrightarrow] \arrow[d,twoheadrightarrow]& K_Y=J\arrow[d,twoheadrightarrow]\\
\leftindex_{Z\cap W}{(K_W)} \ar[r,hookrightarrow] & K_W\arrow[r,twoheadrightarrow] & K_{Y\cap W}=J_{Y\cap W},
    \end{tikzcd}
\end{equation}
where the two bottom rows and the two right columns are exact. From the above commutative diagram, we deduce that $$\ker(\gamma)=\ker(\delta)=\leftindex_{Z\cap W^c}{I}.$$
 Hence, the map $\gamma$ induces an injective morphism 
\begin{equation}\label{E:quotI}
I_{Z\cap W}\hookrightarrow  \leftindex_{Z\cap W}{(K_W)} \text{ with torsion cokernel.}  
\end{equation}

We now compute 
\begin{equation*}
  \begin{aligned}
   & \chi(K_W)= \chi(J_{Y\cap W})+\chi( \leftindex_{Z\cap W}{(K_W)}) & \text{ by \eqref{E:diagIJK}}\\
   & \geq \chi(J_{Y\cap W})+\chi(I_{Z\cap W}) & \text{ by \eqref{E:quotI}}\\
   & \geq \sum_i\s(Y_i)_{Y_i\cap W} + \sum_j\s(Z_j)_{Z_j\cap W} & \text{by Lemma \ref{L:ss-deg}\eqref{L:ss-deg2} using that $I$ and $J$ are $\s$-semistable } \\
   %since Y\cap W\in \BCon(Y), Y^c\cap W\in \BCon(Y^c)\\
  & \geq \sum_i\s_{Y_i\cap W} + \sum_j\s_{Z_j\cap W}=\s_{Y\cap W} + \s_{Z\cap W} & \text{ by Lemma~\ref{L:ExtVfunCompar}\eqref{L:ExtVfunCompar1}}\\
  & \geq \s_W & \text{ by the following Claim.}
  \end{aligned}  
\end{equation*}
Since this is true for any $W\in \BCon(X)$, we conclude that the sheaf $K$ is $\s$-semistable.

%\underline{Claim 1:} $\tilde\s(Y)_{Y\cap W} + \tilde\s(Z)_{Z\cap W}\geq\tilde\s_{Y\cap W} + \tilde\s_{Z\cap W}$.

\underline{Claim:} $\s_{Y\cap W} + \s_{Z\cap W}\geq\s_W$.

By Definition~\ref{D:ext-s}, we have that 

$$
\s_{Y\cap W}=\sum_i \s_{W\cap Y_i}=\sum_{i,k} s_{W_{Y_i,k}}\quad\textup{and}\quad \s_{Z\cap W}=\sum_j \s_{W\cap Z_j}=\sum_{j,h} s_{W_{Z_j,h}},
$$
where, for every $i$ and $j$, $W\cap Y_i=\bigsqcup_kW_{Y_i,k}$ and $W\cap Z_j=\bigsqcup_hW_{Z_j,h}$ are the decomposition into connected components.

Since $X$ is connected, there exist connected components, that we label $Y_0$ and $Z_0$, of $Y$ and $Z$ respectively such that both $Y_0$ and $Y_0\cup Z_0\in\BCon(X)$ (up to exchanging $Y$ and $Z$).

Choose connected components $W_Y$ and $W_Z$ of $W\cap Y_0$ and $W\cap Z_0$ respectively, such that $W_0:=W_Y\cup W_Z\in\Con(X)$. We can assume $W_0\in\BCon(X)$. Indeed, since both $W$ and $Y_0\cup Z_0$ are biconnected, if $W_0\notin \BCon(X)$, then $W=W_0\cup(Y_0\cup Z_0)^\mathsf{c}$ and $\ov{(Y_0\cup Z_0)- W_0}\in\Con(X)$. Thus, our claim follows from Lemma~\ref{L:ExtV-funP}, since $(Y_0\cup Z_0)^\mathsf{c}\in\D(\s)$.

Let $W_0\in\BCon(X)$.
Then, again by Lemma~\ref{L:ExtV-funP}, we obtain:
$$
\s_{W_Y}+ \s_{W_Z}\geq  \s_{W_0}.
$$

Our claim is proven by repeating this reasoning inductively on the connected components of $Y\cap W$ and $Z\cap W$.
\end{proof}

We end this section  by clarifying the relationship between $\s$-semistable, $\s$-polistable and $\s$-stable sheaves. We first prove that $\s$-polystable sheaves are direct sums of $\s$-stable sheaves.

\begin{proposition}\label{P:poly-sum}
    Let $I$ be an $\s$-semistable sheaf. Then $I$ is $\s$-polystable if and only if its decomposition into indecomposable subsheaves $I=\bigoplus_{j\in J}I_{Y_j}$, is such that $I_{Y_j}$ is $\s$-stable for each  $j\in J$.
\end{proposition}
\begin{proof}
    Without loss of generality, we can assume that $\supp(I)=X$. First, we observe that $Y_j\in\wh\D(\s)$, for each $j\in J$. Indeed, since $I$ is $\s$-semistable, we have
    $$
    |\s|=\chi(I)=\sum_{j\in J}\chi(I_{Y_j})\geq\sum_{j\in J}\s_{Y_j},
    $$
    and the observation follows from Remark~\ref{R:ext-s}. We now prove the two implications.
    
    $\un{\text{Proof of } \Rightarrow:}$  By contradiction, suppose there exists $Z\in\D(\s(Y_j))$, such that $\chi(I_{Y_j})_Z=\s(Y_j)_Z$. Note that $Z\in \wh \D(\s)$ since $\D(\s(Y_j))\subseteq \wh \D(\s)$ by Lemma-Definition \ref{LD:V-subgr}.
    %Then,
    %$$    \chi(I_Z)=\chi(I_{Y_i})_Z=\s(Y_i)_Z=\s_Z,$$
    %so $Z\in\wh\D(\s)$.
    Hence, since $I$ is polystable, we have that $I=I_Z\oplus I_{Z^c}$. However, this contraddicts the decomposition  $I=\bigoplus_{j\in J}I_{Y_j}$ into indecomposable subsheaves since $\emptyset \subsetneq Z\subsetneq Y_j$.
    %$I_{Y_i}=I_{Z}\oplus I_{\ov{Y_i- Z}}$ by Lemma~\ref{L:ps-sDhat}(\ref{L:ps-sDhat-ps}), which gives a contradiction since $I_{Y_i}$ is indecomposable.7
     
    $\un{\text{Proof of } \Leftarrow:}$ Let $Z\in\D(\s)$ be such that $\chi(I_Z)=\s_Z$ for some $j \in J$. By the decomposition into indecomposable subsheaves, we have $$
     I_Z=\bigoplus_{j\in J}I_{Z\wedge Y_j},
     $$
     and, thus,
     $$
     \chi(I_Z)=\sum_{j\in J}\chi(I_{Z\wedge Y_j}).
     $$
     We notice that, since both $Z$ and $Y_j$ are $\s$-degenerate, then also the connected components of $Z\wedge Y_j$ must be $\s$-degenerate, for each $j\in J$.
     Hence, since $I_{Y_j}$ is $\s(Y_j)$-stable for every $j$, by Lemma~\ref{L:ps-sDhat}(\ref{L:ps-sDhat-s}), we have that 
     $\chi(I_{Z\wedge Y_j})\geq\s(Y_j)_{Z\wedge Y_j}$, with equality if and only if $Y_j\subseteq Z$.
     Therefore, by using Lemma \ref{L:ExtVfunCompar}\eqref{L:ExtVfunCompar1}, we deduce that
     $$
     \chi(I_Z)=\sum_{j\in J}\chi(I_{Z\wedge Y_j})\geq\sum_{j\in J}\s(Y_j)_{Z\wedge Y_j}\geq\sum_{j\in J}\s_{Z\wedge Y_j}=\s_Z,
     $$
     with equality  only if there exists a subset of indices $J'\subseteq J$, such that $Z=\bigcup_{j\in J}Y_j$, that is, if and only if $I$ decomposes at $Z$.

\end{proof}

In the following, we generalize the formal defining properties of stable and polystable (Definition~\ref{D:ss}) in terms of elements of the extended degeneracy set $\wh{\D}$ (instead of $\D$). 

\begin{lemma}
    \label{L:ps-sDhat}
    Let $I$ be an $\s$-semistable sheaf with support equal to $Y\subseteq X$, and let $Y=\coprod_{i=1}^h Y_i$ be its decomposition into connected components. Then, we have that:
    \begin{enumerate}[(i)]
        \item\label{L:ps-sDhat-ps} $I$ is $\s$-polystable if and only if, for each $i$ and for each $Z\in\wh\D(\s(Y_i))$, whenever $\chi(I_Z)=\chi((I_{Y_i})_Z)= \s(Y_i)_Z$, we have $I_{Y_i}=(I_{Y_i})_Z\oplus (I_{Y_i})_{\ov{Y_i-Z}}$;
    \item\label{L:ps-sDhat-s} $I$ is $\s$-stable if and only if, for each $i$ and for each $Z\in\wh\D(\s(Y_i))-\{Y_i\}$, we have $\chi(I_Z)=\chi((I_{Y_i})_Z)> \s(Y_i)_Z$.
    \end{enumerate}
\end{lemma}
\begin{proof}
    We can assume without loss of generality that $\supp(I)=X$. Let $Z\in\wh\D(\s)-\{X\}$ and denote by $Z^\mathsf{c}=\bigsqcup_{j\in J} Z_j$ the decomposition of $Z^c$ into connected components, so that $Z_j\in\D(\s)$ for each $j\in J$. 

    \vspace{0.1cm}
    
    \un{Claim:} If $\chi(I_Z)=\s_Z$ then $\chi(I_{Z_j^c})=\s_{Z_j^c}$ for any $j\in J$. 

    Indeed, we have that 
    \begin{equation*}\label{E:eqz1}
    \sum_{j\in J}\chi(\leftindex_{Z_i}I)= \chi(\leftindex_{Z^c}I)=\chi(I)-\chi(I_Z)=\s-\s_Z=\sum_{j\in J}\s_{Z_j},
    \end{equation*}
    where we have used $\leftindex_{Z^c}I=\oplus_{j\in J}\leftindex_{Z_i}I$ in the first equality, \eqref{E:add-chi} in the second equality, the assumptions on $I$ on the third equality, and Corollary~\ref{C:add-whn} in the last equality. 

    The above equation together with Lemma \ref{L:ss-deg}\eqref{L:ss-deg2} implies that 
    \begin{equation*}\label{E:eqz2}
        \chi(\leftindex_{Z_i}I)=\s_{Z_i}.
    \end{equation*}
    Then, by applying again \eqref{E:add-chi} and using that $Z_i\in \D(\s)$, we get that
    $$
    \chi(I_{Z_i^c})=\chi(I)-\chi(\leftindex_{Z_i}I)=\s-\s_{Z_i}=\s_{Z_i^c},
    $$
    which concludes the proof of the Claim.

\vspace{0.1cm}

    We now use the Claim to conclude the proof. 
Part \eqref{L:ps-sDhat-s} follows from the Claim and the fact that if $I$ is $\s$-stable then $\chi(I_{Z_i^c})>\s_{Z_i^c}$ for any $j\in J$. 
In order to prove Part \eqref{L:ps-sDhat-ps}, take $Z\in \wh \D(\s)-\{X\}$ such that $\chi(I_Z)=\s_Z$. Then the Claim implies that $\chi(I_{Z_j^c})=\s_{Z_j^c}$ for every $j\in J$. Since $I$ is $\s$-polystable, then $I$ decomposed at $Z_j^c$ (or equivalently at $Z_j$) for any $j\in J$. But then this implies that $I$ decomposes at $Z$ as well. 
\end{proof}

As a corollary of the above proposition, we obtain the following characterization of stable sheaves. 

\begin{corollary}\label{C:stab-poly}
  A sheaf $I$ is $\s$-stable if and only if it is $\s$-polystable and simple.  
\end{corollary}

We finally prove  that every $\s$-semistable sheaf isotrivially specializes to a unique $\s$-polystable sheaf, which moreover does not further isotrivially specialize to other $\s$-semistable sheaves.

\begin{proposition}\label{P:iso-poly}
Let $\s$ be a V-stability condition on $X$. Then we have that:
\begin{enumerate}[(i)]
    \item \label{P:iso-poly1} For any $\s$-semistable sheaf $I$, there exists a unique isotrivial specialization $I\rightsquigarrow I^{pl}$ to some $\s$-polystable sheaf $I^{pl}$.
    \item \label{P:iso-poly2} If we have an isotrivial specialization    $I\rightsquigarrow J$ such that $I$ is $\s$-polystable and $J$ is $\s$-semistable, then we have that $I=J$.
\end{enumerate}
\end{proposition}
\begin{proof}
    We assume without loss of generality that $\supp(I)=X$.

Part~\eqref{P:iso-poly1}: let us first prove the existence. If $I$ is not $\s$-polystable, we pick a subcurve $Y\in \D(\s)$ such that $\chi(I_Y)=\s_Y$ and  $I\neq I_Y\oplus I_{Y^c}$, and we apply Lemma \ref{L:Gr-ss} and Proposition~\ref{P:iso-spec} in order to get  a non-trivial isotrivial specialization 
$$
I\rightsquigarrow \Gr_{(Y,Y^c)}(I)=I_Y\oplus \leftindex_{Y^c}I, \quad \text{ with } I_Y \text{ and } \leftindex_{Y^c}I \text{ $\s$-semistable.}
$$
By iterating the same argument for $I_Y$ and $\leftindex_{Y^c}I$, we arrive at an isotrivial specialization of $I$ into a $\s$-polystable sheaf. The uniqueness of such an isotrivial specialization follows from the Jordan-H\"older theorem. 

 Part~\eqref{P:iso-poly2}: let $I\rightsquigarrow J$ be an isotrivial specialization with $I$ $\s$-polystable and $J$ $\s$-semistable. By Proposition~\ref{P:iso-spec}, $J=\Gr_{Y_\bullet}(I)$ for some ordered partition $Y_\bullet=(Y_0,\ldots,Y_q)$ of $X$. We now proceed by induction in $q$ (observe that the conclusion is obvious if  $q=0$). 
By Proposition \ref{P:iso-spec} and \eqref{E:GrY}, the isotrivial specialization $I\rightsquigarrow J$ can be decomposed as
\begin{equation}\label{E:deco-iso}
I\rightsquigarrow \Gr_{(Y_0,Y_0^c)}(I)=I_{Y_0}\oplus \leftindex_{Y_0^c} I\rightsquigarrow I_{Y_0}\oplus \Gr_{(Y_1,\ldots, Y_q)}(\leftindex_{Y_0^c} I)=\Gr_{(Y_0,\ldots, Y_q)}(I)=J.
\end{equation}
Since $J$ is $\s$-semistable and $\s$-semistability is an open condition, also the intermediate sheaf $\Gr_{(Y_0,Y_0^c)}(I)$ is $\s$-semistable. 
We can now apply Lemma \ref{L:Gr-ss} and Lemma \ref{L:ps-sDhat}\eqref{L:ps-sDhat-ps} to the $\s$-polystable sheaf $I$ in order to conclude that the first isotrivial specialization in \eqref{E:deco-iso} is trivial, i.e. 
\begin{equation}\label{E:Iis-triv}
I=\Gr_{(Y_0,Y_0^c)}(I)=I_{Y_0}\oplus \leftindex_{Y_0^c} I.
\end{equation}
Now notice that the sheaves $I_{Y_0}$ and $\leftindex_{Y_0^c} I$ are $\s$-polystable since they are direct summands of the $\s$-polystable sheaf $I$, and the sheaf $\Gr_{(Y_1,\ldots, Y_q)}(\leftindex_{Y_0^c} I)$ is $\s$-semistable since it is the quotient of the $\s$-semistable sheaf $J$ by the $\s$-semistable direct factor $I_{Y_0}$ (by Proposition \ref{P:morf-ss}). Therefore, we can apply the induction hypothesis and conclude that 
\begin{equation}\label{E:IIis-triv}
 \leftindex_{Y_0^c} I= \Gr_{(Y_1,\ldots, Y_q)}(\leftindex_{Y_0^c} I).
\end{equation}
By combining \eqref{E:Iis-triv} and \eqref{E:IIis-triv}, and using the factorization \eqref{E:deco-iso}, we deduce that $I=J$ as required.

\end{proof}

\begin{lemma}\label{L:Gr-ss}
Let $\s$ be a V-stability condition on $X$ and let $I$ be a $\s$-semistable sheaf with $\supp(I)=X$.
Let $Y\subseteq X$ be a subcurve with decomposition into connected components $Y=\coprod_{i} Y_i$.
Then we have that 
$$
\Gr_{(Y,Y^c)}(I)=I_Y\oplus \leftindex_{Y^c}I \text{ is $\s$-semistable} \Longleftrightarrow  Y_i\in \wh \D(\s) \text{ and } \chi(I_{Y_i})=\s_{Y_i} \text{ for every }i. 
$$
\end{lemma}
\begin{proof}
Observe that the equality $ \Gr_{(Y,Y^c)}(I)=I_Y\oplus \leftindex_{Y^c}I$ follows from \eqref{E:GrY}. 
We prove the two implications separately:

\un{Proof of $\Leftarrow$}: the assumption and Proposition \ref{P:res-ss} imply that $I_Y$ and $\leftindex_{Y^c}I$ are $\s$-semistable. Then Proposition \ref{P:ext-ss} implies that $I_Y\oplus \leftindex_{Y^c}I$ is $\s$-semistable and we are done.

\un{Proof of $\Rightarrow$}: let $J:=\Gr_{(Y,Y^c)}(I)$  and denote by $Y^c=\coprod_{j} W_j$ the decomposition of $Y^c$ into connected components. By the expression for $\Gr_{(Y,Y^c)}(I)$ observed above, we have the decomposition 
$$
J=\bigoplus_{i} I_{Y_i}\bigoplus_{j} \leftindex_{W_j}I,
$$
which implies that $J_{W_j}=\leftindex_{W_j}I$ for any $j$. Then using that $I$ and $J$ are $\s$-semistable and Lemma~\ref{L:ss-deg}\eqref{L:ss-deg2}, we compute
\begin{equation*}\label{E:2ineq}
 \s_{W_j}\leq \chi(J_{W_j})=\chi(\leftindex_{W_j}I)\leq |\s|-\s_{W_j^c}.   
\end{equation*}
The above equation, together with Definition~\ref{D:ext-s}, implies that each connected component of $W_j^c$ is $\s$-degenerate (and hence $W_j\in \wh \D(\s)$) and that $\chi(\leftindex_{W_j}I)=\s_{W_j}$. We now conclude using Proposition~\ref{P:res-ss}.
\end{proof}

\section{V-compactified Jacobians}\label{Sub:VcJ}

In this section, we define V-compactified Jacobians for families of connected reduced curves. 

First of all, we define relative compactified Jacobian stacks and spaces.

\begin{definition}\label{D:cJ-fam}
    Let $\pi:X\to S$ be a family of connected reduced curves. 
A \emph{compactified Jacobian stack} of characteristic $\chi$ for $X/S$ is an open substack $\ov \J_{X/S}^{\chi}\subseteq \TF_{X/S}^{\chi}$, admitting a relative proper good moduli space (in the sense of \cite{Alp}) $F:\ov \J_{X/S}^{\chi}\xrightarrow{\Xi} \ov J_{X/S}^{\chi}\xrightarrow{f} S$, called the associated \emph{compactified Jacobian space}.
\end{definition}

\begin{remark} \label{R: connected}
    Compactified Jacobians of (families of) connected nodal curves and in the fine case in \cite{PTgenus1}, \cite{pagani2023stability} and \cite{viviani2023new} require the further assumption that their fibers over all geometric points should also be \emph{connected}. Here we have dropped the latter assumption motivated by the fact that, to the best of our knowledge, the connectedness is already unknown for the compactified Jacobians defined in \cite{esteves}, even  for a reduced and connected curve $X$ 
 over an algebraically closed field. 

 However, we will later prove  a connectedness result assuming that the curves have planar singularities, as part of Theorem~\ref{T:vcJ-fam-pla}.
\end{remark}

Now we define the V-semistable locus of a relative V-stability condition (recall Definition~\ref{D:VStabXS}). 

\begin{lemma-definition}\label{LD:VStab}
    Let $\pi:X\to S$ be a family of connected reduced curves and let $\s=\{\s^s\}$ be a stability condition on $X/S$ of characteristic $\chi$.

    The associated \emph{V-semistable locus} is the open substack $\ov \J_{X/S}(\s)$ of $\TF^\chi_{X/S}$ whose fiber over $T\xrightarrow{f} S$ is given by 
$$
\ov \J_{X/S}(\s)(T):=\left\{
\begin{aligned}
    I \in  \TF^\chi_{X/S}(T): \: \chi((I_{|X_{t}})_Y)\geq \s^{f(t)}_Y & \text{ for every geometric point $t$ of $T$ }\\
    & \text{ and for every } Y \in \BCon(X_{f(t)})
    \end{aligned}\right\}.
$$
\end{lemma-definition}
\begin{proof}
 We have to check that $\ov \J_{X/S}(\s)$ is open in $\TF_{X/S}^\chi$.
Since $\ov \J_{X/S}(\s)$ is constructible in $\TF_{X/S}^\chi$, it is enough to prove that $\ov \J_{X/S}(\s)$ is stable under generalization. 

Consider an \'etale specialization $\xi: s\rightsquigarrow t$ of geometric points of $S$ and suppose that $I_{|X_t}\in \ov \J_{X/S}(\s)(t)$, which is equivalent to say that 
$$
\chi((I_{|X_t})_Z)\geq \s^t_Z \text{ for every } Z \in \BCon(X_t).
$$
For any $Y\in \BCon(X_s)$, we have that 
$$
\chi((I_{|X_s})_Y)\geq \chi((I_{|X_t})_{\xi_*(Y)})\geq \s^t_{\xi_*(Y)}=\s^s_Y,
$$
where the first inequality follows from the fact that $(I_{|X_t})_{\xi_*(Y)}$ is a quotient of the restriction to $\xi_*(Y)$ of the flat limit of $(I_{|X_s})_Y$ over $\ov{\{s\}}$, the second inequality follows from our assumption on $I_{|X_t}$, and the last equality follows from the equality $\xi^*(\s^t)=\s^s$ of Definition \ref{D:VStabXS}. This implies that $I_{|X_s}\in \ov \J_{X/S}(\s)(s)$ and we are done. 
\end{proof}

\begin{remark} \label{R:open-moregeneral}
    Note that we have not used Conditions~(1) and (2) of Definition~\ref{D:VStabX} in the proof of  Lemma-Definition~\ref{LD:VStab}.  Thus  the  substack $\ov \J_{X/S}(\s)$ of $\TF^\chi_{X/S}$ defined in Lemma-Definition~\ref{LD:VStab} is open under the weaker assumption that $\s=\{\s^s\}$ is \emph{any} collection of functions $\s^s \colon \BCon(X_s) \to \mathbb{Z}$ such that, for every \'etale specialization $\xi: s \rightsquigarrow t$ of geometric points of $S$, we have $\xi^*(\s^t)=\s^s$.  
\end{remark}

\begin{remark}\label{R:inc-VcJ}
    Let $X$ be a connected reduced curve and let $\s,\t\in \VStab(X)$. We have that:
    \begin{enumerate}[(i)]
        \item     If $ \s\geq \t$  then we have an inclusion $\ov\J_X(\s)\subseteq \ov \J_X(\t)\subseteq \TF_X^{|\s|=|\t|}$, as it follows from Definition \ref{D:VStab-pos} and Lemma-Definition \ref{LD:VStab}.
        \item  If $\s$ and $\t$ are equivalent by translation, then there exists a line bundle $L$ on $X$ such that the tensorization by $L$ induces an isomorphism $-\otimes L:\ov \J_X(\s) \xrightarrow{\cong} \ov \J_X(\t)$.

        Indeed, if $\s$ and $\t$ are equivalent by translation then there exists a function $\tau:I(X)\to \ZZ$ such that $\t=\s+\tau$ (see Definition \ref{D:tranVStab}). Then if $L$ is any line bundle on $X$ such that the multidegree of $L$ is equal to $\tau$ (i.e. $\deg(L_{X_v})=\tau_v$ for any $v\in I(X)$), then the tensorization by $L$ induces the required isomorphism using Lemma-Definition \ref{LD:VStab} and the formula $\chi((I\otimes L)_Y)=\chi(I_Y)+\deg(L_Y)$.
        \end{enumerate}
\end{remark}

The main result of this section is the following 

\begin{theorem}\label{T:VcJ}
    Let $\pi:X\to S$ be a family of connected reduced curves over a  quasi-separated and locally Noetherian algebraic space $S$ and let $\s=\{\s^s\}$ be a V-stability condition on $X/S$. Then $\ov \J_{X/S}(\s)$ is a compactified Jacobian stack for $X/S$ of characteristic $|\s|$.
    %i.e. it admits a relative proper good moduli space $\ov J_{X/S}(\s)\to S$.
%with geometrically connected fibers. 
\end{theorem}
We will call $\ov \J_{X/S}(\s)$ the \textbf{compactified Jacobian stack} associated to $\s$ (or  a \textbf{V-compactified Jacobian stack}) and its relative good moduli space $\ov J_{X/S}(\s)$ the \textbf{compactified Jacobian space} associated to $\s$ (or a \textbf{V-compactified Jacobian space}).
We have morphisms
\begin{equation}\label{E:morVcJ}
    F_{X/S}(\s):\ov \J_{X/S}(\s)\xrightarrow{\Xi_{X/S}(\s)} \ov J_{X/S}(\s)\xrightarrow{f_{X/S}(\s)} S.
\end{equation}

\begin{example}\label{Ex:classcJ}(Classical compactified Jacobians) 
If $\psi\in \Pol^{\chi}(X/S)$ as in Lemma-Definition \ref{LD:numpol-fam}, then the associated compactified Jacobian stacks/spaces 
\begin{equation}\label{E:morclcJ}
    F_{X/S}(\psi):\ov \J_{X/S}(\psi):=\ov \J_{X/S}(\s(\psi))\xrightarrow{\Xi_{X/S}(\psi)} \ov J_{X/S}(\psi):=\ov J_{X/S}(\s(\psi))\xrightarrow{f_{X/S}(\psi)} S
\end{equation}
are called \textbf{classical compactified Jacobian} stacks/spaces.

 This special subclass of compactified Jacobians includes all compactified Jacobians constructed in the literature prior to this work (with the exception of the very recent \cite{PTgenus1}, \cite{pagani2023stability}, \cite{viviani2023new} and \cite{fava2024}, which have inspired this theory):
 \begin{enumerate}
     \item If $\psi=\psi(L,\chi)$ is as in Example \ref{Ex:VBstab}\eqref{Ex:VBstab1}, then 
     $\ov \J_{X/S}(L,\chi):=\ov \J_{X/S}(\psi(L,\chi))$ and $\ov J_{X/S}(L,\chi):=\ov J_{X/S}(\psi(L,\chi))$ are the compactified Jacobian stack/space parameterizing relative rank-$1$ torsion-free sheaves of characteristic $\chi$ on $X/S$ that are $L$-slope semistable (see \cite{simpson}). These include Oda-Seshadri \cite{Oda1979CompactificationsOT} compactified Jacobians for a fixed nodal curve and Caporaso's \cite{caporaso} compactified Jacobian for the universal family over the moduli stack of stable curves. 
 \item If $\psi=\psi(E)$  as in Example \ref{Ex:VBstab}\eqref{Ex:VBstab2}, then 
     $\ov \J_{X/S}(E):=\ov \J_{X/S}(\psi(E))$ and $\ov J_{X/S}(E):=\ov J_{X/S}(\psi(E))$ are the compactified Jacobian stack/spaces parameterizing relative rank-$1$ torsion-free sheaves  on $X/S$ that are $E$-semistable in the sense of Esteves \cite{esteves}. 
\end{enumerate}
 We refer the reader to \cite{alexeev}, \cite{meloviviani}, \cite{CMKVlocal},\cite{MRV} for a comparison of the different constructions.
\end{example}

\begin{example}  \label{Ex: fine nodal} (Fine compactified Jacobians of nodal curves). 
If $\pi \colon X \to S$ is a family of nodal curves with $S$ irreducible and with smooth generic fiber, and if $\s$ is general
%(i.e. if $\Deg(\s^s)=\emptyset$ for every geometric point $s$ of $S$) 
then Theorem~\ref{T:VcJ} specializes to \cite[Theorem~6.3]{pagani2023stability}. (See Example~\ref{Ex; PTfine}). 
\end{example}

In order to prove Theorem \ref{T:VcJ}, we will check the three valuative criteria from \cite{AHLH} that ensure that the morphism $\ov \J_{X/S}(\s)\to S$ admits a relative proper good moduli space.

We first prove that  $\ov \J_{X/S}(\s)\to S$ is $\Theta$-complete (see \cite[Def. 3.10]{AHLH}), which means the following. For any discrete valuation ring R with residue field $k$ and quotient field $K$, consider the test object $\Theta_R:=[\AA^1_R/\Gm]$. The topological space of the stack $\Theta_R$ has  four points, which are related by the  specializations of Figure \ref{F:ThetaR}.

 \begin{figure}[hbt!]
 \[
\begin{tikzcd}
    1_{\eta}:=[\AA^1_K\setminus \{0\}/\Gm] \arrow[r] \arrow[d, rightsquigarrow] &  1:=[\AA^1_k\setminus \{0\}/\Gm] \arrow[d, rightsquigarrow] \\
    0_\eta:=[\Spec K/\Gm] \arrow[r] & 0:=[\Spec k/\Gm]
\end{tikzcd}
\]
\caption{The four points of $\Theta_R$ and their specializations: the horizontal arrows are ordinary specializations, while the vertical arrows are isotrivial specializations.\label{F:ThetaR}}
\end{figure}

Then $\ov \J_{X/S}(\s)\to S$ is $\Theta$-complete if any commutative diagram of solid arrows like the one in Figure~\ref{F:Thcompl} can be completed with a dashed arrow as below:  

\begin{figure}[hbt!]
 \[
\begin{tikzcd}
    \Theta_R\setminus\{0\} \arrow[r, "f"] \arrow[d, hookrightarrow] &  \ov \J_{X/S}(\s)\arrow[d] \\
    \Theta_R \arrow[r, "\phi"] \arrow[ru, dashrightarrow, "F"] & S
\end{tikzcd}
\]
\caption{$\Theta$-completeness of $\ov \J_{X/S}(\s)\to S$.\label{F:Thcompl}}
\end{figure}

\begin{proposition}\label{P:Th-compl}
    Notation as in Theorem \ref{T:VcJ}. Then $\ov \J_{X/S}(\s)\to S$ is $\Theta$-complete. 
\end{proposition}
\begin{proof}
Consider a commutative diagram of solid arrows as in Figure \ref{F:Thcompl}. 
We can assume that the residue field $k$ of $R$ is algebraically closed by \cite[Prop. 3.17]{AHLH} (and by using that $\ov \J_{X/S}(\s)\to S$ has affine diagonal, by \cite[\href{https://stacks.math.columbia.edu/tag/0DLZ}{Lemma 0DLZ}]{stacks-project}). 

Since $S$ is an algebraic space, the morphism $\phi$ factors as $\phi: \Theta_R\to \Spec R \xrightarrow{\ov{\phi}} S$. Therefore, we can base change the family $X\to S$ via the map $\ov{\phi}$ and hence assume that we have a family $X\to S=\Spec R=\Delta$.  We will use the same notation as in Subsection~\ref{S:spec}. 

By the description of $\Theta_R$ in Figure \ref{F:ThetaR}, and arguing as in the proof of Proposition \ref{P:iso-spec}, the morphism $f$ corresponds to a family $\I$ of  $\s$-semistable relative rank-$1$ torsion-free sheaves on $X/\Delta$ together with   a decreasing filtration of $\I_{\eta}$:
\begin{equation}\label{E:filt-eta}
\I_{\eta,\geq}: \I_{\eta}=\I_{\eta, 0}\supsetneq \I_{\eta,1}\supsetneq \ldots \supsetneq \I_{\eta, q}\supsetneq \I_{\eta, q+1}=0
\end{equation}
such that 
\begin{equation}\label{E:ssGreta}
\Gr(\I_{\eta, \geq}):=\bigoplus_{i=0}^q \I_{\eta,i}/\I_{\eta,i+1} \in \ov\J_{X_{\eta}}(\s^{\ov \eta}).
\end{equation}
In order to produce the required dashed arrow $F$ in Figure \ref{F:Thcompl}, we have to construct a  decreasing filtration of $\I$:
\begin{equation}\label{E:filt-Delta}
\I_{\geq}: \I=\I_{0}\supsetneq \I_{1}\supsetneq \ldots \supsetneq \I_{q}\supsetneq \I_{q+1}=0
\end{equation}
such that 
\begin{enumerate}[(i)]
    \item \label{E:Icon} $(\I_i)_{\eta}=\I_{\eta, i}$ for any $1\leq i \leq q$,\\
    \item \label{E:IIcon} $\Gr(\I_{\geq}):=\bigoplus_{i=0}^q \I_{i}/\I_{i+1} \in \ov\J_{X/\Delta}(\s)$.
\end{enumerate}

We argue as follows. By using the properness of the relative Quot scheme for $X/\Delta$, and arguing by induction on $i$, we deduce that there exists a unique filtration $\I_{\geq}$ as in \eqref{E:filt-Delta} satisfying \eqref{E:Icon} and such that, moreover, 
\begin{equation}\label{E:flat-filt}
 \I_i/\I_{i+1} \text{ is flat over $\Delta$ for any } 0\leq i \leq q.    
\end{equation}

%This implies that $\Gr(\I_{\geq})$ is a relative rank-$1$ torsion-free sheaf on $X/\Delta$, and hence, in order to prove \eqref{E:IIcon}, it remains to show that 
%\begin{equation}\label{E:IIIcon}
%    \Gr(\I_{\geq})_o:=\bigoplus_{i=0}^q (\I_{i}/\I_{i+1})_o \in \ov\J_{X_o}(\s^o).
%\end{equation}
%In order to show \eqref{E:IIIcon}, we first of all 
Consider now the central fiber of the decreasing filtration of \eqref{E:filt-Delta}:
\begin{equation}\label{E:filt-o}
\I_{o, \geq}: \I_o=\I_{o, 0}\supsetneq \I_{o, 1}\supsetneq \ldots \supsetneq \I_{o, q}\supsetneq \I_{o, q+1}=0
\end{equation}
Arguing as in Subsection~\ref{S:isospec}, we obtain a decreasing filtration of $X_o$ by subcurves  and an associated ordered filtration:
\begin{equation}\label{E:filtXo}
  \begin{aligned}
      & W_{o,\supseteq}: X_o=W_{o,0}\supseteq W_{o, 1}\supseteq \ldots \supseteq  W_{o,q}\supseteq W_{o,q+1}=\emptyset \text{ where } W_{o,i}:=\supp(\I_{o,i}), \\
      & Y_{o,\bullet}:=(Y_{o,0},\ldots, Y_{o,q}) \text{ where } Y_{o,i}:=\supp(\I_{o,i}/\I_{o,i+1}).
  \end{aligned}  
\end{equation}
Similarly, the filtration of \eqref{E:filt-eta} induces a decreasing filtration of $X_\eta$ by subcurves and an associated ordered filtration:
\begin{equation}\label{E:filtXeta}
  \begin{aligned}
      & W_{\eta,\supseteq}: X_\eta=W_{\eta,0}\supseteq W_{\eta,1}\supseteq \ldots \supseteq  W_{\eta,q}\supseteq W_{\eta,q+1}=\emptyset \text{ where } W_{\eta,i}:=\supp(\I_{\eta,i}), \\
      & Y_{\eta,\bullet}:=(Y_{\eta,0},\ldots, Y_{\eta,q}) \text{ where } Y_{\eta,i}:=\supp(\I_{\eta,i}/\I_{\eta,i+1}).
  \end{aligned}  
\end{equation}
Since $\I_i/\I_{i+1}$ is flat over $\Delta$ by \eqref{E:flat-filt} with generic fiber $\I_{\eta,i}/\I_{\eta,i+1}$ and special fiber $\I_{o,i}/\I_{o,i+1}$, it follows that 
\begin{equation}\label{E:2filt}
Y_{o,i}=\ov{Y_{\eta, i}}\cap X_o\subseteq X \text{ and } W_{o,i}=\ov{W_{\eta, i}}\cap X_o\subseteq X.
\end{equation}

%FILTRATIONS together 
%(resp. \eqref{E:filt-o}) as in the proof of Proposition \ref{P:iso-spec}:
%\begin{equation}\label{E:filtX}
%  \begin{aligned}
%      & Y_{\eta,\bullet}:=(Y_{\eta,0},\ldots, Y_{\eta,q}) \text{ where } Y_{\eta,i}:=\supp(\I_{\eta,i}/\I_{\eta,i+1}), \\
%       & Y_{o,\bullet}:=(Y_{o,0},\ldots, Y_{o,q}) \text{ where } Y_{o,i}:=\supp(\I_{i,o}/\I_{i+1,o}).
%  \end{aligned}  
%\end{equation}

\un{Claim 1}: $\I_{o,i}$ is $\s^o$-semistable on $W_{o,i}$ for every $0\leq i \leq q$.

Indeed, since $\I_{\eta}$ is $\s^{\ov \eta}$-semistable by assumption and $\I_{\eta,i}=\leftindex_{W_{\eta,i}}{(\I_{\eta})}$ is $\s^{\ov \eta}$-semistable by  \eqref{E:ssGreta} and Proposition \ref{P:ext-ss}, we deduce that Condition \eqref{P:res-ss4} of Proposition \ref{P:res-ss} holds for $\I_\eta$ and every connected component of  $W_{\eta,i}$. By the flatness of $\I$, the second condition in \eqref{E:2filt} and the  fact that $\s^{\ov \eta}=\xi^*(\s^o)$ for the specialization $\xi: \ov \eta\to o$ (see Definition \ref{D:VStabXS}), we deduce that Condition~\eqref{P:res-ss4} of Proposition~\ref{P:res-ss} holds for $\I_o$ and every connected component of  
$W_{o,i}$. Since $\I_o$ is $\s^o$-semistable by assumption, Proposition~\ref{P:res-ss} implies that $\I_{o,i}=\leftindex_{W_{o,i}}{(\I_{o})}$ is $\s^{o}$-semistable, q.e.d.

\un{Claim 2}: $\I_{o, i}/\I_{o, i+1}$ is $\s^o$-semistable on $Y_{o,i}$ for every $0\leq i \leq q$. 

Indeed, since $\I_{\eta,i}$ and $\I_{\eta,i}/\I_{\eta,i+1}=(\I_{\eta, i})_{Y_{\eta,i}}$ are $\s^{\ov \eta}$-semistable by  \eqref{E:ssGreta} and Proposition \ref{P:ext-ss}, we deduce that Condition \eqref{P:res-ss2} of Proposition \ref{P:res-ss} holds  for $\I_{\eta, i}$ and every connected component of  $Y_{\eta,i}$. By the flatness of $\I_{i}$, the first condition in \eqref{E:2filt} and the  fact that $\s^{\ov \eta}=\xi^*(\s^o)$ for the specialization $\xi: \ov \eta\to o$ (see Definition \ref{D:VStabXS}), we deduce that Condition~\eqref{P:res-ss2} of Proposition \ref{P:res-ss} holds  for $\I_{o,i}$ and every connected component of  $Y_{o,i}$. Since $\I_{o,i}$ is $\s^o$-semistable by Claim 1, Proposition \ref{P:res-ss} implies that $\I_{o,i}/\I_{o,i+1}=(\I_{o, i})_{Y_{o,i}}$ is $\s^{o}$-semistable, q.e.d.

\un{End of the proof}: We have already constructed a filtration as in \eqref{E:filt-Delta} satisfying \eqref{E:Icon}. In order to prove \eqref{E:IIcon}, observe first of all that $\Gr(\I_{\geq})$ is a relative rank-$1$ torsion-free sheaf on $X/\Delta$ because of \eqref{E:flat-filt}. 
Hence, Property \eqref{E:IIcon} amounts to check that 
\begin{equation}\label{E:IIIcon}
    \Gr(\I_{\geq})_o:=\bigoplus_{i=0}^q (\I_{i}/\I_{i+1})_o \in \ov\J_{X_o}(\s^o).
\end{equation}
Again by \eqref{E:IIIcon}, we have that 
\begin{equation}\label{E:Gr=}
   \Gr(\I_{\geq})_o=\Gr(\I_{\geq, o}):=\bigoplus_{i=0}^q \I_{i, o}/\I_{i+1, o}.
\end{equation}
We now conclude that $\Gr(\I_{\geq, o})$ is $\s^o$-semistable by using Claim 2 together with Proposition~\ref{P:ext-ss}.

%OLD FILTRATION
%and its associated ordered filtration (as in \ref{S:isospec}):
%\begin{equation}\label{E:filtXeta}
%  \begin{aligned}
%      & W_{\eta,\supseteq}: X_\eta=W_{\eta,0}\supseteq W_{\eta,1}\supseteq \ldots \supseteq  W_{\eta,q}\supseteq W_{\eta,q+1}=\emptyset \text{ where } W_{\eta,k}:=\supp(\I_{\eta,k}), \\
%      & Y_{\eta,\bullet}:=(Y_{\eta,0},\ldots, Y_{\eta,q}) \text{ where } Y_{\eta,k}:=\supp(\I_{\eta,k}/\I_{\eta,k+1}).
%  \end{aligned}  
%\end{equation}
%And similarly, consider  the decreasing filtration of $X_o$ by subcurves induced by \eqref{E:filt-o} and its associated ordered filtration:
%\begin{equation}\label{E:filtXo}
%  \begin{aligned}
%      & W_{o,\supseteq}: X_o=W_{o,0}\supseteq W_{o, 1}\supseteq \ldots \supseteq  W_{o,q}\supseteq W_{o,q+1}=\emptyset \text{ where } W_{o,k}:=\supp(\I_{\eta,k}), \\
%      & Y_{o,\bullet}:=(Y_{o,0},\ldots, Y_{o,q}) \text{ where } Y_{o,k}:=\supp(\I_{k,o}/\I_{k+1,o}).
%  \end{aligned}  
%\end{equation}
\end{proof}

We now prove that  $\J_X(\s)\to S$ is  S-complete (see \cite[Def. 3.38]{AHLH}),  which means the following. For any discrete valuation ring R with residue field $k$ and quotient field $K$, consider the test object $\ST_R:=[\Spec R[x,y]/(xy-t)/\Gm]$ where $t$ is a uniformizer of $R$ and the action of $\Gm$ is trivial on $R$ and it has weight $1$ on $x$ and weight $-1$ on $y$.
 The topological space of the stack $\ST_R$ has  four points which are related by the  specializations described in Figure \ref{F:STR}.

 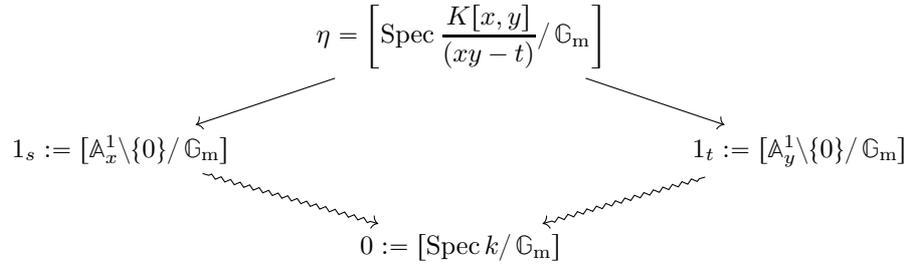
\begin{figure}[hbt!]
 \[
\begin{tikzcd}
    & \eta=\displaystyle\left[\Spec \frac{K[x,y]}{(xy-t)}/\Gm\right]\arrow[rd] \arrow[dl] \\
    1_s:=[\AA^1_x\setminus \{0\}/\Gm] \arrow[dr, rightsquigarrow] & & 
    1_t:=[\AA^1_y\setminus \{0\}/\Gm]  \arrow[dl, rightsquigarrow] \\
    & 0:=[\Spec k/\Gm]
\end{tikzcd}
\]
\caption{The four points of $\ST_R$ and their specializations: the solid arrows are ordinary specializations while the squiggly arrows are isotrivial specializations.}\label{F:STR}
\end{figure}

Then $\ov \J_{X/S}(\s)\to S$ is S-complete if any commutative diagram of solid arrows like the one in Figure~\ref{F:Scompl} below  can be completed with a dashed arrow:

\begin{figure}[hbt!]
 \[
\begin{tikzcd}
    \ST_R\setminus\{0\} \arrow[r, "f"] \arrow[d, hookrightarrow] &  \ov \J_{X/S}(\s)\arrow[d] \\
    \ST_R \arrow[r, "\phi"] \arrow[ru, dashrightarrow, "F"] & S
\end{tikzcd}
\]
\caption{S-completeness of $\ov \J_{X/S}(\s)\to S$.\label{F:Scompl}}
\end{figure}

\begin{proposition}\label{P:S-compl}
     Notation as in Theorem \ref{T:VcJ}. Then $\ov \J_{X/S}(\s)\to S$ is S-complete.
\end{proposition}
\begin{proof}
Consider a commutative diagram of solid arrows as in Figure \ref{F:Scompl}. 
We can assume that the residue field $k$ of $R$ is algebraically closed by \cite[Prop. 3.41]{AHLH} (using that $\ov \J_{X/S}(\s)\to S$ has affine diagonal by \cite[\href{https://stacks.math.columbia.edu/tag/0DLZ}{Lemma 0DLZ}]{stacks-project}). 

Since $S$ is an algebraic space, the morphism $\phi$ factors as $\phi: \ST_R\to \Spec R \xrightarrow{\ov{\phi}} S$. Therefore, we can base change the family $X\to S$ via the map $\ov{\phi}$ and hence  assume that we have a family $X\to S=\Spec R=\Delta$.  We will use the same notation as in Subsection~\ref{S:spec}. 

By the description of $\ST_R$ in Figure \ref{F:STR}, the morphism $f$ corresponds to two families $\I$ and $\J$ of  $\s$-semistable relative rank-$1$ torsion-free sheaves on $X/\Delta$ such that $\I(\eta)\cong \J(\eta)$.  Moreover, according to Proposition \ref{P:iso-spec} and using that the action of $\Gm$ on $R[x,y]/(xy-t)$ has opposite weights on $x$ and $y$, the required dashed arrow $F$ in Figure \ref{F:Scompl} corresponds to the choice of an ordered partition $Y_\bullet=(Y_0,\ldots, Y_q)$  such that 
\begin{equation}\label{E:GR=}
\Gr_{Y_\bullet}(\I(o))\cong \Gr_{\ov{Y_\bullet}}(\J(o))\in \ov \J_{X_o}(\s^o),
\end{equation}
where $\ov{Y_\bullet}=(Y_q,\ldots, Y_0)$ is the opposite order partition. 

From the proof of Theorem \ref{T:limits}, there exists  a non-increasing chain of subcurves 
\begin{equation}\label{E:chain}
X_o=:Z_0\supsetneq Z_1\supseteq Z_2\supseteq \ldots \supseteq Z_r\supsetneq Z_{r+1}:=\emptyset \text{ such that } \I \xrightarrow{\cong} \J^{\sum_i Z_i}.
\end{equation}
We will denote by $\lambda_i:\I\xrightarrow{\cong}\J^{\sum_j Z_j}\hookrightarrow  \J_i:=\J^{\sum_{1\leq j\leq i}Z_j}$ the successive morphisms (for any $0\leq i \leq r$). Recall that, by the construction in Theorem \ref{T:limits}, we have that 
\begin{equation}\label{E:Imlambda}
\Im \lambda_i(o)=\I(o)_{Z_{i+1}^c}.
\end{equation}

\vspace{0.1cm}

\un{Claim 1:} The sheaves $\I(o)_{Z_{i}^c}$ and $\J_i$ are $\s$-semistable (for any $0\leq i \leq r$).

Indeed, the Claim is true for $i=0$ because  $\I(o)_{Z_{0}^c}=0$ and  $\J_0=\J$ is $\s$-semistable by assumption. We now assume, by induction, that the Claim is true for $i-1$ and we prove it for $i\geq 1$. 
 Consider the morphism $\lambda_{i-1}(o):\I(o)\to \J_{i-1}(o)$ between $\s$-semistable sheaves. Proposition \ref{P:morf-ss} implies that $\Im(\lambda_{i-1}(o))$ and $\coker(\lambda_{i-1}(o))$ are $\s$-semistable. In particular, $\coker(\lambda_{i-1}(o))$ is torsion-free rank-$1$ and supported on a subcurve of $X_o$ and hence, using that $\Im(\lambda_{i-1}(o))=\I(o)_{Z_i^c}$, we deduce that 
\begin{equation}\label{E:Im-coker}
\begin{sis}
& \Im(\lambda_{i-1}(o))=\leftindex_{Z_i^c}{\J_{i-1}(o)}, \\
&  \coker(\lambda_{i-1}(o))=\J_{i-1}(o)_{Z_i}.
\end{sis}
\end{equation}
Lemma \ref{L:twist} implies that $\J_{i}(o)=\J_{i-1}^{Z_i}(o)$ sits in the following exact sequence
$$
0\to \J_{i-1}(o)_{Z_i}= \leftindex_{Z_i}{\J_i(o)}\to \J_i(o)\to \J_i(o)_{Z_i^c}= \leftindex_{Z_i^c}{\J_{i-1}(o)}\to 0.
$$
Therefore, $\J_i(o)$ is $\s$-semistable by Proposition \ref{P:ext-ss}, and the Claim is proved. 

%VECCHIO CLAIM: vero ma non  serve
%\vspace{0.1cm}

%\un{Claim 2:} If $Z_i=Z_{i+1}$ for a certain $1\leq i \leq r-1$, then $\J_i(o)= {\J_i(o)}_{Z_i}\oplus \J_i(o)_{Z_i^c}=\leftindex_{Z_i}\J_i(o)\oplus\leftindex_{Z_i^c}{\J_i(o)}$.

%Indeed, by applying \eqref{E:Im-coker} and our assumption, we get that 
%$$
%\Im(\lambda_{i}(o))=\I(o)_{Z_{i+1}^c}=\I(o)_{Z_i^c}=\Im(\lambda_{i-1}(o))=\leftindex_{Z_i^c}{\J_{i-1}(o)}.
%$$
%Therefore, since $\J_i=\J_{i-1}^{Z_i}$, the subsheaf $\Im(\lambda_{i}(o))\subset \J_i(o)$ provides a splitting of the exact sequence of Lemma  \ref{L:twist}
%$$
%0\to \leftindex_{Z_i}\J_i(o)\to \J_i(o) \to \J_i(o)_{Z_i^c}=\leftindex_{Z_i^c}{\J_{i-1}(o)}\to 0,
%$$
%which proves the Claim. 

\vspace{0.1cm}

Denote now by 
\begin{equation}\label{E:chain2}
W_{\supseteq}: X_o=W_0\supsetneq W_1\supsetneq W_2\supsetneq \ldots \supsetneq W_q\supsetneq W_{q+1}=\emptyset
\end{equation}
the unique decreasing subchain of \eqref{E:chain} consisting of distinct subcurves, and denote by $m_i$ (for any $1\leq i \leq q$) the number of times that $W_i$ appears in \eqref{E:chain} so that 
\begin{equation}\label{E:WZ}
W_i=Z_{m_1+\ldots + m_{i-1}+k} \text{ for any } 1\leq k \leq m_i.
\end{equation}
Consider the ordered partition of $X_o$ associated to the chain \eqref{E:chain2}:
\begin{equation}\label{E:ordpart}
Y_\bullet:=Y_\bullet(W_{\supseteq})=(Y_0,\ldots, Y_q) \text{ defined by } Y_i:=W_i- W_{i+1}.    
\end{equation}

We now prove that \eqref{E:GR=} holds for the ordered partition $Y_\bullet$ by combining  the following two claims.

\vspace{0.1cm}

\un{Claim 2:} We have that $\Gr_{Y_\bullet}(\I(o))\cong \Gr_{\ov{Y_\bullet}}(\J(o))$.

Consider the decreasing chain of subcurves of $X_o$ which is opposite to the chain $W_{\supseteq}$ of \eqref{E:chain2}:
\begin{equation}\label{E:chain3}
\ov{W_{\supseteq}}: X_o=W_{q+1}^c\supsetneq W_q^c\supsetneq W_{q-1}^c\supsetneq \ldots \supsetneq W_1^c\supsetneq W_{0}^c=\emptyset.
\end{equation}
Observe that the ordered partition $Y_{\bullet}(\ov{W_{\supseteq}})$ of $X_o$ associated to $\ov{W_{\supseteq}}$ is equal to $\ov{Y_{\bullet}}$. 

From \eqref{E:GrY}, it follows that 
\begin{equation}\label{E:2Gr}
  \begin{sis}
    & \Gr_{Y_\bullet}(\I(o))=\bigoplus_{i=0}^q (\leftindex_{W_i}{\I(o)})_{Y_i},\\ 
    & \Gr_{\ov{Y_\bullet}}(\J(o))=\bigoplus_{i=0}^q (\leftindex_{W_{i+1}^c}{\J(o)})_{Y_i}.
  \end{sis}  
\end{equation}
Fix now an index $0\leq i \leq q$. Applying \eqref{E:Imlambda} and \eqref{E:Im-coker} to $m_1+\ldots+m_i$ (which we set to be equal to $0$ if $i=0$) and using \eqref{E:WZ}, we get that 
\begin{equation}\label{E:IJ}
    \I(o)_{W_{i+1}^c}=\leftindex_{W_{i+1}^c}{\J_{m_1+\ldots+m_i}(o)}=\leftindex_{W_{i+1}^c}{\J^{m_1W_1+\ldots+m_iW_i}(o)}.
\end{equation}
This implies, using Lemma \ref{L:IYZ} and the fact that $Y_i=W_i-W_{i+1}$, that 
\begin{equation}\label{E:IJres}
     (\leftindex_{W_i}{\I(o)})_{Y_i}=\leftindex_{Y_i}{(\I(o)_{W_{i+1}^c})}=\leftindex_{Y_i}{(\leftindex_{W_{i+1}^c}{\J^{m_1W_1+\ldots+m_iW_i}(o)})}=\leftindex_{Y_i}{\J^{m_1W_1+\ldots+m_iW_i}(o)}.
\end{equation}
If $i=0$ this implies, using that $W_0=X_o$ and $W_1^c=Y_0$, that 
\begin{equation}\label{E:Gr=0}
    (\leftindex_{W_0}{\I(o)})_{Y_0}=\I(o)_{Y_0}=\leftindex_{Y_0}{\J(o)}=(\leftindex_{W_1^c}{\J(o)})_{Y_0}. 
\end{equation}
Assume from now on that $1\leq i \leq q$. Then,  using that $W_j=Y_j+\ldots+Y_q$ for any $1\leq j \leq q$ by \eqref{E:ordpart}, we compute in $\ZZ^{I(X_o)}$:
\begin{equation*}
m_1W_1+\ldots+m_iW_i=\sum_{1\leq h \leq i-1} (m_1+\ldots+m_{h})Y_h +(m_1+\ldots+m_i) \sum_{i\leq k \leq q} Y_k=    
\end{equation*}
\begin{equation}\label{E:eqsub}
=(m_1+\ldots+m_i-1)X_o+W_i-\sum_{1\leq h \leq i-1} (m_{h+1}+\ldots+m_{i}-1)Y_h.
\end{equation}
Consider now the two families of rank-$1$ torsion free sheaves on $X/\Delta$ defined by  
\begin{equation}\label{E:newsh}
   \wt{\J}:=\J^{(m_1+\ldots+m_i-1)X_o} \text{ and } \K:=\wt{\J}^{W_1}.
\end{equation}
Observe that 
\begin{equation}\label{E:newsh2}
  \begin{sis}
     & \wt{\J}\cong \J & \text{ by Lemma \ref{L:twist}\eqref{L:twist4}}, \\
     & \K=\J^{m_1W_1+\ldots+m_iW_i+\sum_{1\leq h \leq i-1} (m_{h+1}+\ldots+m_{i}-1)Y_h} & \text{ by \eqref{E:eqsub}}.
  \end{sis}  
\end{equation}
We now compute 
\begin{equation}\label{E:IJres2}
\begin{aligned}
 \leftindex_{Y_i}{\J^{m_1W_1+\ldots+m_iW_i}(o)}&= \leftindex_{Y_i}{\K(o)} 
& \text{ by \eqref{E:newsh2} and Lemma \ref{L:twist}\eqref{L:twist5},}\\
&= \leftindex_{Y_i}{\wt\J^{W_i}(o)} & \text{ by \eqref{E:newsh}}\\
&=\leftindex_{Y_i}{(\leftindex_{W_i}{\wt\J^{W_i}(o)})}=\leftindex_{Y_i}{(\wt\J(o)_{W_i})} & \text{ by Lemmas \ref{L:IYZ}\eqref{L:IYZ1} and  \ref{L:twist}\eqref{L:twist2}}\\
&= \leftindex_{Y_i}{(\J(o)_{W_i})} & \text{ by \eqref{E:newsh2}}\\
& = (\leftindex_{W_{i+1}^c}{\J(o)})_{Y_i} & \text{ by Lemma \ref{L:IYZ}\eqref{L:IYZ2},}
\end{aligned}
\end{equation}
where we have used that $Y_i=W_i-W_{i+1}=W_{i+1}^c-W_i^c$ in the last identity. 

The Claim now follows by using Formulas \eqref{E:2Gr}, and putting together \eqref{E:IJres} and \eqref{E:IJres2} for $1\leq i \leq q$ and \eqref{E:Gr=0} for $i=0$.

\vspace{0.1cm}

\un{Claim 3:} We have that $\Gr_{Y_\bullet}(\I(o))\in \ov \J_{X_o}(\s^o)$.

Indeed, using Formula \eqref{E:2Gr} and Proposition \ref{P:ext-ss}, it is enough to show that $(\leftindex_{W_i}{\I(o)})_{Y_i}$ is $\s$-semistable for any $0\leq i \leq q$. Using Lemma \ref{L:IYZ}\eqref{L:IYZ2},  we compute 
$$
(\leftindex_{W_i}{\I(o)})_{Y_i}=\leftindex_{Y_i}{(\I(o)_{W_{i+1}^c})}=\ker(\I(o)_{W_{i+1}^c}\twoheadrightarrow \I(o)_{W_i^c}).
$$
This, together with Proposition \ref{P:res-ss} and Claim 1, implies that $
(\leftindex_{W_i}{\I(o)})_{Y_i}$ is $\s$-semistable.    
\end{proof}

Finally, we prove that $\ov \J_{X/S}(\s)\to S$ satisfies the existence part of the valuative criterion for properness, i.e. for any discrete valuation ring $R$ with residue field $k$ and quotient field $K$, any commutative diagram of solid arrows as in Figure \ref{F:UnivClosed} can be completed with a dashed arrow:

\begin{figure}[hbt!]
 \[
\begin{tikzcd}
    \Spec K \arrow[r, "f"] \arrow[d, hookrightarrow] &  \ov \J_{X/S}(\s)\arrow[d] \\
    \Spec R \arrow[r, "\phi"] \arrow[ru, dashrightarrow, "F"] & S
\end{tikzcd}
\]
\caption{Existence part of the valuative criterion of properness for  $\ov \J_{X/S}(\s)\to S$.\label{F:UnivClosed}}
\end{figure}
We can base change the family $X/S$ along the map $\phi$, and assume that the family $X\to S=\Spec R=\Delta$ is defined over a DVR. We will use the same notation as in Subsection~\ref{S:spec}.

Let $\I$ be a torsion free rank $1$ sheaf on $X/\Delta$, and let $\s$ be a relative V-stability condition. Consider the function 
\begin{align*}
    \beta_{\I(o)}\colon\{\textup{Subcurves of }X_o\}&\to\ZZ\\
    Z&\mapsto \chi(\I(o)_Z)-\s^o_Z,
\end{align*}
where $\s^o$ is the extended $V$-function associated to the stability condition $\s^o$ on the central fiber $X_o$. Clearly, the sheaf $\I$ is $\s$-semistable if and only if $\I(o)$ is $\s^o$-semistable (by Lemma-Definition \ref{LD:VStab}) if and only if $\beta_{\I(o)}(Z)\geq 0$ for each $Z\in \BCon(X_o)$.

\begin{comment}

Let $\I$ be a torsion free rank $1$ sheaf on $X/S$, and let $\s$ be a relative V-stability condition. For any geometric point $s$ of $S$, we define the function
\begin{align*}
    \beta^{\s}_{\I(s)}=\beta_{\I(s)}\colon\{\textup{Subcurves of }X_s\}&\to\ZZ\\
    Z&\mapsto \chi(\I(s)_Z)-\s^s_Z,
\end{align*}
where $\s^s$ is the extended $V$-function associated to the stability condition $\s^s$. Clearly, the sheaf $\I(s)$ is $\s^s$-semistable if and only if $\beta_\I(s)(Z)\geq0$ for each $Z\in \BCon(X_s)$.

We will prove that the open substack $\ov \J_{X/S}(\s)\to S$ of $\TF^\chi_{X/S}$ is universally closed, by checking that it satisfies the existence part of the valuative criterion for properness. In other words, for any discrete valuation ring $R$ with residue field $k$ and quotient field $K$, any commutative diagram of solid arrows as in Figure \ref{F:UnivClosed} can be completed with a dashed arrow:

\begin{figure}[hbt!]
 \[
\begin{tikzcd}
    \Spec K \arrow[r, "f"] \arrow[d, hookrightarrow] &  \ov \J_{X/S}(\s)\arrow[d] \\
    \Spec R \arrow[r, "\phi"] \arrow[ru, dashrightarrow, "F"] & S
\end{tikzcd}
\]
\caption{Universal closedness of $\ov \J_{X/S}(\s)\to S$.\label{F:UnivClosed}}
\end{figure}

We can base change the family $X/S$ along the map $\phi$, and assume that the family $X\to S=\Spec R=\Delta$ is defined over a DVR. We will use the same notation as in \ref{S:spec}.

\end{comment}

\begin{definition}\label{LD:twistingsubcurve}
Let $\I$ be a torsion free rank $1$ sheaf on $X/\Delta$, and let $\s$ be a relative V-stability condition.
We define, inductively in $k \in \NN$, a sheaf $I_k$ and a subcurve $Y_k \subseteq X_o$, as follows. 

Firstly, let $I_0:=\I(o)$, and let $Y_0$ be any maximal subcurve that realizes the equality
\begin{equation*}
    \beta_{I_0}(Y_0)= \min_{Z \subseteq X} \beta_{I_0}(Z).
\end{equation*}

Assume that $Y_{k-1}$ has been defined, and set $I_{k}:=\I^{Y_{k-1}}(o)$. 
If 
$$\min_{Z \subseteq X} \beta_{I_k}(Z)\leq\min_{Z \subseteq X} \beta_{I_0}(Z),$$ 
and there exists a %maximal 
subcurve $W\subseteq X$ such that $\beta_{I_k}(W)=\min_{Z \subseteq X} \beta_{I_k}(Z)$, with $W\nsubseteq Y_{k-1}$, then we define $$ Y_{k} := Y_{k-1}  \cup W. $$
Otherwise, we set $Y_k=Y_{k-1}$.

 We define $Y(\mathcal{I}(o), \s^o)$ to equal $\bigcup_{k = 0}^\infty Y_k$.
\end{definition}

\begin{remark}\label{R:twistingsubcurve}
    We observe that the above sequence $(Y_k)_{k=0}^{\infty}$ satisfies the following.
    
    \begin{enumerate}
    \item There exists $k$ such that $Y_{k}=Y_{k-1}$, since $\{Y_k\}_{k=0}^{\infty}$ is a non-decreasing sequence of subcurves of $X_o$.
        \item  We have that $Y_{k}=Y_{k-1}$ if and only if the following condition holds for the sheaf $I_k$: either $\min_{Z \subseteq X} \beta_{I_k}(Z)>\min_{Z \subseteq X} \beta_{I_0}(Z)$ or all the subcurves $W\subseteq X_o$ such that  $\beta_{I_k}(W)=\min_{Z \subseteq X} \beta_{I_k}(Z)$ are contained in $Y_{k-1}$. 
       If the above condition is satisfied, then  $I_{k+1} = I_k$, so the same condition is verified for the sheaf $I_{k+1}$ and we deduce that $Y_{k+1}=Y_k$. Iterating the argument, we deduce that  $Y_j=Y_{k-1}$ for all $j \geq k-1$ and hence $Y(\mathcal{I}(o), \s^o)=Y_{k-1}$.
    %$Y_j=Y_{k-1}$ for all $j \geq k-1$. This follows, inductively, from the fact that $I_{k+1} = I_k$, hence $\beta_{I_{k+1}}=\beta_{I_k}$, which implies that $Y_{k+1}=Y_{k}$ and so on. 
    \end{enumerate}
\end{remark}

\begin{lemma}\label{L:twistY}
    Let $\I$ be a torsion free rank $1$ sheaf on $X/\Delta$, and let $\s$ be a relative $V$-stability. Let $Y_0$ be any maximal subcurve of $X_o$ such that $\beta_{I_0}(Y_0)= \min_{Z \subseteq X} \beta_{I_0}(Z)$
    and construct a curve $Y_0\subseteq Y:=Y(\mathcal{I}(o), \s^o)\subseteq X_o$ as in Definition \ref{LD:twistingsubcurve}. 
    
   Then, for each subcurve $Z\subseteq X_o$ we have the inequality $\beta_{\I^Y(o)}(Z)\geq\beta_{\I(o)}(Y_0)$, and the equality holds only if $Z\subseteq Y$. 

    Moreover, if $\beta_{\I^Y(o)}(Y)=\beta_{\I(o)}(Y_0)$, then $Y=Y_0$ and $\I^Y(o)=\I^Y(o)_Y\oplus\I^Y(o)_{Y^\mathsf{c}}$.
\end{lemma}
\begin{proof} 
    Let $W\subseteq X$ be any %maximal 
    subcurve of $X_o$ such that $\beta_{\I^Y(o)}(W)= \min_{Z \subseteq X} \beta_{\I^Y(o)}(Z)$. Then by Remark~\ref{R:twistingsubcurve}
    %construction of $Y=Y(\mathcal{I}(o), \s^o)$, 
    we have that 
    \begin{equation}\label{E:UnivClosed1}
    \text{either } \quad  \beta_{\I^Y(o)}(W)>\beta_{\I(o)}(Y_0) \quad \text{ or }\quad 
    \beta_{\I^Y(o)}(W)\leq\beta_{\I(o)}(Y_0)\textup{ and }W\subseteq Y.
    \end{equation}
   % or
   % \begin{equation}\label{E:UnivClosed2}
   % \beta_{\I^Y(o)}(W)\leq\beta_{\I(o)}(Y_0)\textup{ and }W\subseteq Y.
   % \end{equation}

    On the other hand, for any subcurve  $V\subseteq Y$,  by restricting the exact sequence (see Lemma~\ref{L:twist}\eqref{L:twist2})
    $$
    0\to \I(o)_Y\to\I^Y(o)\to\I^Y(o)_{Y^\mathsf{c}}\to 0
    $$
   to $Y^\mathsf{c}\cup V$, we obtain the exact sequence
    $$
    0\to \I(o)_V\to\I^Y(o)_{Y^\mathsf{c}\cup V}\to\I^Y(o)_{Y^\mathsf{c}}\to 0.
    $$
   Then Lemma~\ref{L:IY}\eqref{L:IY1} gives
     \begin{equation}\label{E:UnivClosed2}
    \chi(\I^Y(o)_V)\geq\chi(\I(o)_V) \text{ with equality if and only if } \I^Y(o)_{Y^\mathsf{c}\cup V}=\I^Y(o)_{Y^\mathsf{c}}\oplus \I^Y(o)_{V}.
    \end{equation}
    From this, we deduce that  
    \begin{equation}\label{E:UnivClosed3}
    V\subseteq Y \Rightarrow \beta_{\I^Y(o)}(V)\geq\beta_{\I(o)}(V)\geq\beta_{\I(o)}(Y_0).
    \end{equation}
    %We address the second case as follows. 
    %From Lemma~\ref{L:twist}(\ref{L:twist2}), there exists an exact sequence
    %$$
    %0\to \I(o)_Y\to\I^Y(o)\to\I^Y(o)_{Y^\mathsf{c}}\to0.
    %$$
    %Restricting it to $Y^\mathsf{c}\cup W$, we obtain the exact sequence
    %$$
    %0\to \I(o)_W\to\I^Y(o)_{Y^\mathsf{c}\cup W}\to\I^Y(o)_{Y^\mathsf{c}}\to0,
    %$$
    %which, combined with Lemma~\ref{L:IY}(\ref{L:IY1}), gives
    %$$
    %\chi(\I^Y(o)_W)\geq\chi(\I(o)_W),
    %$$
    %and, subsequently,
    %\begin{equation}\label{E:UnivClosed3}
    %\beta_{\I^Y(o)}(W)\geq\beta_{\I(o)}(W)\geq\beta_{\I(o)}(Y_0).
    %\end{equation}
    Therefore, by \eqref{E:UnivClosed1} and \eqref{E:UnivClosed3}, for any subcurve $Z\subseteq X_o$, we have 
    $$
    \beta_{\I^Y(o)}(Z)\geq\beta_{\I^Y(o)}(W)\geq\beta_{\I(o)}(Y_0) \text{ with equality only if  } Z\subseteq Y,
    $$
    which proves the first claim. 
    
   % Moreover, by \eqref{E:UnivClosed2}, if the equality holds, then 
   % $$
   %Z\subseteq W\subseteq Y,
   % $$
   % by construction of $Y=Y(\I(o),\s^o)$.
    
    %We, now, prove the last statement.
    %Since $Y_0$ is a minimum point for $\beta_{\I(o)}$, if  $\beta_{\I^Y(o)}(Y)=\beta_{\I(o)}(Y_0)$, then, again by combining Lemma~\ref{L:twist}(\ref{L:twist2}) and Lemma~\ref{L:IY}(\ref{L:IY1}), we have that
    %$$
    %\beta_{\I^Y(o)}(Y)=\beta_{\I(o)}(Y)=\beta_{\I(o)}(Y_0).
    %$$
    %By maximality of $Y_0$, the second equality shows that $Y=Y_0$, while the first gives $$\I^Y(o)=\I^Y(o)_Y\oplus\I^Y(o)_{Y^\mathsf{c}},$$ since 
    %$$
    %\chi(\I^Y(o)_Y)\geq\chi(\I(o)_Y),
    %$$
    %with equality if and only if the sheaf $\I^Y(o)$ decomposes at $Y$.

     We now prove the last statement. If  $\beta_{\I^Y(o)}(Y)=\beta_{\I(o)}(Y_0)$, then \eqref{E:UnivClosed3} implies that 
    $$
    \beta_{\I^Y(o)}(Y)=\beta_{\I(o)}(Y)=\beta_{\I(o)}(Y_0),
    $$
    while \eqref{E:UnivClosed3} implies that  $$\I^Y(o)=\I^Y(o)_Y\oplus\I^Y(o)_{Y^\mathsf{c}},$$
    and we are done.
\end{proof}

%\begin{proposition}
 %   Let $\pi:X\to S$ be a family of connected reduced curves over a  quasi-separated and locally Noetherian algebraic space $S$ and let $\s=\{\s^s\}$ be a V-stability condition on $X/S$. Then $\ov \J_{X/S}(\s)$ is universally closed.
%\end{proposition}
\begin{proposition}\label{P:univ-cl}
    Notation as in Theorem \ref{T:VcJ}. Then $\ov \J_{X/S}(\s)\to S$ satisfies the existence part of the valuative criterion of properness. 
\end{proposition}
\begin{proof}
    We follow the proof strategy of Part~3 of \cite[Theorem~32]{esteves}.
    
    Let $I_\eta$ be a $\s^\eta$-semistable torsion-free rank $1$ sheaf on $X_\eta$, and let $\I$ be an extension of $I_\eta$ to an element of $\TF^{|s|}_{X/\Delta}$, which exists by \cite[Lemma 7.8(i)]{altmankleiman}. Consider the infinite filtration
    $$
    \ldots\subseteq\I^i\subseteq\ldots\subseteq\I^0:=\I,
    $$
    where the quotients are 
    $$
    \frac{\I^i}{\I^{i+1}}=\I^i(o)_{Y^i},
    $$
    and the subcurves $Y^i$ are defined as follows. For each $i\in \mathbb{N}$, we let $Y_0^i$ be a maximal subcurve such that $\beta_{\I^i(o)}(Y_0^i)= \min_{Z \subseteq X_o} \beta_{\I^i(o)}(Z)$, and  then let $Y^i:=Y(\I^i(o),\s^0)$, as produced by Definition~\ref{LD:twistingsubcurve}.
    
    We claim that there exists $i$ such that $\I^i(o)$ is $\s^o$-semistable. Suppose, by contradiction, that the claim is false. Then $\beta_{\I^i(o)}(Y_0^i)<0$ for each $i \in \mathbb{N}$ and, by Lemma~\ref{L:twistY}, up to replacing $\I$ with $\I^j$ for some $j \in \mathbb{N}$, we may assume that $Y:=Y^i=Y_0^i$ does not depend on $i$ and that $\I^i(o)$ decomposes at $Y$ for each $i \in \mathbb{N}$. By \cite[Lemma~24]{esteves}, there exists an $S$-flat quotient $\F$ of $\I$ such that $\F(o)=\I(o)_Y$. By flatness, since $\chi(\F(o))=\chi(\I(o)_Y)<\s^o$, we also have $\chi(\F(\eta))<\s^\eta$; hence $\F(\eta)$ is not $s^\eta$-semistable. By Proposition~\ref{P:morf-ss}  we deduce that $\I(\eta)\cong\I_\eta$ is also not $\s^\eta$-semistable, a contradiction that proves our claim.
\end{proof}

\begin{remark}
   Note that we have not used Condition~(2)  of Definition~\ref{D:VStabX} in the proof of  Proposition~\ref{P:univ-cl}.  Thus  the  substack $\ov \J_{X/S}(\s)$ of $\TF^\chi_{X/S}$ satisfies the existence part of the valuative criterion of properness under the weaker assumption that $\s=\{\s^s\}$ is \emph{any} collection of functions $\s^s \colon \BCon(X_s) \to \mathbb{Z}$ such that (a) $\s^s_Y+(s^s)_Y^\mathsf{c} - |\s^s| \in \{0,1\}$ for all $Y \in \BCon(X_s)$; and (b) for every \'etale specialization $\xi: s \rightsquigarrow t$ of geometric points of $S$, we have $\xi^*(\s^t)=\s^s$. (See also Remark~\ref{R:open-moregeneral}). 
    %Proposition~\ref{P:univ-cl} holds with a weaker hypothesis. Our proof shows that, given a collection $\s=\{\s^s\}$  of maps $\s^s:\BCon(X_s)\to\ZZ$ for each geometric point $s$ of $S$ that satisfies Condition~\ref{D:VStabX}\eqref{E:condi1}, the associated open substack $\ov \J_{X/S}(\s)\to S$ of $\TF^\chi_{X/S}$ (see Remark~\ref{R:open-moregeneral}) is universally closed.
\end{remark}
The following generalizes the fact, established by Esteves \cite{esteves}, that $\Simp_{X/S} \subseteq \tf_{X/S}$ satisfies the existence part of the valuative criterion of properness.
\begin{remark} 
Let $X\to \Delta=\Spec R$ be a family of connected reduced curves over a DVR and let $\s$ be a relative V-stability condition for $X/\Delta$. Given a sheaf $\I_{\ov \eta}\in\ov\J_{X_{\ov \eta}}(\s^{\ov \eta})$ (resp. $\I_{\ov \eta}\in\TF_{X_{\ov \eta}}$)  whose decomposition into indecomposable subsheaves is 
    $$
\I_{\ov \eta}=\bigoplus_{i=1}^k(\I_{\ov \eta})_{W_i},
    $$
then there exists an extension $\I\in\ov\J_{X/\Delta}(\s)$ of $\I_{\ov \eta}$ (resp. $\I\in \TF_{X/\Delta}$) such that the special fiber has the following decomposition into indecomposable subsheaves
    $$
\I(o)=\bigoplus_{i=1}^k\I(o)_{\xi_*(W_i)},
    $$
    where $\xi:\ov \eta\rightsquigarrow o$ is the \'etale specialization from the geometric generic point to the special point of $\Delta$.

    This follows from Parts $(2)$ and $(3)$ of the proof of \cite[Theorem~32]{esteves}, applied to the indecomposable subsheaves $\I(t)_{W_i}$ of $\I(t)$.
\end{remark}

We can now prove the main result of this section. 

\begin{proof}[Proof of Theorem \ref{T:VcJ}]
We apply the necessary and sufficient conditions for the existence of proper good moduli spaces in \cite{AHLH}. Note that:
\begin{itemize}
\item The stack $\TF_{X/S}$ (and hence also its open substack $\ov \J_{X/S}(\s)$) is locally of finite type and it has affine diagonal over $S$ by \cite[\href{https://stacks.math.columbia.edu/tag/0DLY}{Lemma 0DLY}]{stacks-project} and \cite[\href{https://stacks.math.columbia.edu/tag/0DLZ}{Lemma 0DLZ}]{stacks-project}.
    \item The morphism $\ov \J_{X/S}(\s)\to S$ is of finite type (and hence of finite presentation) by the same argument of \cite[p.~3071-3072]{esteves}: upon restricting to a Noetherian open subspace of $S$ and fixing a relatively ample line bundle $\O_X(1)$ on $X/S$, one shows (using Lemma-Definition~\ref{LD:VStab}) that there exists an integer $m(\s)\gg 0$ such that $\ov \J_{X/S}(\s)$ is contained in the open substack of $m(\s)$-regular sheaves of $\TF_{X/S}^{|\s|}$. The latter open substack is of finite type over $S$ by \cite[Lemma~7.3 and Thm.~7.4, pp.~98-99]{altmankleiman}. 
    \item The stabilizers of all geometric points of $\ov\J_{X/S}(\s)$ are isomorphic to $\mathbb{G}_m^r$ for some $r\geq 0$ (see~\eqref{E:AutI}), and hence they are linearly reductive.
    \item $\ov \J_{X/S}(\s)\to S$ is $\Theta$-complete by Proposition \ref{P:Th-compl}.
    \item $\ov \J_{X/S}(\s)\to S$ is S-complete by Proposition \ref{P:S-compl}.
    \item $\ov \J_{X/S}(\s)\to S$ satisfies the existence part of the vaulative criterion of properness by Proposition~\ref{P:univ-cl}.
\end{itemize}
Therefore, we can apply \cite[Theorems A, 5.2 and 5.4]{AHLH} and deduce that $\ov \J_{X/S}(\s)\to S$ admits a relative proper good moduli over $S$. 
\end{proof}

We  now describe the closed points of a V-compactified Jacobian space and the map from a V-compactified Jacobian stack to its good moduli space in terms of the isotrivial specialization   s of Proposition~\ref{P:iso-poly}.

\begin{proposition}\label{P:pt-VcJspa}
 Let $X$ be a connected reduced curve over $k=\ov k$ and let $\s$ be a V-stability condition on $X$. Consider the good moduli space morphism $\Xi_X(\s):\ov \J_X(\s)\to \ov J_X(\s)$. Then
 \begin{enumerate}[(i)]
     \item we have a bijection 
     $$
     \left\{\s\text{-polystable sheaves on } X\right\}\xrightarrow[\cong]{\Xi_X(\s)} \ov J_X(\s)(k).
     $$
     \item the map $\Xi_X(\s)$ on $k$-points sends an $\s$-semistable sheaf $I$ on $X$ onto $\Xi_X(\s)(I)=\Xi_X(\s)(I^{pl})$, where $I^{pl}$ is the unique polystable isotrival specialization of $I$, as defined in Proposition~\ref{P:iso-poly}. 
     \item the restriction of the map $\Xi_X(\s)$ to the open locus $\ov \J_X(\s)^{st}$ parameterizing $\s$-stable sheaves is the $\Gm$-rigidification. 
 \end{enumerate}
\end{proposition}
\begin{proof}
Proposition \ref{P:iso-poly} implies that the closed $k$-points of $\ov \J_X(\s)$ are the $\s$-polystable sheaves on $X$. Hence, the result follows from \cite[Thm. 4.16(iv)]{Alp}.
\end{proof}

\begin{corollary}\label{C:VcJgeneral}
Let $\pi:X\to S$ be a family of connected reduced curves over a  quasi-separated and locally Noetherian algebraic space $S$ and let $\s=\{\s^s\}$ be a \emph{general} V-stability condition on $X/S$. Then $\ov J_{X/S}(\s)=\ov \J_{X/S}(\s)\fatslash \Gm$ and $\Xi_{X/S}(\s): \ov \J_{X/S}(\s)\to \ov J_{X/S}(\s)$ is a tame (and hence also a coarse) moduli space morphism. 
\end{corollary}
\begin{proof}
 This follows by combining Proposition \ref{P:pt-VcJspa} and Remark  \ref{R:gen-s}, and using \cite[Def. 7.1]{Alp}.
\end{proof}

\section{Projectivity and Schematicity of compactified Jacobian spaces}\label{sub:scheme}

Theorem \ref{T:VcJ} provides, for any $\s\in \VStab^{\chi}(X/S)$, a compactified Jacobian  space $\ov J_{X/S}(\s)$, which is a proper algebraic space over $S$. It is natural to ask whether these compactified Jacobian spaces are projective, or at least schematic, over $S$. In this section we provide a partial answer to these questions.

First of all, we show how the results of Esteves \cite{esteves} imply that  \emph{classical} compactified Jacobians are locally projective, under same natural additional assumptions on the family $X/S$.

\begin{theorem}\label{T:projective}
Let $\pi:X\to S$ be a family of connected reduced curves over a quasi-separated and locally Noetherian  algebraic space  $S$. 
Let $E$ be a vector bundle on $X/S$ with constant integral relative slope. 
%$\mu_{X/S}(E):=\mu(E_{|X_s})\in \ZZ$ and let $\psi(E)$ be its associated numerical polarization on $X/S$ as in Example \ref{Ex:VBstab}\eqref{Ex:VBstab2}.

Then the associated compactified Jacobian space $\ov J_{X/S}(E)$ (see Example \ref{Ex:classcJ}) is locally projective over $S$.
\end{theorem}
When the compactified Jacobian space $\ov J_{X/S}(E)$ is  \emph{fine}, the above result follows immediately from \cite[Thm.~C]{esteves}. We now show how to adapt the argument of loc. cit. to deal with the non-fine case.
\begin{proof}
 Consider the fiber product $X\times_S \ov \J_{X/S}(E)$ and denote by $p_1$ and $p_2$ the two projections. For any vector bundle $F$ on $X/S$ of relative slope $\mu_{X/S}(F)=\mu_{X/S}(E)$, consider the \emph{theta line bundle} associated to $F$:
$$
\L_F:=\D_{p_2}(\I\otimes p_1^*(F))\in \Pic(\ov \J_{X/S}(E))
$$
where $\D_{p_2}$ denotes the determinant of cohomology with respect to $p_2$ and $\I$ denotes the universal sheaf on $X\times_S \ov \J_{X/S}(E)$. Since $\mu_{X/S}(F)=\mu_{X/S}(E)$,  the sheaf $\I\otimes p_1^*(F)$ has $p_2$-relative Euler characteristic equal to $0$, and hence $\mathcal{L}_F$ comes with a canonical section $\theta_F\in H^0(\ov \J_{X/S}(E), \L_F)$, called the \emph{theta function} associated to $F$. The zero locus of $\theta_F$ is equal to 
$$
Z(\theta_F)=\{[I]\in \ov \J_{X_s}(E_s): h^0(X_{\ov s},I\otimes F_{X_{\ov s}})=h^1(X_{\ov s},I\otimes F_{X_{\ov s}})\neq 0\},
$$
where $I$ is a $\psi(E_{X_{\ov s}})$-semistable rank-one torsion free sheaf on the geometric fiber $X_{\ov s}$ above $s$. 
We refer to \cite[Sec. 6.2]{esteves} for details. 

It is proved in \cite[p. 3083]{esteves} that if the vector bundle $F$ satisfies the following conditions 
\begin{equation}\tag{*}
    \begin{sis}
    & \rk_{X/S}(F)=m\rk_{X/S}(E), \\
    & \det(F)\cong \det(E)^{\otimes m},\\
    & \D_{\pi}(F)\cong \D_{\pi}(E)^{\otimes m},
    \end{sis}
    \quad \text{ for some } m\geq 1,
\end{equation}
then there exists an isomorphism $\L_F\cong \L_E^{\otimes m}$ (uniquely determined  up to an invertible regular function on $S$), by which we can interpret $\theta_F$ as a section of $\L_E^{\otimes m}$ (up to an invertible regular function on $S$).

We now show that $\L_E$ descends, along the good moduli space morphism $\Xi: \ov \J_{X/S}(E)\to \ov J_{X/S}(E)$, to an $S$-ample line bundle on $\ov J_{X/S}(E)$. This will complete the proof. We split the argument in steps.

\un{Step 1:} The line bundle $\L_E$ descends to a (unique) line bundle $L_E$ on $\ov J_{X/S}(E)$ such that $\Xi^*:H^0(\ov J_{X/S}(E),L_E^{\otimes m})\xrightarrow{\cong} H^0(\ov \J_{X/S}(E),\L_E^{\otimes m})$ for any $m\geq 1$. 

Indeed, the second assertion follows from the first one and the fact that $\Xi_*\O_{\ov \J_{X/S}(E)}=\O_{\ov J_{X/S}(E)}$ by definition of a good moduli space. In order to show that the line bundle $\L_E$ on $\ov \J_{X/S}(E)$ descends to a line bundle $L_E$ on $\ov J_{X/S}(E)$ (which is necessarily unique since $\Xi_*(\L_E)=L_E$), it is enough by \cite[Thm.~10.3]{Alp} to show that $\L_E$ has trivial stabilizer action at all closed points. Fix one such closed point $[I]\in \ov \J_{X_{\ov s}}(E_{X_{\ov s}})$ lying over a geometric point $\ov s$ of $S$. 
The canonical decomposition \eqref{E:decI} of $I$
$$I=I_{Y_1}\oplus \ldots \oplus I_{Y_r}, $$
is such that each sheaf $I_{Y_i}$ is a $\psi(E_{\ov s})$-stable (and hence simple) rank-one torsion-free sheaf on the connected subcurve $Y_i$ (for $1\leq i \leq r$). 
The automorphism group of $I$ is equal to $\Aut(I)=\Gm^r$, where the $i$-th copy of $\Gm$ acts via scalar multiplication on the $i$-th factor of the decomposition of $I$. 
By the definition of $\L_E$ via the determinant of cohomology, an element $g=(g_i)_{i=1}^r\in \Aut(I)$ acts on the fiber of $\L_E$ at $I$ via multiplication by 
$$
\prod_{i=1}^r g_i^{\chi(E_{Y_i}\otimes I_{Y_i})}. 
$$
Since $I_{Y_i}$ is  $\psi(E_{\ov s})$-stable, we have  that $\chi(I_{Y_i})=-\mu(E_{Y_i})$ and then we compute (using Hirzebruch-Riemann-Roch)
$$\chi(E_{Y_i}\otimes I_{Y_i})=\rk(E_{Y_i})[\mu(E_{Y_i})+\chi(I_{Y_i})]=0.$$
Therefore, the action of $\Aut(I)$ on the fiber of $\L_E$ at $I$ is trivial, which shows that $\L_E$ has trivial stabilizer action at all closed points and we are done. 
  
\un{Step 2:} The line bundle $L_E^m$ is $S$-relatively globally generated for $m\gg 0$.

Indeed, fix a point $[I]$ of $\ov \J_{X/S}(E)$ lying over $s\in S$.
%which is the image of a sheaf $I\in \ov \J_{X_{\ov s}}(E_{\ov s})$ where $\ov s$ is a geometric point lying over $s$. 
By \cite[Thm. 19, Lemma 46]{esteves}, there exists, up to passing to an \'etale neighboorhood of $s\in S$, a vector bundle $F$ on $X$ satisfying conditions (*) for some $m\geq 1$ and such that 
$$
h^0(X_{\ov s}, I\otimes F_{X_{\ov s}})=h^1(X_{\ov s}, I\otimes F_{X_{\ov s}})=0.
$$
This implies that $\theta_F(s)\neq 0$, and we are done.  

\un{Step 3:} The line bundle $L_E^m$ separates points relative to $S$ for $m\gg 0$.

Indeed, fix two points $[I_1], [I_2]$ of $\ov \J_{X/S}(E)$ lying over $s\in S$ such that $\Xi([I_1])\neq \Xi([I_2]$.
By combining \cite[Lemma 10]{Est-sep} with \cite[Lemmas 17 and 46]{esteves}, there exists, up to passing to an \'etale neighboorhood of $s\in S$, a vector bundle $F$ on $X$ satisfying conditions (*) for some $m\geq 1$ and such that 
$$
\begin{aligned}
& h^0(X_{\ov s}, I_1\otimes F_{X_{\ov s}})=h^1(X_{\ov s}, I_1\otimes F_{X_{\ov s}})=0,\\
&  h^0(X_{\ov s}, I_2\otimes F_{X_{\ov s}})=h^1(X_{\ov s}, I_2\otimes F_{X_{\ov s}})\neq 0.\\
\end{aligned}
$$
This implies that $\theta_F([I_1])\neq 0$ and $\theta_F([I_2])=0$, and we are done.

\un{Step 4:} The line bundle $L_E$ is $S$-relatively ample. 

Indeed, consider the locally projective morphism 
$$
\Proj \bigoplus_{n\geq 0} f_*(L_E^n)\to S
$$
where $f=f(\psi(E)):\ov J_{X/S}(E)\to S$, together with its canonical $S$-relatively ample line bundle $\O(1)$. By Step~2, we have a regular $S$-morphism 
$$
H:\ov J_{X/S}(E)\to \Proj \bigoplus_{n\geq 0} f_*(L_E^n),
$$
such that $H^*\O(1)=L_E$.
Step 3 implies that $H$ is quasi-finite, and hence finite since the domain and codomain are proper over $S$. We conclude that $L_E=H^*\O(1)$ is $S$-relatively ample.    
\end{proof}
Under some additional assumptions, we can extend the conclusion of Theorem~\ref{T:projective} to all classical (as defined in Example~\ref{Ex:classcJ}) compactified Jacobian spaces.
\begin{corollary}\label{C:projective}
Let $\pi:X\to S$ be a family of connected reduced curves over a quasi-separated and locally Noetherian  algebraic space  $S$. 
Assume that Condition (*) of Theorem \ref{ThmB} holds true for $X/S$.

%Old Assumptions
%\begin{enumerate}[(i)]
%    \item For any $s\in S$, any irreducible component of the fiber $X_s:=X\times_S s$ is geometrically integral.
%    \item There are sections $\{\sigma_i:S\to X\}_{i=1}^n$ of $\pi$ factoring through the smooth locus of $\pi$ such that, for any $s\in S$, any irreducible component of the fiber $X_s$ contains $\sigma_i(s)$ for some $i$. 
%\end{enumerate}
Then any classical compactified Jacobian space $\ov J_{X/S}(\psi)$ of $X/S$ is locally projective over $S$.    

In particular, if $X$ is a connected reduced curve over $k=\ov k$, then  any classical compactified Jacobian space $\ov J_X(\psi)$ is projective.
\end{corollary}
Observe that Condition (*) implies that, for any $s\in S$, any irreducible component of the fiber $X_s$ is geometrically integral by \cite[\href{https://stacks.math.columbia.edu/tag/0CDW}{Tag 0CDW}]{stacks-project}.
\begin{proof}
Let $s_o\in S$ and let $s$ be a geometric point lying over $s_o$. Up to shrinking $S$ around $s_o$, by the discussion after Lemma-Definition \ref{LD:numpol}, we can perturb the relative numerical polarization $\psi$ without changing the associated compactified Jacobians stack/space so that $\psi^{s}$ takes rational values. Then, using the two assumptions on the family and arguing as in \cite[Lemma 35]{esteves}, we can find a vector bundle $E$ on $X/S$ with constant integral slope such that $\psi^s=\psi(E)^s$. Up to shrinking again the family $S$ around $s_o$, we can assume that $\psi=\psi(E)$. Then the conclusion follows from Theorem \ref{T:projective}.
\end{proof}

The above Theorem prompts the following natural

\begin{question}
 Is \emph{any} compactified Jacobian space $\ov J^\chi_{X/S}$ locally projective over $S$? Is this true at least for V-compactified Jacobian spaces?   
\end{question}

As a step towards the above Question, we now prove that, under some mild assumption on the family of curves $X/S$, any relative compactified Jacobian space is represented by schemes (and not only by algebraic spaces) over $S$. 

\begin{theorem}\label{T:scheme}
Let $\pi:X\to S$ be a family of connected reduced curves over a quasi-separated and locally Noetherian algebraic space  $S$. Assume that Condition (*) of Theorem \ref{ThmB} is satisfied.
%VECCHIE IPOTESI
%\begin{enumerate}[(i)]
%    \item For any $s\in S$, any irreducible component of the fiber $X_s:=X\times_S s$ is geometrically integral.
%    \item There are sections $\{\sigma_i:S\to X\}_{i=1}^n$ of $\pi$ factoring through the smooth locus of $\pi$ such that, for any $s\in S$, any irreducible component of $X_s$ contains $\sigma_i(s)$ for some $i$. 
%\end{enumerate}
Then any compactified Jacobian space $\ov J_{X/S}^{\chi}$ of $X/S$ is representable by schemes over $S$.   
\end{theorem}
This result was proved for \emph{fine} and \emph{classical} relative compactified Jacobian spaces in \cite[Thm.~B]{esteves}. (As we observed immediately after Corollary~\ref{C:projective}, the additional condition of loc. cit., requiring  each irreducible component of every fiber $X_s$   to be geometrically integral, follows from Condition~(*), and it is therefore redundant).

\begin{proof}\footnote{We thank Dario Weissmann for some inspiring discussions that led to this proof.}
We can assume that $S$ is a scheme and prove that $\ov J_{X/S}^{\chi}$ is a scheme. Let $p\in \ov J_{X/S}^{\chi}$ be a closed point lying over the closed point $s_o\in S$. It is enough to find an open neighbourhood of $p$ in $\ov J_{X/S}^{\chi}$, that is a scheme. 

By the properties of the relative good moduli morphism $\Xi:\ov \J_{X/S}^{\chi}\to \ov J_{X/S}^\chi$, there exists a unique closed point in $\Xi^{-1}(p)$, let us call it $I\in \ov \J_{X_{s_o}}^{\chi}$. We now apply Lemma \ref{L:clospoint} below to the base change $\ov I$ of $I$ to the geometric fiber $X_s$, where $s$ is any geometric point lying over $s_o$: we find $\psi\in \Pol^{\chi}(X_s)$ such that $\ov I$ is a closed point of $\ov \J_{X_s}(\psi)$. Then we argue as in the proof of Corollary \ref{C:projective}: up to perturbing slightly $\psi$, by using the two assumptions on the family and arguing as in \cite[Lemma 35]{esteves}, we  find a vector bundle $E$ on $X/S$ of constant integral slope such that $\psi=\psi(E)^s$. By construction, we have that $I$ is also the unique closed point of the fiber of the relative good moduli morphism $\wt\Xi:=\Xi(E):\ov \J_{X/S}(E)\to \ov J_{X/S}(E)$ above $\wt \Xi(I)$.
Therefore, since $I$ is the unique closed point of the fibers $\Xi^{-1}(\Xi(I))$ and $\wt\Xi^{-1}(\wt \Xi(I))$, we have the following identity of topological subspaces inside the topological space $|\TF_{X_{s_o}}^{\chi}|$ of the stack $\TF_{X_{s_o}}^{\chi}$:
\begin{equation}\label{E:2fibers}
    |\Xi^{-1}(\Xi(I))|=\big\{J\in |\TF_{X_{s_o}}^{\chi}|\: : I\in \ov{\{J\}}\big\}= |\wt \Xi^{-1}(\wt \Xi(I))|.
 \end{equation}
 
Consider now the open substack $\V:=\ov \J_{X/S}^{\chi}\cap \ov \J_{X/S}(E)\subseteq \TF_{X/S}^{\chi}$, which contains the Subspace \eqref{E:2fibers} by construction. Since $\Xi$ is universally closed (as it is a good moduli morphism), the subset 
$$U:=\ov J_{X/S}^\chi\setminus \Xi(\ov \J_{X/S}^{\chi}\setminus (\ov \J_{X/S}^{\chi}\cap \ov \J_{X/S}(E)))$$
is an open subspace of $\ov J_{X/S}^{\chi}$ containing $p=\Xi(I)$. Since relative good moduli morphisms are stable under base change, the open substack $\U:=\Xi^{-1}(U)$, which is contained in $\V$ and contains \eqref{E:2fibers} by construction, is such that 
$\Xi:\U \to U$ is a relative good moduli morphism. We have the following diagram over $S$
 \begin{equation}\label{E:2gms}
\begin{tikzcd}
    \U \arrow[r] \arrow[d, "\Xi"] \arrow[r, hookrightarrow] &  \ov \J_{X/S}(E) \arrow[d, "\wt \Xi"] \\
   U \arrow[r, dashrightarrow, "g"] & \ov J_{X/S}(E),
\end{tikzcd}
\end{equation}
where the top horizontal arrow is an open inclusion, and the two vertical arrows are relative good moduli space morphisms. Because relative good moduli space morphisms are universal with respect to maps to relative algebraic spaces, we get a morphism $g$ making Diagram~\eqref{E:2gms} commutative. From the commutativity, we deduce
$$
\wt \Xi^{-1}(\wt \Xi(I))=\U\cap \wt \Xi^{-1}(\wt \Xi(I))=\Xi^{-1}(g^{-1}(\wt \Xi(I))).
$$
This implies that $g^{-1}(\wt \Xi(I))=\{\Xi(I)\}$, from which we deduce that $g$ is a quasi-finite morphism above $\wt \Xi(I)$. Hence, there exists an open neighborhood $W$ of $\wt \Xi(I)$ in $\ov J_{X/S}(E)$ such that 
$$
U\supseteq V:=g^{-1}(W)\xrightarrow{h:=g_{|V}} W\subseteq \ov J_{X/S}(E) \text{ is quasi-finite.}
$$
Note that $g$, and hence also $h$, is separated since $V\subseteq U\subseteq \ov J_{X/S}^{\chi}$ are open and $\ov J_{X/S}^{\chi}$ is proper over $S$. Therefore, we can apply Zariski's Main Theorem  \cite[\href{https://stacks.math.columbia.edu/tag/082K}{Tag 082K}]{stacks-project} for algebraic spaces and obtain a  factorization of $h$ as
\begin{equation}\label{E:fact-h}
h: V\stackrel{j}{\hookrightarrow} T \xrightarrow{p} W,
\end{equation}
where $j$ is an open embedding and $p$ is a finite morphism. Since $W$ is a scheme (being an open subset of $\ov J_{X/S}(E)$ which is locally projective over $S$ by Theorem \ref{T:projective}), the factorization of $h$ in \eqref{E:fact-h} implies that $V$ is also a scheme. Since $V$ is an open neighborhood of $p=\Xi(I)$ in $\ov J_{X/S}^{\chi}$, we are done.    
\end{proof}

\begin{lemma}\label{L:clospoint}
    Let $X$ be a connected reduced curve over $k=\ov k$ and let $I\in \TF_X^{\chi}(k)$. Then there exists a numerical polarization $\psi$ on $X$ of characteristic $\chi$ such that $I$ is a closed point of the compactified Jacobian stack $\ov \J_X(\psi)$.
\end{lemma}
\begin{proof}
Consider the decomposition of \eqref{E:decI} of $I$
$$
I=I_{Y_1}\oplus \ldots \oplus I_{Y_c},
$$
so that each $I_{Y_i}$ is a simple rank-$1$ torsion-free sheaf on the subcurve $Y_i\subseteq X$ (which implies that each $Y_i$ is connected). In particular we have that 
\begin{equation}\label{E:decomI}
\chi=\chi(I)=\sum_{i=1}^c \chi(I_{Y_i}).
\end{equation}
We now apply \cite[Cor. 15]{esteves} in order to find a vector bundle $E_i$ on $Y_i$ such that $I_{Y_i}$ is $\psi(E_i)$-stable (in the sense of Definition \ref{D:ss}). 
%i.e. $\s(\psi(E_i))$-semistable with all the inequalities in Definition \ref{D:ss} being strict. 
Pick now a general numerical polarization $\psi_i$ that lies on a chamber whose closure contains $\psi(E_i)$ in the wall and chamber decomposition \eqref{E:arr-hyper}. Then, it is easy to see that $I_{Y_i}$ is also $\psi_i$-stable.

The numerical polarizations $\psi_i\in \Pol^{\chi(I_{Y_i})}(Y_i)$ define a unique numerical polarization $\psi\in \Pol^{\chi}(X)$ by setting 
$$
\psi_Z=\sum_{i=1}^c (\psi_i)_{Y_i\cap Z} \text{ for any subcurve } Z\subseteq X.
$$
Notice that $Y_i\in \wh \D(\s(\psi))$ and $\s(\psi)(Y_i)=\s(\psi_i)$ by construction. Therefore, $I_{Y_i}$ is $\s(\psi)$-semistable because it is $\s(\psi_i)$-semistable. This implies, using Proposition \ref{P:ext-ss}, that $I$ is $\s(\psi)$-semistable, i.e. $I\in \ov \J_X(\psi)$. 

We now claim that $I$ is a closed point of $\J_X(\psi)$. Indeed, consider an isotrivial specialization of $I$  to another sheaf $J\in \ov \J_X(\psi)$, which then by Proposition \ref{P:iso-spec} is of the form $J=\Gr_{W_\bullet}(I)$ for some ordered partition $W_\bullet$ of $X$. Since we have the decomposition \eqref{E:decomI} of $I$, the ordered partition $W_\bullet$ must be a refinement of the ordered partition $Y_\bullet:=(Y_1,\ldots, Y_c)$. However, we must have that $W_\bullet=Y_\bullet$, for otherwise some $I_{Y_i}$ would isotrivially specialize to a $\s(\psi_i)$-semistable sheaf of the form $I_Z\oplus I_{\ov{Y_i-Z}}$, for some non-trivial subcurve $Z\subset Y_i$, which is however impossible because $\s(\psi_i)$ is a general V-stability on $Y_i$. 
Then we have that 
$$
J=\Gr_{W_\bullet}(I)=\Gr_{Y_\bullet}(I)=I.
$$
In other words, $I$ does not have non-trivial isotrivial specializations in $\ov \J_X(\psi)$, which  implies that $I$ is closed point of $\J_X(\psi)$.
\end{proof}

As a special case of Theorem \ref{T:scheme}, we get the following

\begin{corollary}\label{C:scheme}
    Let $X$ be a reduced connected curve over a field $k=\ov k$. Then any compactified Jacobian space $\ov J_X^{\chi}$ is a scheme. 
\end{corollary}

\section{V-compactified Jacobians for reduced curves with planar singularities}\label{sec:VcJ-planar}

The aim of this section is to study V-compactified Jacobians for families of connected reduced curves with planar singularities. The case of fine classical compactified Jacobians was studied in \cite{MRV}, and our goal is to extend most of the results of loc. cit. to V-compactified Jacobian stacks/spaces (not necessarily classical, not necessarily fine). 

%We will extend to this more general class of compactified Jacobians most of the results obtained by Melo-Rapagnetta-Viviani in \cite{MRV} for fine classical compactified Jacobians. 

Consider first a reduced connected curve $X$ over $k=\ov k$ with planar singularities, i.e. such that for any $p\in X$ the complete local ring $\wh \O_{X,p}$ has embedded dimension at most two or, equivalently, $\wh \O_{X,p}=\frac{k[[x,y]]}{f(x,y)}$ for some reduced series $f(x,y)\in k[[x,y]]$. For such curves, the stack of rank-$1$ torsion-free sheaves is particularly well-behaved.

\begin{theorem}\label{T:TFX-pla}
    Let $X$ be a reduced connected curve over $k=\ov k$ with planar singularities.
    Then
    \begin{enumerate}[(i)]
        \item the stack $\TF_X$ is  reduced with complete intersection singularities and embedded dimension at most $2g(X)$ at any point. 
        \item The open substack $\PIC_X\subseteq \TF_X$ is dense and it is the smooth locus of $\TF_X$. In particular, $\TF_X$ is equidimensional of dimension $g(X)-1$.
    \end{enumerate}
\end{theorem}
Note that the dimension of the $\Gm$-rigidification $\TF_X\fatslash \Gm$ is $g(X)$. 
\begin{proof}
    The proof is similar to the proof of \cite[Thm. 2.3]{MRV}: one shows that the $M$-twisted Abel maps of degree $g(X)$ (where $M\in \PIC_X$)
    $$
    \begin{aligned}
    A_M:\Hilb^{g(X)}_X& \longrightarrow \TF_X\\
    D & \mapsto I_D\otimes M
    \end{aligned}
    $$
   are smooth and surjective locally on the codomain (same proof as \cite[Prop. 2.5]{MRV}) and then the results follow from the properties of $\Hilb^d_X$ (see \cite[Fact 2.4]{MRV} and the references therein).
\end{proof}
From the fact that compactified Jacobians are open in $\TF_X$, we immediately deduce the following 

\begin{corollary}\label{C:cJ-pla}
   Let $X$ be as in Theorem \ref{T:TFX-pla}. Then any compactified Jacobian stack $\ov \J_X$ of $X$ satisfies:
   \begin{enumerate}[(i)]
       \item $\ov \J_X$ is reduced with complete intersection singularities of embedded dimension at most $2g(X)$ at every point;
       \item $\ov \J_X$ is equidimensional of dimension $g(X)-1$;
       \item the smooth locus of $\ov \J_X$ coincides with the open substack parameterizing line bundles. 
   \end{enumerate}
\end{corollary}

Consider now the effective semiuniversal deformation of $X$:
\begin{equation}\label{E:univdef}
    \begin{tikzcd}
     X \ar[r,hookrightarrow] \arrow[d] \ar[dr, phantom, "\square"] & \fX \arrow[d, "\Phi"]\\ 
     o=\Spec(k) \arrow[r,hookrightarrow] & \Spec R_X
    \end{tikzcd}
\end{equation}
where $R_X$ is a complete local $k$-algebra whose tangent space has dimension equal to $\dim \Ext^1(\Omega_X^1,\O_X)$, $\Phi$ is a projective family of curves with central fiber isomorphic to $X$, and with the property that any deformation of $X$ is obtained, \'etale locally, by pull-back from the family $\Phi$. Since the curve $X$ has locally planar singularities, then $R_X$ is a regular ring, which implies that  $R_X=k[[X_1,\ldots, X_N]]$ where $N=\dim \Ext^1(\Omega_X^1,\O_X)$. 

The stack of relative torsion-free rank-$1$ sheaves on the family $\Phi:\fX\to \Spec R_X$ has the following nice properties.

\begin{theorem}\label{T:TF-univ}
    Let $X$ be a reduced connected curve over $k=\ov k$ with planar singularities and consider its effective semiuniversal deformation as in \eqref{E:univdef}. 
    Then
    \begin{enumerate}[(i)]
        \item \label{T:TF-univ1} The stack $\TF_{\fX/R_X}$ is  regular.
        \item \label{T:TF-univ2} The morphism $\TF_{\fX/\Spec R_X}\to \Spec R_X$ is flat of relative dimension $g(X)-1$. 
    \end{enumerate}
\end{theorem}
\begin{proof}
Part \eqref{T:TF-univ1} is proved as \cite[Thm. 4.5(iii)]{MRV}. 

Part \eqref{T:TF-univ2} follows from miracle flatness \cite[\href{https://stacks.math.columbia.edu/tag/00R4}{Tag 00R4}]{stacks-project} using that $\Spec R_X$ is regular, $\TF_{\fX/R_X}$ is regular by Part \eqref{T:TF-univ1}, hence Cohen-Macaulay, and the fibers of the morphism are equidimensional of dimension $g(X)-1$ by Theorem \ref{T:TFX-pla}.      
\end{proof}

We  now describe the relative V-stability conditions on the family $\Phi:\fX\to \Spec R_X$. Any V-stability condition $\s$ on $X$ determines a relative V-stability condition $\ov \s$ on $\fX/\Spec R_X$ as follows. Since $\Spec R_X$ is a local scheme with closed point $o$, for any geometric point $s$ of $\Spec R_X$ there is a unique specialization $\xi_s:s\rightsquigarrow o$. Then we set 
$$
\ov \s:=\{\ov\s^s:=\xi_s^*(\s)\}\in \VStab(\fX/R_X). 
$$
This is well-defined since any \'etale specialization $\xi:s\rightsquigarrow t$ of geometric points of $\Spec R_X$ is such that $\xi_t\circ \xi=\xi_s$. The following Lemma is obvious.

\begin{lemma}\label{L:VStab-un}
   The map 
   \begin{equation}\label{E:VStab-un}
   \begin{aligned}
   \VStab(X) & \to \VStab(\fX/R_X)  \\
   \s &\mapsto \ov \s
   \end{aligned}
   \end{equation}
is a bijection with inverse given by sending $\s\in \VStab(\fX/R_X)$ to $\s^o\in \VStab(X)$.
\end{lemma}

We now consider relative V-compactified Jacobians stacks and spaces for $\Phi: \fX\to \Spec R_X$.

\begin{theorem}\label{T:VcJ-univ}
   Notation as in Theorem \ref{T:TF-univ}. Let $\s$ be a V-stability condition on $X$ and let $\ov s$ be its associated V-stability condition on $\fX/R_X$ via Lemma~\ref{L:VStab-un}. 
   \begin{enumerate}[(i)]
\item \label{T:VcJ-univ1} $\ov \J_{\fX/R_X}(\ov \s)$ is regular and the morphism $F_{\fX/R_X}(\s):\ov \J_{\fX/R_X}(\ov \s)\to \Spec R_X$ is flat of relative dimension $g(X)-1$.
\item \label{T:VcJ-univ2} the morphism $f_{\fX/R_X}(\ov \s):\ov J_{\fX/R_X}(\ov \s)\to \Spec R_X$ is proper and flat of relative dimension $g(X)$.
   \end{enumerate}
\end{theorem}
\begin{proof}
Part \eqref{T:VcJ-univ1} follows from Theorem \ref{T:TF-univ} using that $\ov \J_{\fX/R_X}(\ov \s)$ is open in $\TF_{\fX/R_X}$. 

Part \eqref{T:VcJ-univ2}: the properness of $f_{\fX/R_X}(\ov \s)$ follows from Theorem \ref{T:VcJ}; the flatness of $f_{\fX/R_X}(\ov \s)$ follows from the flatness of $F_{\fX/R_X}(\ov \s)$ by Part~\eqref{T:VcJ-univ1} together with the fact that the flatness is preserved when passing to  the relative good moduli space (see \cite[Thm. 4.16(ix)]{Alp});  the relative dimension of $f_{\fX/R_X}(\ov \s)$ equals $g(X)$ since the generic fiber is the Jacobian of characteristic $|\s|$ of the generic fiber of $\fX/R_X$, which is a smooth curve of genus $g(X)$.
\end{proof}

\begin{corollary}\label{C:VcJ-pla}
Let $X$ be a reduced connected curve over $k=\ov k$ with planar singularities and let $\s\in \VStab^\chi(X)$. Then we have that:
 \begin{enumerate}[(i)]
 \item \label{C:VcJ-pla1}  $\ov J_X(\s)$ is connected, reduced and equidimensional of dimension $g(X)$.
 \item \label{C:VcJ-pla2} $\ov \J_X(\s)$ is connected,  equidimensional of dimension $g(X)-1$, reduced with locally complete intersection singularities and trivial dualizing sheaf.  
 \end{enumerate}
\end{corollary}
\begin{proof}
 Since the morphism  $f_{\fX/R_X}(\s):\ov J_{\fX/R_X}(\ov \s)\to \Spec R_X$ is proper and flat by Theorem \ref{T:VcJ-univ}\eqref{T:VcJ-univ2} and the generic fiber is geometrically connected (being the Jacobian of characteristic $|\s|$ of the smooth generic fiber of $\Phi$), we deduce that also the central fiber $\ov J_X(\s)$ must be connected by  \cite[\href{https://stacks.math.columbia.edu/tag/0BUI}{Tag 0BUI}]{stacks-project}. We deduce that also $\ov \J_X(\s)$ is connected since the good moduli morphism $\ov \J_X(\s)\to \ov J_X(\s)$ has geometrically connected fibers by  \cite[Thm. 4.16(ix)]{Alp}). Moreover:
 \begin{itemize}
     \item $\ov \J_X(\s)$ is reduced with locally complete intersection singularities and equidimensional of dimension $g(X)-1$ by Corollary \ref{C:cJ-pla}.
     \item $\ov \J_X(\s)$ has trivial dualizing sheaf by the same proof of \cite[Cor. 5.7]{MRV}. 
     \item $\ov J_X(\s)$ is reduced since $\ov \J_X(\s)$ is reduced and the property of being reduced descends along good moduli spaces (see \cite[Thm. 4.16(viii)]{Alp}).
     \item $\ov J_X(\s)$ is equidimensional of dimension $g(X)$ by Theorem \ref{T:VcJ-univ}\eqref{T:VcJ-univ2}.
 \end{itemize}

The fact that $\ov \J_X(\s)$ is reduced with locally complete intersection singularities follows from Corollary~\ref{C:cJ-pla} and the fact that it has trivial dualizing sheaf is proved as in \cite[Cor.~5.7]{MRV}.

\end{proof}

%Consider now a $1$-parameter regular smoothing $\pi:\X\to \Delta=\Spec R$ of $X$. Note that such a $1$-parameter regular smoothing exists if and only if $X$ has  planar singularities  (see \cite[]{MRV}). 

%The V-stability condition $\s$ on $X=\X_o$ gives rise to a relative V-stability condition on $\X/S$, which we also call $\s$ by abuse of notation, by adding to it the unique V-stability of characteristic $|\s|$ on the smooth curve $\X_{\ov \eta}$. Theorem \ref{T:VcJ} now implies that  $\ov \J_{\X/S}(\s)$ admits a relative proper good moduli space $\ov J_{\X/S}(\s)\to \Delta$. Since the generic fiber of $\ov J_{\X/S}(\s)_{\eta}$  is isomorphic to the Jacobian of characteristic $|\s|$ of the generic fiber $\X_{\eta}$, and hence it is geometrically connected, then also the special fiber $\ov J_{\X/S}(\s)_o=\ov J_X(\s)$ is connected. This implies that also the V-compactified Jacobian stack $\ov \J_X(\s)$ is connected. 

Finally, we consider relative V-compactified Jacobian stacks/spaces for families of reduced curves with planar singularities. First, we collect in the following theorem the properties of relative V-compactified Jacobian stacks/spaces. 

\begin{theorem}\label{T:vcJ-fam-pla}
 Let $\pi:X\to S$ be a family of connected reduced curves with planar singularities and let $\s$ be a relative V-stability condition on $X/S$.
 Then we have that: 
 \begin{enumerate}[(i)]
     \item\label{T:vcJ-fam-pla1} The morphism $F_{X/S}(\s):\ov \J_{X/S}(\s)\to S$ is flat of relative dimension $g(X)-1$ with geometrically connected and reduced fibers.
     \item\label{T:vcJ-fam-pla2} The morphism $f_{X/S}(\s):\ov J_{X/S}(\s)\to S$ is proper and flat of relative dimension $g(X)$ with geometrically connected and reduced fibers.
 \end{enumerate}
\end{theorem}
\begin{proof}
The properties of the geometric fibers of $F_{X/S}(\s)$ and $f_{X/S}(\s)$ follow from Corollary \ref{C:VcJ-pla}. The properness of  $f_{X/S}(\s)$ follows from Theorem \ref{T:VcJ}. The flatness of $F_{X/S}(\s)$ (which then implies the flatness of $f_{X/S}(\s)$ as in Theorem \ref{T:VcJ-univ}\eqref{T:VcJ-univ2}) follows from Theorem \ref{T:VcJ-univ}\eqref{T:VcJ-univ1} using that $F_{X/S}(\s)$ is \'etale locally at any geometric point $s$ of $S$ the pull-back of $F_{\fX_s/R_{X_s}}(\ov{\s^{s}})$.    
\end{proof}

Next, we prove that a compactified Jacobian stack for a family is a V-compactified Jacobian as soon as it is such on every geometric fiber. 

\begin{theorem}\label{T:char-vcJ-pla}
 Let $\pi:X\to S$ be a family of connected reduced curves with planar singularities. A compactified Jacobian stack for $X/S$ is a V-compactified Jacobian stack if and only if it is such on every geometric fiber of $X/S$.
\end{theorem}
\begin{proof}
The implication only if is obvious. Let us prove the if implication.

Consider a compactified Jacobian stack $\ov \J_{X/S}^{\chi}$ for $X/S$ and assume that, for every geometric point $s$ of $S$, the fiber of $\ov \J_{X/S}^{\chi}$ over $s$ is equal to $\ov \J_{X_s}(\s^s)$ for some (uniquely determined) V-stability condition $\s^s\in \VStab^{\chi}(X_s)$.
Proposition \ref{P:s-comp} below implies that $\s=\{\s^s\}$ is a V-stability condition on $X/S$ of characteristic $\chi$. By construction, we have that  
$$
\ov \J_{X/S}^{\chi}=\ov \J_{X/S}(\s),
$$
and we are done.     
\end{proof}

\begin{proposition}\label{P:s-comp}
 Let $\pi:X\to \Delta=\Spec R$ be a family of connected reduced curves with planar singularities, with $R$ a strictly Henselian DVR with residue field $k=\ov k$. Denote by $o=\Spec k$ the special point of $\Delta$, by $\eta$ (resp. $\ov \eta$) its  generic (resp. geometric generic) point and by $\xi:\ov\eta \rightsquigarrow o$ the \'etale specialization induced by $o\in \ov{\{\eta\}}$. 
 
 Let $\s^o$ be a V-stability condition on the special fiber $X_o$ and $\s^{\ov \eta}$ be a V-stability condition on the geometric generic fiber $X_{\ov \eta}$, and assume that $|\s^o|=|\s^{\ov \eta}|=:\chi$.
Consider the constructible subset  $\ov \J_{X/\Delta}(\s^o,\s^{\ov \eta})\subseteq \TF_{X/\Delta}^{\chi}$  whose fiber over $T\xrightarrow{f} \Delta$ is given by 
$$
\ov \J_{X/S}(\s^o,\s^{\ov \eta})(T):=\left\{
\begin{aligned}
    I \in  \TF^\chi_{X/S}(T): \: \chi((I_{|X_{t}})_Y)\geq \s^{f(t)}_Y & \text{ for every geometric point $t$ of $T$ }\\
    & \text{ and for every } Y \in \BCon(X_{f(t)})
    \end{aligned}\right\}.
$$
Then we have that 
\begin{enumerate}
\item \label{P:s-comp1} $\ov \J_{X/\Delta}(\s^o,\s^{\ov \eta})\subseteq \TF_{X/\Delta}^{\chi} \text{ is open } \Longleftrightarrow \s^{\ov \eta}\leq \xi^*(\s^o)$.
\item \label{P:s-comp2} The following conditions are equivalent:
\begin{enumerate}[(i)]
\item \label{P:s-comp2i} $(o \mapsto \s^o, \eta \mapsto \s^{\ov \eta})$ defines a V-stability on  $X/\Delta$, i.e. $\s^{\ov \eta}= \xi^*(\s^o)$.
\item \label{P:s-comp2ii} $\ov \J_{X/\Delta}(\s^o,\s^{\ov \eta})$ is a compactified Jacobian stack for $X/\Delta$.
\item \label{P:s-comp2iii} $\ov \J_{X/\Delta}(\s^o,\s^{\ov \eta})\subseteq \TF_{X/\Delta}^{\chi}$ is open and the map $\ov \J_{X/\Delta}(\s^o,\s^{\ov \eta})\to \Delta$ is universally closed. 
\end{enumerate}
\end{enumerate}
\end{proposition}
\begin{proof}
Part \eqref{P:s-comp1}: assume first that $\s^{\ov \eta}\leq \xi^*(\s^o)$. 
Consider the V-stability condition $\wt \s:=(\s^o,\xi^*(\s^o))$ on $X/\Delta$ and its associated compactified Jacobian stack $\ov \J_{X/\Delta}(\wt \s)$ by Theorem \ref{T:VcJ}. Remark \ref{R:inc-VcJ} and the assumption $\s^{\ov \eta}\leq \xi^*(\s^o)$ imply that $\ov \J_{X/\Delta}(\s^o,\s^{\ov \eta})=\ov \J_{X/\Delta}(\wt \s)\bigcup \ov \J_{X_{\eta}}(\s^{\ov \eta})$. Hence, $\ov \J_{X/\Delta}(\s^o,\s^{\ov \eta})$ is open in $\TF_{X/\Delta}^{\chi}$ since $\ov \J_{X_{\eta}}(\s^{\ov \eta})$ is open in $\TF_{X_\eta}^{\chi}$ by Lemma-Definition \ref{LD:VStab}, and hence also in $\TF_{X/\Delta}^{\chi}$.

Conversely, assume that $\ov \J_{X/\Delta}(\s^o,\s^{\ov \eta})\subseteq \TF_{X/\Delta}^{\chi}$ is open. Since $R$ is stricly Henselian and $\pi:X\to \Delta$ is a family of reduced curves, we have that the irreducible components of $X_{\eta}$ are geometrically irreducible (see e.g. the proof of \cite[Lemma 5.1]{MRV}) so that any subcurve of $X_{\ov \eta}$ is given by $Y_{\ov \eta}$ for some unique subcurve $Y$ of $X_{\eta}$. Fix now a biconnected subcurve $Y_{\ov \eta}$ of $X_{\ov \eta}$ and consider the biconnected subcurve (see \eqref{E:xi*})
$$Y_o:=\xi_*(Y_{\ov \eta})=\ov Y\cap X_o\subset X_o.$$ 
Lemma \ref{L:exis-sh} below implies that there exists a line bundle $L\in \ov J_{X_o}(\s^o)$ such that $\chi(L_{Y_o})=\s^o_{Y_o}$. Since $\PIC_{X/\Delta}^{\chi}\to \Delta=\Spec R$ is smooth and $R$ is strictly Henselian, we can find $\L\in \PIC_{X/\Delta}^{\chi}(\Delta)$ such that $\L_{o}=L$ (see \cite[Sec. 2.3, Prop. 5]{Neron}). Since $\ov \J_{X/\Delta}(\s^o,\s^{\ov \eta})\subseteq \TF_{X/\Delta}^{\chi}$ is open by assumption, we must have that $\L\in \ov \J_{X/\Delta}(\s^o,\s^{\ov \eta})(\Delta)$. We now compute
\begin{equation*}
  \begin{aligned}
   & (\xi^*(\s^o))_{Y_{\ov \eta}}=\s^o_{Y_o}=\chi(L_{Y_o}) &\text{ by the definition }\eqref{E:VStab-fun} \text{ of } \xi^*(\s^o) \text{ and the property of $L$} \\ 
   &  =\chi((\L_{o})_{Y_{o}})=\chi((\L_{\eta})_{Y_{\eta}}) =\chi((\L_{\ov \eta})_{Y_{\ov \eta}}) & \text{ since $\L$ is a line bundle and $\ov{Y}\to \Delta$ is a flat family}\\ 
   & \geq \s^{\ov \eta}_{Y_{\ov \eta}} & \text{ since } \L_{\eta}\in \ov \J_{X_\eta}(\s^{\ov \eta}).
  \end{aligned}  
\end{equation*}
Since this is true for any $Y_{\ov \eta}\in \BCon(X_{\ov \eta})$, we deduce that $\s^{\ov \eta}\leq \xi^*(\s^o)$.

Part \eqref{P:s-comp2}: the implication $\eqref{P:s-comp2i}\Rightarrow \eqref{P:s-comp2ii}$ follows from Theorem \ref{T:VcJ}, while the implication $\eqref{P:s-comp2ii}\Rightarrow \eqref{P:s-comp2iii}$ follows from the definition of compactified Jacobian stack. Assume now Property \eqref{P:s-comp2iii} and let us show that \eqref{P:s-comp2i} holds true. 
Using that $\ov \J_{X/\Delta}(\s^o,\s^{\ov \eta})\subseteq \TF_{X/\Delta}^{\chi}$ is open and Part \eqref{P:s-comp1}, we obtain the inequality $\s^{\ov \eta}\leq \xi^*(\s^o)$. In order to show the other inequality $\s^{\ov \eta}\geq \xi^*(\s^o)$, fix,  as before, a biconnected subcurve $Y_{\ov \eta}$ of $X_{\ov \eta}$ and consider the biconnected subcurve 
$$Y_o:=\xi_*(Y_{\ov \eta})=\ov Y\cap X_o\in \BCon(X_o).$$ 
Lemma \ref{L:exis-sh} below implies that, up to a finite base change of $R$, there exists a line bundle $L\in \ov J_{X_\eta}(\s^{\ov \eta})(\eta)$ such that $\chi(L_{Y_\eta})=\s^{\ov \eta}_{Y_{\ov \eta}}$. Since $\ov \J_{X/\Delta}(\s^o,\s^{\ov \eta})\to \Delta$ is universally closed, up to a finite base change of $R$, we can find $\I\in \ov \J_{X/\Delta}(\s^o,\s^{\ov \eta})(\Delta)$ such that $\I_{\eta}=L$. 
We now compute
\begin{equation*}
  \begin{aligned}
& \s^{\ov \eta}_{Y_{\ov \eta}}=\chi(L_{Y_\eta})=\chi((\I_\eta)_{|Y_\eta}) &
\text{ by the property of $L$ and $\I_\eta=L$}\\
&=\chi((\I_o)_{|Y_o}) & \text{ since $\ov{Y}\to \Delta$ is a flat family and $\I$ is flat over $\Delta$}\\
& \geq \chi((\I_o)_{Y_o}) & \text{ since $(\I_o)_{Y_o}$ is the quotient of $(\I_o)_{|Y_o}$ by the torsion subsheaf}\\
   & \geq \s^o_{Y_o}=(\xi^*(\s^o))_{Y_{\ov \eta}} & \text{ since } \I_o\in \J_{X_o}(\s^o)   \text{ and by the  definition }\eqref{E:VStab-fun} \text{ of } \xi^*(\s^o).  \\ 
  \end{aligned}  
\end{equation*}
Since this is true for any $Y_{\ov \eta}\in \BCon(X_{\ov \eta})$, we deduce that $\s^{\ov \eta}\geq \xi^*(\s^o)$, as required. 
\end{proof}

\begin{lemma}\label{L:exis-sh}
Let $X$ be a connected reduced curve with planar singularities over $k=\ov k$ and let $\s\in \VStab^{\chi}(X)$.
Then for any $Y\in \BCon(X)$, there exists a line bundle $L$ belonging to $\ov\J_X(\s)$ such that 
$$
\chi(L_Y)=\s_Y.
$$
\end{lemma}
\begin{proof}
 Since planar singularities can be smoothened out independently (see \cite{Sernesi}), we can smoothen out all the singularities of $X$ except the ones in $Y\cap Y^c$ in order to produce a family $\pi:\X\to \Delta=\Spec R$ of connected reduced curves over a Henselian DVR such that $\X_o=X$ and $\X_{\eta}$ has two irreducible components $\X_\eta^1$ and $\X_\eta^2$, which are geometrically irreducible (see e.g. the proof of \cite[Lemma 5.1]{MRV}), and such that 
 $$
 Y=\xi_*(\X_{\ov \eta}^1)=\ov{\X_\eta^1}\cap X_o \text{ and } Y^c=\xi_*(\X_{\ov \eta}^2)=\ov{\X_\eta^2}\cap X_o.
 $$
 Denote by $\xi:\ov \eta\rightsquigarrow o$ the \'etale specialization induced by $\Delta$.  Consider the V-stability $\xi^{*}(\s)$ on $\X_{\ov \eta}$, which, by definition, is such that 
 $$
 \xi^*(\s)_{\X_{\ov \eta}^1}=\s_Y \text{ and } \xi^*(\s)_{\X_{\ov \eta}^2}=\s_{Y^c}.
 $$
Take a line bundle $M$ on $\X_{\eta}$ such that 
\begin{equation}\label{E:prop-M}
\chi(M_{\X_{\ov \eta}^1})=\s_Y \text{ and } \chi(M)=|\s|=|\xi_*(\s)|.
\end{equation}
We claim that $M\in \ov \J_{\X_{\eta}}(\xi^*(\s))(\eta)$ because $\X_{\ov \eta}$ has only two irreducible components, $\X_{\ov \eta}^1$ and $\ov \X_{\ov \eta}^2$, and we have \eqref{E:prop-M} together with 
\begin{equation*}
  \begin{aligned}
     & \chi(M_{\X_{\ov \eta}^2})>\chi(\leftindex_{\X_{\ov \eta}^2}M) & \text{ by Lemma \ref{L:IY} and using that } M_{\ov \eta}\neq M_{\X_{\ov \eta}^1}\oplus M_{\X_{\ov \eta}^2},  \\ 
     & = \chi(M)-\chi(M_{\X_{\ov \eta}^1})=|\s|-\s_Y & \text{ by \eqref{E:add-chi} and the properties of } M,\\
     & \geq s_{Y^c}-1=\xi^*(\s)_{\X_{\ov \eta}^2}-1 & \text{ by \eqref{E:sum-n}.} 
  \end{aligned}  
\end{equation*}
 Consider now the relative V-stability condition $\wt \s=(\s,\xi^*(\s))$ on $\X/\Delta$ and the associated relative V-compactified Jacobian stack $\ov \J_{\X/\Delta}(\wt \s)$. Since $\ov \J_{\X/\Delta}(\wt \s)\to \Delta$ is universally closed by Proposition~\ref{P:univ-cl}, there exists, up to a finite base change of $\Delta$, a sheaf  $\I\in \ov \J_{\X/\Delta}(\wt \s)$ such that $\I_{\eta}=M$. The central fiber $\I_o\in \ov \J_{X}(\s)$ is such that 
 \begin{equation}\label{E:prop-Io}
     \begin{sis}
&      \chi(\I_o)=\chi(M)=\chi \text{ and }  \chi((\I_o)_Y)\geq \s_Y & \text{ since $\I_o$ is $\s$-semistable,}\\ 
& \chi((\I_o)_Y)\leq \chi((\I_o)_{|Y})=\chi((\I_{\eta})_{|\X_{\eta}^1})=\chi(M_{\X_\eta^1})=\s_Y & \text{ by the flatness of $\ov \X_{\eta}^1\to \Delta$ and \eqref{E:prop-M}.} \\
     \end{sis}
 \end{equation}
We deduce that 
 \begin{equation}\label{E:prop-Io2}
\begin{sis}
& \chi((\I_o)_Y)=\s_Y,\\
&  (\I_o)_{|Y}=(\I_o)_Y \Longrightarrow  \I_o \text{ is locally free at the points } Y\cap Y^c.
\end{sis}
 \end{equation}
Since $X$ has planar singularities, we can deform $\I_o$ to a line bundle $L$ on $X$ by Theorem \ref{T:TFX-pla}. This line bundle $L$  belongs to $\ov \J_X(\s)$ since $\ov \J_X(\s)\subseteq \TF_X$ is open. Moreover, since $\I_o$ is locally free at the points of $Y\cap Y^c$, we have 
$$
\chi(L_Y)=\chi((\I_o)_Y)=\s_Y,
$$
which concludes our proof. 
\end{proof}

The above Lemma implies that a V-stability condition is uniquely determined by its associated V-compactified Jacobian stack.

\begin{corollary}\label{C:unic-s}
Let $X$ be a connected reduced curve with planar singularities over $k=\ov k$ and let $\s, \t\in \VStab(X)$. If $\ov\J_X(\s)=\ov \J_X(\t)$ then $\s=\t$. 
\end{corollary}

    \bibliographystyle{alpha}	
    \bibliography{bibtex}

\newcommand{\etalchar}[1]{$^{#1}$}
\begin{thebibliography}{CMKV15}

\bibitem[AHLH23]{AHLH}
Jarod Alper, Daniel Halpern-Leistner, and Jochen Heinloth.
\newblock Existence of moduli spaces for algebraic stacks.
\newblock {\em Invent. Math.}, 234(3):949--1038, 2023.

\bibitem[AK80]{altmankleiman}
Allen~B Altman and Steven~L Kleiman.
\newblock Compactifying the picard scheme.
\newblock {\em Advances in Mathematics}, 35(1):50--112, 1980.

\bibitem[Ale04]{alexeev}
Valery Alexeev.
\newblock Compactified {J}acobians and {T}orelli map.
\newblock {\em Publications of the Research Institute for Mathematical Sciences}, 40, 12 2004.

\bibitem[Alp13]{Alp}
Jarod Alper.
\newblock Good moduli spaces for {A}rtin stacks.
\newblock {\em Ann. Inst. Fourier (Grenoble)}, 63(6):2349--2402, 2013.

\bibitem[BLR90]{Neron}
Siegfried Bosch, Werner L\"utkebohmert, and Michel Raynaud.
\newblock {\em N\'eron models}, volume~21 of {\em Ergebnisse der Mathematik und ihrer Grenzgebiete (3) [Results in Mathematics and Related Areas (3)]}.
\newblock Springer-Verlag, Berlin, 1990.

\bibitem[Cap94]{caporaso}
Lucia Caporaso.
\newblock A compactification of the universal {P}icard variety over the moduli space of stable curves.
\newblock {\em Journal of the American Mathematical Society}, 7(3):589--660, 1994.

\bibitem[CMKV15]{CMKVlocal}
Sebastian Casalaina-Martin, Jesse~Leo Kass, and Filippo Viviani.
\newblock The local structure of compactified {J}acobians.
\newblock {\em Proc. Lond. Math. Soc. (3)}, 110(2):510--542, 2015.

\bibitem[Est99]{Est-sep}
Eduardo Esteves.
\newblock Separation properties of theta functions.
\newblock {\em Duke Math. J.}, 98(3):565--593, 1999.

\bibitem[Est01]{esteves}
Eduardo Esteves.
\newblock Compactifying the relative {J}acobian over families of reduced curves.
\newblock {\em Trans. Amer. Math. Soc.}, 353(8):3045--3095, 2001.

\bibitem[Fav]{fava2024}
Marco Fava.
\newblock On a combinatorial classification of fine compactified universal {J}acobians.
\newblock Preprint arXiv:2403.17871.

\bibitem[FPV24]{FPV}
Marco Fava, Nicola Pagani, and Filippo Viviani.
\newblock A complete theory of smoothable compactified {J}acobians of nodal curves, 2024.
\newblock Preprint arXiv:2412.03532.

\bibitem[FPV25]{FPV3}
Marco Fava, Nicola Pagani, and Filippo Viviani.
\newblock A complete classification of compactified universal {J}acobians., 2025.
\newblock to appear.

\bibitem[Hei17]{heinloth-HM}
Jochen Heinloth.
\newblock Hilbert-{M}umford stability on algebraic stacks and applications to {$\mathcal{G}$}-bundles on curves.
\newblock {\em \'{E}pijournal G\'{e}om. Alg\'{e}brique}, 1:Art. 11, 37, 2017.

\bibitem[HMP{\etalchar{+}}25]{panda}
David Holmes, Samouil Molcho, Rahul Pandharipande, Aaron Pixton, and Johannes Schmitt.
\newblock Logarithmic double ramification cycles, 2025.
\newblock Preprint arXiv:2207.06778.

\bibitem[Igu56]{igusa}
Jun-ichi Igusa.
\newblock Fibre systems of {J}acobian varieties.
\newblock {\em Amer. J. Math.}, 78:171--199, 1956.

\bibitem[Lan75]{langton}
Stacy~G. Langton.
\newblock Valuative criteria for families of vector bundles on algebraic varieties.
\newblock {\em Ann. of Math. (2)}, 101:88--110, 1975.

\bibitem[MM]{mayermumford}
A.~Mayer and D.~Mumford.
\newblock Further comments on boundary points.
\newblock Unpublished lecture notes distributed at the Amer. Math. Soc. Summer Institute, Woods Hole, 1964.

\bibitem[MRV17]{MRV}
Margarida Melo, Antonio Rapagnetta, and Filippo Viviani.
\newblock Fine compactified {J}acobians of reduced curves.
\newblock {\em Trans. Amer. Math. Soc.}, 369(8):5341--5402, 2017.

\bibitem[MSV21]{Migliorini_2021}
Luca Migliorini, Vivek Shende, and Filippo Viviani.
\newblock A support theorem for hilbert schemes of planar curves, {II}.
\newblock {\em Compositio Mathematica}, 157(4):835--882, apr 2021.

\bibitem[MSY24]{MSY}
Davesh Maulik, Junliang Shen, and Qizheng Yin.
\newblock Algebraic cycles and {H}itchin systems, 2024.
\newblock Preprint arXiv:2407.05177.

\bibitem[MT18]{mauliktoda}
Davesh Maulik and Yukinobu Toda.
\newblock Gopakumar-{V}afa invariants via vanishing cycles.
\newblock {\em Invent. Math.}, 213(3):1017--1097, 2018.

\bibitem[MV12]{meloviviani}
Margarida Melo and Filippo Viviani.
\newblock Fine compactified {J}acobians.
\newblock {\em Mathematische Nachrichten}, 285(8–9):997–1031, January 2012.

\bibitem[OS79]{Oda1979CompactificationsOT}
Tadao Oda and C.~S. Seshadri.
\newblock Compactifications of the generalized {J}acobian variety.
\newblock {\em Transactions of the American Mathematical Society}, 253:1--90, 1979.

\bibitem[PT23]{PTgenus1}
Nicola Pagani and Orsola Tommasi.
\newblock Geometry of genus one fine compactified universal {J}acobians.
\newblock {\em Int. Math. Res. Not. IMRN}, 2023(10):8495--8543, 2023.

\bibitem[PT24]{pagani2023stability}
Nicola Pagani and Orsola Tommasi.
\newblock Stability conditions for line bundles on nodal curves.
\newblock {\em Forum Math. Sigma}, 12:Paper No. e87, 31, 2024.

\bibitem[Ray70]{raynaud70}
M.~Raynaud.
\newblock Sp\'{e}cialisation du foncteur de {P}icard.
\newblock {\em Inst. Hautes \'{E}tudes Sci. Publ. Math.}, (38):27--76, 1970.

\bibitem[Ser06]{Sernesi}
Edoardo Sernesi.
\newblock {\em Deformations of algebraic schemes}, volume 334 of {\em Grundlehren der mathematischen Wissenschaften [Fundamental Principles of Mathematical Sciences]}.
\newblock Springer-Verlag, Berlin, 2006.

\bibitem[Sim94]{simpson}
Carlos~T. Simpson.
\newblock Moduli of representations of the fundamental group of a smooth projective variety {I}.
\newblock {\em Publications Math\'ematiques de l'IH\'ES}, 79:47--129, 1994.

\bibitem[{Sta}24]{stacks-project}
The {Stacks project authors}.
\newblock The stacks project.
\newblock \url{https://stacks.math.columbia.edu}, 2024.

\bibitem[Viv]{viviani2023new}
Filippo Viviani.
\newblock On the classification of fine compactified {J}acobians of nodal curves.
\newblock Preprint arXiv:2310.20317.

\end{thebibliography}
\end{document}